\documentclass[pagesize,paper=A4,11pt,bibliography=totoc,oneside]{scrartcl}

\usepackage{mathtools} 
\usepackage{amsmath,amssymb,amsthm,amsxtra}
\usepackage{mathrsfs}

\usepackage{lmodern}
\usepackage[T1]{fontenc}
\usepackage[tracking]{microtype}

\usepackage{graphicx}
\usepackage[dvipsnames]{xcolor}
\definecolor{emph}{HTML}{FFC107}
\definecolor{zero}{HTML}{7CA1CC}

\usepackage[obeyspaces,hyphens,spaces]{url}

\usepackage{booktabs,colortbl}
\newcommand{\cx}{\cellcolor{zero}0} 
\newcommand{\cp}{\cellcolor{emph}}

\usepackage[all,cmtip]{xy}

\usepackage{numprint}
\npdecimalsign{.}

\usepackage[colorlinks=true,linkcolor=violet,citecolor=blue]{hyperref}
\usepackage{cleveref}
\usepackage{orcidlink}
\usepackage{enumitem}

\numberwithin{equation}{section}

\usepackage[backend=bibtex,style=alphabetic,maxnames=9]{biblatex}
\renewbibmacro{in:}{}
\newbibmacro{string+doiurlisbn}[1]{%
  \iffieldundef{doi}{%
    \iffieldundef{url}{%
      \iffieldundef{isbn}{%
        \iffieldundef{issn}{%
          \iffieldundef{eprint}{%
            #1%
          }{%
            \href{https://arxiv.org/abs/\thefield{eprint}}{#1}%
          }
        }{%
          \href{https://books.google.com/books?vid=ISSN\thefield{issn}}{#1}%
      }
      }{%
        \href{https://books.google.com/books?vid=ISBN\thefield{isbn}}{#1}%
      }%
    }{%
      \href{\thefield{url}}{#1}%
    }%
  }{%
    \href{https://doi.org/\thefield{doi}}{#1}%
  }%
}
\DeclareFieldFormat{title}{\usebibmacro{string+doiurlisbn}{\mkbibemph{#1}}}
\DeclareFieldFormat[article,incollection,inproceedings]{title}{\usebibmacro{string+doiurlisbn}{\mkbibquote{#1}}}
\DeclareFieldFormat{postnote}{#1}
\DeclareFieldFormat{eprint}{\href{https://www.arxiv.org/abs/#1}{arXiv:#1}}

\setlist[enumerate,1]{label=\textnormal{(\roman*)}}

\newtheorem{thm}{Theorem}[section]
\newtheorem{proposition}[thm]{Proposition}
\newtheorem{lemma}[thm]{Lemma}
\newtheorem{corollary}[thm]{Corollary}
\newtheorem{definition}[thm]{Definition}
\newtheorem{example}[thm]{Example}
\newtheorem{remark}[thm]{Remark}
\newtheorem{question}[thm]{Question}

\newcommand{\defas}{\coloneqq}

\newcommand{\pff}{\phi}

\newcommand{\Gra}{\mathsf{Gra}}
\newcommand{\Li}[1]{\operatorname{Li}_{#1}} 
\newcommand{\asyO}{O} 
\newcommand{\dimeps}{\varepsilon}
\newcommand{\FIdim}{D} 
\newcommand{\pFq}[5]{{}_{#1}F_{#2}\left[\genfrac{}{}{0pt}{}{#3}{#4} ; #5\right]}
\newcommand{\iu}{\mathrm{i}} 
\newcommand{\ipi}{\iu\pi}
\newcommand{\Catalan}{\mathcal{C}} 
\newcommand{\EM}{\gamma_{\text{E}}} 
\newcommand{\Q}{\mathbb{Q}} 
\newcommand{\R}{\mathbb{R}} 
\newcommand{\C}{\mathbb{C}} 
\newcommand{\HC}{\mathbb{H}} 
\newcommand{\Z}{\mathbb{Z}} 
\newcommand{\Pro}{\mathbb{P}} 
\newcommand{\td}[1][]{\mathrm{d}^{#1}} 
\newcommand{\Transpose}{\intercal}
\newcommand{\Graph}[2][1.0]{\vcenter{\hbox{\includegraphics[scale=#1]{graphs/#2}}}}
\newcommand{\pInt}{\mathcal{I}} 
\newcommand{\fr}[1]{\mathbf{#1}} 
\newcommand{\ESP}{S} 

\DeclareMathOperator{\ExtA}{\bigwedge} 
\newcommand{\SymA}{\mathcal{S}} 
\newcommand{\UE}{\mathcal{U}} 
\DeclareMathOperator{\Prim}{Prim} 

\newcommand{\ctp}{\mathbin{\hat{\otimes}}}

\newcommand{\dR}{\mathrm{dR}} 
\newcommand{\rt}{\mathrm{rt}} 

\newcommand{\sslash}{\mathord{/\mkern-6mu/}}

\newcommand{\IntProd}[2]{#1 \lrcorner #2}

\newcommand{\hGC}{\mathcal{GC}}
\newcommand{\cGC}{\mathsf{GC}}

\newcommand{\MCdip}{\mathfrak{m}} 
\newcommand{\gcdeg}[1]{\left|#1\right|} 
\newcommand{\inner}[1]{\left\langle #1 \right\rangle} 

\newcommand{\Symanzik}{\Psi} 
\newcommand{\Laplacian}{\Lambda}
\newcommand{\FP}{\widetilde{\sigma}} 
\newcommand{\FS}{\sigma} 
\newcommand{\oFS}{\FS} 
\newcommand{\cFS}{\overline{\FS}} 

\newcommand{\sqX}{q}

\newcommand{\To}{\longrightarrow}
\newcommand{\Quad}{\mathcal{Q}}
\newcommand{\trop}{\mathrm{trop}}
\newcommand{\perf}{\mathrm{perf}}
\newcommand{\LA}{L\mathcal{A}}
\newcommand{\ob}{\mathcal{O}}
\newcommand{\BB}{\mathcal{B}}
\newcommand{\AdR}[1][\bullet]{\mathsf{A}^{#1}} 
\newcommand{\Evf}{\mathcal{E}} 
\newcommand{\bds}{b} 
\newcommand{\IdMat}{I}

\newcommand{\SL}{\mathsf{SL}}
\newcommand{\GL}{\mathsf{GL}}
\newcommand{\St}{\mathrm{St}} 
\newcommand{\OG}{\mathsf{O}}
\newcommand{\SO}{\mathsf{SO}}
\newcommand{\SPD}[1]{\mathcal{P}_{#1}}
\newcommand{\Perms}[1]{\mathfrak{S}_{#1}}
\newcommand{\id}{\mathrm{id}}
\renewcommand{\Im}{\mathop{\mathrm{Im}}}
\renewcommand{\Re}{\mathop{\mathrm{Re}}}
\DeclareMathOperator{\Hom}{Hom}
\DeclareMathOperator{\Aut}{Aut}
\DeclareMathOperator{\Pf}{Pf}
\DeclareMathOperator{\im}{im}
\DeclareMathOperator{\gr}{gr}
\DeclareMathOperator{\Gr}{Gr}
\DeclareMathOperator{\sgn}{sgn}
\DeclareMathOperator{\tr}{tr}
\newcommand{\cfs}[1][\bullet]{\Omega_{\mathrm{can}}^{#1}} 

\DeclarePairedDelimiter{\abs}{\lvert}{\rvert}
\DeclarePairedDelimiter{\norm}{\lVert}{\rVert}

\newcommand{\Maple}{\href{http://www.maplesoft.com/products/Maple/}{\textsf{\textup{Maple}}}}

\newcommand{\nauty}{\href{http://pallini.di.uniroma1.it/}{\texttt{\textup{nauty}}}}
\newcommand{\HyperInt}{\href{https://bitbucket.org/PanzerErik/hyperint}{\texttt{HyperInt}}}
\newcommand{\HyperlogProcedures}{\href{https://www.math.fau.de/person/oliver-schnetz/}{\texttt{HyperlogProcedures}}}
\newcommand{\HypExpTwo}{\texttt{HypExp~2}}
\newcommand{\pySecDec}{\textup{pySecDec}}
\newcommand{\LiteRed}{\href{http://www.inp.nsk.su/~lee/programs/LiteRed/}{\texttt{LiteRed}}}

\title{Unstable cohomology of $\GL_{2n}(\Z)$ and the  odd commutative graph complex}

\newcommand{\email}[1]{\href{mailto:#1}{#1}}
\author{
 \thanks{All Souls College, Oxford, OX1 4AL, UK, \email{francis.brown@all-souls.ox.ac.uk}}
    Francis Brown
    \orcidlink{0000-0002-9295-2572}
    \and
    \thanks{Mathematical Institute, University of Oxford, OX2 6GG, UK, \email{simone.hu@maths.ox.ac.uk}}
    Simone Hu
    \and
    \thanks{Mathematical Institute, University of Oxford, OX2 6GG, UK, \email{erik.panzer@maths.ox.ac.uk}}
    Erik Panzer
    \orcidlink{0000-0002-9897-5812}
}

\begin{document}

\maketitle

\begin{abstract}
    We study a closed differential form on the symmetric space of positive definite matrices, which is defined using the Pfaffian and is  $\GL_{2n}(\Z)$ invariant up to a sign. It gives rise to an infinite family of unstable classes in the compactly-supported cohomology of the locally symmetric space for $\GL_{2n}(\Z)$ with coefficients in the orientation bundle. Furthermore, by applying the Pfaffian forms to the dual Laplacian of graphs, and integrating them over the space of edge lengths, we construct an infinite family of cocycles for the odd commutative graph complex. By explicit computation, we show that the first such cocycle gives a non-trivial class in $H^{-6}(\cGC_3)$.
\end{abstract}

\setcounter{tocdepth}{1}
\tableofcontents
\setcounter{tocdepth}{3}

\section{Introduction}

The differential forms studied  in this paper are constructed out of what we call  the invariant Pfaffian form. It is defined on the space $\SPD{2n}$ of positive definite matrices  $X$ of even rank $2n$ by the formula
\begin{equation*}
\pff^{2n}_X = \frac{\Pf (\td X\cdot X^{-1} \cdot \td X)}{\sqrt{\det(X)}}  \ ,
\end{equation*}
where $\Pf$ denotes the Pfaffian of a skew-symmetric matrix.
It is a closed  smooth differential form  of degree $2n$ with many good properties. In particular, it transforms by
\begin{equation} \label{intro: transformationgamma}
  \pff^{2n}_{P^\Transpose X P } = \sgn \left(\det(P)\right)  \pff^{2n}_X
\end{equation}
for any $P \in \GL_{2n}(\R)$. It  therefore defines a differential form on the locally symmetric space $\SPD{2n}/\GL_{2n}(\Z)$ with coefficients in the orientation bundle $\ob$, which is the rank one real vector bundle  induced by the determinant representation of $\GL_{2n}(\Z)$.  We call such a form an orientation form.
 Importantly,  $\pff^{2n}$  is  also invariant under the action of $\R^{\times}_{>0}$ on $\SPD{2n}$ by scalar multiplication, and therefore descends to the quotient $\R^{\times}_{>0} \setminus  \SPD{2n}$. 

By taking exterior products with the invariant differential forms
\begin{equation} \label{intro: omegadef}
\beta^{4k+1}_X = \tr \left((X^{-1} \td X)^{4k+1}\right)
\end{equation}
for $k\geq 1$, we  obtain  many further orientation forms on $\SPD{2n}/\GL_{2n}(\Z)$.  For example, we show that the   volume form (a non-vanishing orientation form of top degree)
on the space  of projective classes of positive definite symmetric matrices
is proportional to the product $  \beta^5 \wedge \ldots \wedge \beta^{4n-3} \wedge \pff^{2n}$. The invariant  Pfaffian therefore gives a factorisation of the volume form.

In this paper, we use these orientation forms for two different purposes. On the one hand, we show that those of \emph{compact type} give rise to infinitely many non-zero  classes in the  compactly-supported cohomology of $\SPD{2n}/\GL_{2n}(\Z)$ with coefficients in $\ob$. We show, furthermore, that they are primitive with respect to a Hopf algebra structure on $\bigoplus_{n} H^{\bullet}_c(\SPD{2n}/\GL_{2n}(\Z);\ob)$ recently introduced in \cite{AMP}, and deduce that the latter contains the  symmetric algebra generated by the orientation forms of compact type.

On the other hand, we use the Pfaffian orientation forms which are \emph{not of compact type} to construct cocycles on the odd commutative graph complex by interpreting them as orientation forms on a moduli space of metric graphs. By integrating over cells indexed by graphs, we assign transcendental  invariants to elements in the odd graph complex and compute them explicitly for a closed linear combination of graphs with $6$ loops and $12$ edges. This proves that the first of our cocycles defines a non-trivial class in $H^{-6}(\cGC_3)$.

We emphasize that the orientation forms we use to construct cohomology classes for $\GL_{2n}(\Z)$ (which are of `compact type'), and for the odd graph complex (which are of `non-compact type'), are not the same, with exactly one exception (namely $\pff^2$).

\subsection{Unstable cohomology of \texorpdfstring{$\GL_{2n}(\Z)$}{GL\_2n(Z)}}
The Pfaffian forms can be used to write down explicit representatives  for the  Poincar\'e duals of the    cohomology classes for $\SPD{2n}/\GL_{2n}(\Z)$  generated by the invariant forms $\beta$ (of non-compact type) in the unstable range. More precisely, let
\begin{equation*}
\Omega_{nc}^\bullet(2n-1) =  \ExtA^\bullet   \bigg(\bigoplus_{1\leq k< n-1}  \Q\, \beta^{4k+1}\bigg)
\end{equation*}
denote the graded exterior algebra generated by formal symbols $\beta^{4k+1}$  in degree $4k+1$ corresponding to the  invariant differential forms \eqref{intro: omegadef}. The subscript $nc$ stands for `non-compact type'. It was proven in \cite{BrBord} that
\begin{equation}  \label{intro: stableclasses}
\Omega^\bullet_{nc}(2n-1) \To H^{\bullet}(\SPD{m}/\GL_m(\Z);\R) \quad \hbox{ is injective  for all } m\geq 2n-1\ 
\end{equation}
having been announced in  \cite{LeeUnstable} (see also \cite{Franke}). 
Together with results of Matsushima and Garland this implies Borel's theorem (see \cite{Borel} and references therein)  which states that the stable cohomology of $\GL(\Z)$ is generated by the graded exterior algebra with one element in degree $4k+1$, for each $k\geq 1$.

Now define a bigraded vector space   $\Omega^{\bullet}_{c}[\pff](2n)$   of orientation forms, which we call  of `compact type', which is
generated by the symbols $\pff^2$ and
\begin{equation} \label{generatorcompacttype}  \beta^{4i_1+1}\wedge \beta^{4i_2+1} \wedge \ldots  \wedge \beta^{4i_k+1}\wedge \pff^{2n}  \end{equation}
for all $1\leq i_1<i_2<\ldots<i_{k}$  where   $i_k=n-1.$ The grading is by genus and degree minus genus, where the generator \eqref{generatorcompacttype} has degree $(4i_1+1)+\ldots + (4i_k+1)+(2n+1)$ and genus $2n$, and $\pff^2$ has degree $3$ and genus $2$.

As a first application of Pfaffian forms, we show that they can be interpreted as compactly supported cohomology classes: we define classes denoted
\begin{align*}
    [(\pff^2)_c] &\in H_c^{3}(\SPD{2}/\GL_2(\Z);\ob) & &\text{and} \\
    [( \beta^{4n-3} \wedge \pff^{2n} )_c] &\in H^{6n-2}_c(\SPD{2n} /\GL_{2n}(\Z) ; \ob) & & \text{for all $n>1$,}
\end{align*}
and we write $(\pff^2)_c$, $(\beta^{4n-3} \wedge \pff^{2n} )_c$ for any compactly supported forms (representatives) in such a class. The key point of our construction is that while the orientation forms themselves do \emph{not} have compact support, they can be interpreted as canonical \emph{relative} forms on a compactification of $\SPD{2n}/\GL_{2n}(\Z)$. In particular, integrals like
\begin{equation*}
    \int_\sigma \beta^{4n-3} \wedge \pff^{2n}
\end{equation*}
converge even for locally finite, $\ob$-oriented chains $\sigma$ on $\SPD{2n} /\GL_{2n}(\Z)$.

All the cohomology classes thus obtained are non-trivial and independent:
\begin{thm}  \label{intro: thmcohomology}
For all $n>1$, we obtain a canonical injective map of graded vector spaces
\begin{eqnarray} \label{intro:gammaformsinject}
\Omega^{\bullet}_{c}[\pff](2n) & \To &  H^{\bullet}_c(\SPD{2n} /\GL_{2n}(\Z) ; \ob)   \\
 \beta^{4i_1+1}\wedge \beta^{4i_2+1} \wedge \ldots  \wedge \beta^{4n-3}\wedge \pff^{2n}    & \mapsto & [  \beta^{4i_1+1}\wedge \beta^{4i_2+1} \wedge \ldots  \wedge (\beta^{4n-3}\wedge \pff^{2n})_c  ] \nonumber
\end{eqnarray}
The images of these classes pair non-trivially, under Poincar\'e duality,  with the image of \eqref{intro: stableclasses} when $m=2n$. For $n=1$, the class $[(\pff^2)_c]$  generates $ H_c^{3}(\SPD{2}/\GL_2(\Z);\ob)\cong \R$.
\end{thm}

Note that  the compactly supported   twisted  cohomology of $\SPD{2n}/\GL_{2n}(\Z)$ is Poincar\'e dual to the ordinary cohomology of $\GL_{2n}(\Z)$ via
\begin{equation*}
H^k(\GL_{2n}(\Z);\R) \cong H^k(\SPD{2n}/\GL_{2n}(\Z);\R) \cong \left( H_c^{d_{2n}-k}( \SPD{2n}/\GL_{2n}(\Z);\ob)\right)^{\vee}
\end{equation*}
where $d_m= m(m+1)/2$
and furthermore   appears as a summand in the cohomology  of the special linear group  by an application of Shapiro's lemma which implies that: 
\begin{equation*}
H^k(\SL_{2n}(\Z);\Q) \cong  H^k(\GL_{2n}(\Z);\Q) \oplus   H^k(\GL_{2n}(\Z);\det)
\end{equation*}
where $\det$ denotes the determinant representation.

The proof of  theorem  \ref{intro: thmcohomology} involves studying   the asymptotic behaviour of the forms $ \beta^{4n-3} \wedge \pff^{2n} $  and proving that they extend to  an algebro-geometric incarnation of the  Borel-Serre compactification \cite{BrBord}, and vanish along the boundary.

\subsubsection{Four spaces of differential forms}\label{sec:formspaces}
The relationship between  canonical forms, the Pfaffian invariant form, and different cohomology groups associated to  the locally symmetric space of $\GL_g(\Z)$ may be summarised as follows.
Let $g>1$ be odd, and consider the  four graded vector spaces:
\begin{enumerate}
\item  $\Omega^{\bullet}_{nc}(g) = \ExtA^{\bullet} \left( \Q \beta^5 \oplus \ldots \oplus \Q \beta^{2g-5}\right)$
\item $\Omega^{\bullet}_c(g) = \Omega^{\bullet}_{nc}(g) \wedge \beta^{2g-1}$
\item $\Omega^{\bullet}_{nc}[\pff](g+1) = \Omega^{\bullet}_{nc}(g) \wedge \pff^{g+1} $
\item $\Omega^{\bullet}_{c}[\pff](g+1) = \Omega^{\bullet}_{c}(g) \wedge \pff^{g+1} $
\end{enumerate}
The first space, $\Omega^{\bullet}_{nc}(g)$ has the structure of a graded algebra, but the others are to be viewed merely  as graded vector spaces.  The subscript $c$ stands for `compact type', and $nc$ for `non-compact type'.
When $n=g$ is odd, the locally  symmetric space $\SPD{g}/\GL_g(\Z)$ is orientable and we have embeddings of graded vector spaces \cite{BrBord}:
\[ \Omega_{nc}^{\bullet}(g)   \hookrightarrow    H^{\bullet}(\SPD{g}/\GL_g(\Z);\R) \qquad \ , \qquad
 \Omega_{c}^{\bullet}(g)     \hookrightarrow    H_c^{\bullet+1}(\SPD{g}/\GL_g(\Z);\R) \ .
\]
The two spaces $H^{\bullet}(\SPD{g}/\GL_g(\Z);\R)$ and $H_c^{\bullet}(\SPD{g}/\GL_g(\Z);\R)$ are related by Poincar\'e duality,  which is manifested by the Hodge star operator which relates $\Omega^{\bullet}_{nc}$ to $\Omega^{\bullet}_{c}$.

When $2n=g+1$ is even, the orbifold $\SPD{2n}/\GL_{2n}(\Z)$ is not orientable and all spaces (i)--(iv)  are implicated.   We have   a pair of injections of graded vector spaces
\[   \Omega^{\bullet}_{nc}(2n-1) \hookrightarrow   H^{\bullet}(\SPD{2n}/\GL_{2n}(\Z);\R)  \qquad  , \qquad
  \Omega^{\bullet}_{c}[\pff](2n)    \hookrightarrow    H_c^{\bullet}(\SPD{2n}/\GL_{2n}(\Z);\ob)
\]
which are related by Poincar\'e duality. The first injection is a consequence of \eqref{intro: stableclasses}, the second is the statement of \cref{intro: thmcohomology}. In addition, we have maps
\[   \Omega^{\bullet}_{c}(2n-1) \hookrightarrow   H_c^{\bullet+2}(\SPD{2n}/\GL_{2n}(\Z);\R)  \qquad  , \qquad
  \Omega^{\bullet}_{nc}[\pff](2n)   \mapsto 0 \in   H^{\bullet}(\SPD{2n}/\GL_{2n}(\Z);\ob)
\]
where the first map is the  inflation morphism of \cite{BBCMMW},  and  was shown to be injective  in \cite{BrBord}. The fact that $\Omega_{nc}^{\bullet}[\pff](2n)$  maps to zero in  $   H^{\bullet}(\SPD{2n}/\GL_{2n}(\Z);\ob)$ follows from the fact \cite{Sullivan:EulerSLZ} that the cohomology class of $[\pff^{2n}]$ in $H^{2n}(\SPD{2n}/\GL_{2n}(\Z);\ob)$ vanishes.

\subsubsection{Symmetric algebra on Pfaffian invariant forms of compact type}

It was recently shown in \cite{AMP} that  the bigraded $\Q$-vector space
\[ H^{\det} =  \bigoplus_{k,n\geq 0}       H_k (\GL_n(\Z); \St_n(\Q)\otimes \det) \]
where $\St_n(\Q)$ denotes the Steinberg module and $\det$  the determinant representation, has the structure of a bigraded commutative Hopf algebra.  Its graded linear dual  $H_{\det}=(H^{\det})^{\vee}$ is  a cocommutative bigraded Hopf algebra over $\Q$.  Since the dualising module for $\GL_{2n}(\Z)$ is $\St_{2n}(\Q)\otimes \det$ (\cite[\S11.4]{BorelSerre:Corners}, see, e.g. \cite[\S7.2]{ElbazVincentGanglSoule}, \cite{PutmanStudenmund:DualGLnO}), we obtain  an isomorphism
\[ H_{\det} \otimes_{\Q} \R = \bigoplus_{n\geq 0}     H^{\bullet}_c\left( \SPD{2n}/\GL_{2n}(\Z) ;\ob\right)  \]
since the local system $\ob$ is given by the representation $\det\otimes \R$ in the even rank case.

\begin{table}
\centering
\setlength{\tabcolsep}{3pt}
\begin{tabular}{r@{\hskip 10pt}cccc}
\toprule
$k-2n$ & $\GL_8(\Z)$ & $\GL_6(\Z)$ & $\GL_4(\Z)$ &  $\GL_2(\Z)$   \\
\midrule
1 & & & &  $[\pff^2] $  \\
6 & &  &$[\beta^5 \wedge \pff^4]$  &   \\
7 & &$\cp [\pff^2]\cdot [\beta^5 \wedge \pff^4]$  & &   \\
10 & &$[\beta^9 \wedge \pff^6]$  & &   \\
11 &   $\cp [\pff^2] \cdot [\beta^9\wedge \pff^6]$   & & &  \\
12 &$\cp [\beta^5 \wedge \pff^4]^2$  &  & &   \\
14 &$[\beta^{13} \wedge \pff^8] $  & & &   \\
15 & &$[\beta^5\wedge \beta^9 \wedge \pff^6]$  & &   \\
16 & $\cp [\pff^2] \cdot [\beta^5 \wedge \beta^9\wedge \pff^6]$  & & &   \\
19 & $[\beta^5\wedge \beta^{13} \wedge \pff^8]$ & & &   \\
23 & $[\beta^9\wedge \beta^{13} \wedge \pff^8]$ & & &   \\
28 & $[\beta^5\wedge \beta^9\wedge \beta^{13} \wedge \pff^8]$  & & &   \\
\bottomrule
\end{tabular}%
\caption{Cohomology classes in $H_c^{k}(\SPD{2n}/\GL_{2n}(\Z); \ob)$ obtained from \cref{thm:introSymembeds}, for small $n$.  By Poincar\'e duality, which  reverses the degrees in each column, one deduces classes in the cohomology of $\GL_{2n}(\Z)$. The highlighted classes denote non-trivial products with respect to the Hopf algebra structure of \cite{AMP}.
}%
\label{table:cohom}%
\end{table}

By showing that the Pfaffian forms of compact type are primitive for the commutative coproduct, we deduce from the Milnor-Moore and Poincar\'e-Birkhoff-Witt theorems   that there  exists  a
large amount of new unstable cohomology for $\GL_n(\Z)$ for even $n$.
\begin{thm}\label{thm:introSymembeds}
    There is an embedding of bigraded vector spaces
    \begin{equation} \label{intro: SymAtoHSPD}
        \SymA^{\bullet} \left( \Omega_{c}[\pff] \right) \otimes \R  \To   \bigoplus_{n\geq 0}     H^{\bullet}_c\left( \SPD{2n}/\GL_{2n}(\Z) ;\ob\right)
    \end{equation}
    from the symmetric algebra on the bigraded vector space  $\Omega_{c}[\pff]= \bigoplus_{n} \Omega_c[\pff](2n)$  into the compactly supported cohomology of the locally symmetric space for $\GL_{2n}(\Z)$ with twisted coefficients.  The former is isomorphic (as a bigraded vector space) to the tensor product of the free commutative polynomial ring on its generators in even degrees, with the graded exterior algebra on those in odd degrees.

    Dually, there is a surjective map of bigraded vector spaces
    \begin{equation} \label{intro: dualSymAtoHSPD}
    \bigoplus_{n\geq 1}     H^{d_{2n}-\bullet}\left( \GL_{2n}(\Z) ;\R\right)  \To     \SymA(  \Omega_{c}[\pff])^{\vee} \otimes \R   \qquad \text{where}\  d_{2n} =\binom{2n+1}{2} \ .
    \end{equation}
\end{thm}

See \cref{table:cohom} for an illustration.
We note that \cref{thm:introSymembeds} explains all homology classes of $\GL_{2n}(\Z)$ in the range $2n\leq 6$ where it has been computed in all degrees  \cite{ElbazVincentGanglSoule}, and consistent with \cite{GL8}.  Note that there is some freedom in the choice of  map \eqref{intro: SymAtoHSPD}: it can be any choice of splitting of the length grading which arises in the Poincar\'e-Birkhoff-Witt theorem. If one chooses the map \eqref{intro: SymAtoHSPD} to be given by (graded) symmetrised products of length $n$ with respect to the Hopf algebra structure on $H_{\det} \otimes_{\Q} \R$ (with the appropriate prefactor $1/n!$) then it may be promoted to  a map of bigraded coalgebras. The corresponding dual map \eqref{intro: dualSymAtoHSPD} is then a map of bigraded algebras. 

In particular, it follows that the  graded Lie bracket (antisymmetrised product) of $[\phi^2]$ and $[\beta^5 \wedge \phi^4]$ necessarily vanishes in $H^7_c(\SPD{6}/\GL_6(\Z);\ob)$ since it has rank 1, i.e., they commute. 
It is an interesting  open question, therefore, whether the Lie algebra generated by the classes in $\Omega_c[\phi]$ is abelian or not: we do not know of any non-trivial Lie brackets at present.

\subsection{Cochains in the odd commutative graph complex}\label{sec:intro-gc}
The commutative graph complexes $\hGC_N$ introduced by Kontsevich \cite{kontfeynman} have numerous applications in deformation theory \cite{Willwacher:GCgrt}, geometry \cite{CGP18}, and topology.
Thanks to isomorphisms $\hGC_N\cong \hGC_{N+2}$, there are essentially two cases: $N$ even and $N$ odd.

In this paper we consider the odd graph complex $\hGC_3$.\footnote{
Geometrically, $H_\bullet(\hGC_3)$ can be realized as equivariant  relative homology of the simplicial closure of outer space twisted by the determinant representation \cite[Proposition~27]{onatheorem}, or as twisted homology of the moduli space of metric graphs \cite[Example~5.9]{Berghoff:GCgeom} (see also \cite[bottom of p.~6]{kontsympl}).}
A connected combinatorial graph $G$ with $n$ vertices and $m$ edges has \emph{loop number} $\ell=m-n+1$ and \emph{degree} $k=m-3\ell$.
As a $\Q$-vector space, $\hGC_3$ is spanned by isomorphism classes of such graphs, equipped with an  orientation. As a chain complex, its boundary $\partial G=\sum_e G/e$  contracts edges, which is homogeneous of degrees $(0,-1)$ for the bigrading by $(\ell,k)$.
The dual cochain complex $\cGC_3\cong\Hom(\hGC_3,\Q)$ has a coboundary $\delta$ and a compatible (graded) Lie bracket.
The cohomology $H^k(\cGC_3)=\prod_{\ell} \gr_{\ell} H^{k}(\cGC_3)$ is known by \cite[Lemma~1.4]{bounds} to vanish outside the region
\begin{equation}\label{eq:H(GC_3)-bound}
    -\ell \leq k \leq -3,
\end{equation}
with the sole exception of the class $\gr_2 H^{-3}(\cGC_3)=\Q[D_3]$ spanned by the theta graph, which we denote $D_3$. It provides the first class in the cohomology group $H^{-3}(\cGC_3)$ at the upper bound $k=-3$. The  group  $H^{-3}(\cGC_3)$ has an additional algebraic structure and is also called the \emph{algebra of 3-graphs} \cite{3algebra} and has been studied extensively due to its role in the enumeration of Vassiliev knot invariants. In lower degrees, however, only very few classes are known, as summarized in \Cref{table:GC3}, even though there must exist many  cohomology classes for large $\ell$  \cite{Borinsky:EulerGC}.

\begin{table}
\centering
\setlength{\tabcolsep}{3pt}
\begin{tabular}{r@{\hskip 10pt}ccccccccc>{\columncolor{emph}}c}
\toprule
$\ell$ & $H^{-12}$ & $H^{-11}$ & $H^{-10}$ & $H^{-9}$ & $H^{-8}$ & $H^{-7}$ & $H^{-6}$ & $H^{-5}$ & $H^{-4}$ & \cellcolor{white} $H^{-3}$ \\
\midrule
 2 &     &     &     &     &     &     &     &     & \cx & 1 \\
 3 &     &     &     &     &     &     & \cx & \cx & \cx & 1 \\
 4 &     &     &     &     & \cx & \cx & \cx & \cx &  0  & 1 \\
 5 &     &     & \cx & \cx & \cx & \cx & \cx &  0  &  0  & 2 \\
 6 & \cx & \cx & \cx & \cx & \cx & \cx &  1  &  0  &  0  & 2 \\
 7 & \cx & \cx & \cx & \cx & \cx &  0  &  1  &  0  &  0  & 3 \\
 8 & \cx & \cx & \cx & \cx &  0  &  0  &  2  &  0  &  0  & 4 \\
 9 & \cx & \cx & \cx &  0  &  0  &  0  &  3  &  0  &  0  & 5 \\
10 & \cx & \cx &  0  &  0  &  0  &  0  &  5  &  0  &  0  & 6 \\
11 & \cx &  ?  &  ?  &  ?  &  ?  &  ?  &  ?  &  ?  &  ?  & 8 \\
12 &  ?  &  ?  &  ?  &  ?  &  ?  &  ?  &  ?  &  ?  &  ?  & 9 \\
\bottomrule
\end{tabular}%
\caption{
  Dimensions of the cohomology $\gr_{\ell}H^{k}(\cGC_3)$, bigraded by degree $k$ (columns) and loop number $\ell$ (rows), known from computer calculations \cite{BNM,diff,GHcode}.
  Empty cells indicate trivial zeroes (no graphs), whereas \colorbox{zero}{~$0$~} indicates vanishing in the range $k<-\ell$ known from \eqref{eq:H(GC_3)-bound}.
  The column $k=-3$ highlights the algebra of 3-graphs; e.g.\ $\gr_2 H^{-3}(\cGC_3)\cong\Q$ is spanned by the theta graph.
}%
\label{table:GC3}%
\end{table}

Generalizing a key idea in \cite{invariant}, we construct linear functionals $I(\omega)\colon \hGC_3\longrightarrow \R$, in other words cochains (elements of $\cGC_3\ctp \R$), by integrating the pullback of differential forms on $\SPD{\ell}$ to a moduli space of metric graphs via the tropical Torelli map \cite{BrannettiMeloViviani:TropicalTorelli}.

Concretely, to a graph $G$ we associate the space $\R_+^m$ parametrizing positive edge lengths. For any basis $\mathcal{C}$ of the cycle space $H_1(G)$, the dual graph Laplacian $\Laplacian_\mathcal{C}\colon \R_+^m\longrightarrow \SPD{\ell}$ is a positive definite symmetric $\ell \times \ell$ matrix whose entries are linear in the edge lengths.
Pulling back under $\Laplacian_{\mathcal{C}}$, we obtain smooth differential forms
\[
\omega_G  \wedge \pff_G = \omega_{\Laplacian_\mathcal{C}}  \wedge \pff^{\ell}_{\Laplacian_{\mathcal{C}}}
\]
on $\R_+^m$, for any polynomial $\omega$ in the forms \eqref{intro: omegadef}. In fact, these pullbacks are projectively invariant and descend to smooth forms on the open simplex $\oFS_G=\R_+^m/\R_+ \subset\Pro(\R^m)$.

For example, the theta graph has $n=2$ vertices and $m=3$ edges, hence $\ell=2$ loops and degree $k=-3$. The dual Laplacian for the cycle basis $\mathcal{C}=(C_1,C_2)$ as indicated in
\begin{equation*}
    D_3=\Graph[0.8]{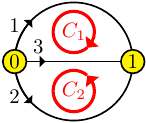}
    \qquad\text{is}\qquad
    \Laplacian_{\mathcal{C}} =
  \begin{pmatrix}
    x_1 + x_3 & x_3 \\
    x_3 & x_2 + x_3
  \end{pmatrix}
\end{equation*}
in terms of the edge length variables $x_1,x_2,x_3$. Computing the invariant Pfaffian form $\pff^2$ for this family of positive definite matrices provides the projective 2-form
\begin{equation*}
    \pff_{D_3} = -\frac{x_1 \td x_2 \wedge \td x_3 - x_2 \td x_1 \wedge \td x_3 + x_3 \td x_1 \wedge \td x_2}{(x_1x_2 + x_2x_3 + x_1x_3)^{3/2}}
    .
\end{equation*}

\begin{thm}\label{thm:intro-convergence}
	For any homogeneous polynomial $\omega$ of degree $m-\ell-1$ in the canonical forms \eqref{intro: omegadef}, and any oriented graph $G\in\hGC_3$ with loop number $\ell$, the integral
  \begin{equation} \label{intro: IGdefn}
  I_G(\omega) = \frac{1}{(-2\pi)^{\ell/2}}\int_{\oFS_G}  \omega_G   \wedge \pff_G
  \end{equation}
  is absolutely convergent and defines a linear function $I(\omega)\colon \hGC_3\longrightarrow \R, G\mapsto I_G(\omega)$.
\end{thm}
The convergence follows from the asymptotic behaviour of the forms at the boundary of a suitable compactification $\FP_G$ of the simplex $\oFS_G$. To define $I_G(\omega)$, we also have to specify the orientation of $\oFS_G$ and fix the sign of $\pff_G$, since the latter depends implicitly on the cycle basis $\mathcal{C}$. These ambiguities are resolved by the notion of orientation of a graph $G$ in $\hGC_3$, which by \cite{onatheorem} amounts to a choice of generator of $\det \Z^m\otimes\det H_1(G)$ where, for a free $\Z$-module $M$ of finite rank $n$, $\det(M)$ denotes the rank one $\Z$-module $\ExtA^n M$. The sign in \eqref{intro: transformationgamma} is therefore crucial for our construction.

We thus obtain canonical integrals, and hence numbers, associated to oriented graphs in $\hGC_3$. In the simplest case, $\omega=1$ so we integrate only the bare Pfaffian form. For example, $\pff_{D_3}$ defined above integrates to $I_{D_3}(1)=1$ for the theta graph. Our normalization of $I_G(\omega)$ by powers of $-2\pi$ ensures that $I_G(1)\in\{-1,0,1\}$ for all graphs, since we show that   $I_G(1)=\pm1$ if and only if $G\cong \pm D_{2k+1}$ is isomorphic to a dipole graph $D_{2k+1}$ with even loop number $2k$.
These are the graphs on $2$ vertices with $2k+1$ edges:
\begin{equation*}
    D_3=\Graph{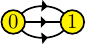},\qquad
    D_5=\Graph{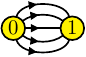},\qquad
    D_7=\Graph{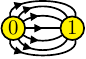},\qquad\ldots
\end{equation*}
For non-trivial $\omega$, the integrals \eqref{intro: IGdefn} are much more complicated. An example for an integral of the 11-form $\beta^5\wedge\pff^6$ is (this graph is labelled $G_{234}$ in \cref{sec:integration6})
\begin{equation*}
   G
   =\Graph[0.8]{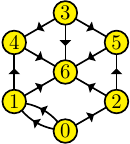}
   \quad\mapsto\quad
   \begin{aligned}
    I_G
    (\beta^5)
    &=
    -80
    +\tfrac{140}{9}\pi^2
    -\tfrac{4}{3}\pi^4
    +210\zeta(3)-20\pi^2\ln 2
    \\ & \quad
    +\tfrac{20}{3}\left[
            (\ln 2)^4+24\Li{4}\big(\tfrac{1}{2}\big)
    \right]
    -\tfrac{40}{9}\left[
            \pi^2\Catalan + 24 \Im \Li{4}(\iu)
    \right]
    \end{aligned}
\end{equation*}
where $\Catalan=\Im\Li{2}(\iu)$ denotes Catalan's constant and $\Li{s}(z)=\sum_{k=1}^{\infty} z^k/k^s$ is the polylogarithm function.

\subsubsection{Stokes' relations and cocycles}
In the context of graphs, the relevant algebra of `canonical' forms  is  $\cfs$, 
 the exterior algebra generated by formal symbols $\beta^{4k+1}$, one for each positive integer $k$. The above integrals of orientation forms constitute a linear map
\begin{equation*}
    I\colon \cfs\longrightarrow \Hom(\hGC_3,\R).
\end{equation*}
Since $\Hom(\hGC_3,\R)\cong\cGC_3\ctp \R$, we can apply the coboundary $\delta$ and the Lie bracket of $\cGC_3$ to the cochains $I(\omega)$. Stokes' theorem for the manifold with boundary $\FP_G$ relates the coboundary of such cochains to Lie brackets. To state this result, we declare
\begin{equation*}
    \Delta\beta^{4k+1}=1\otimes\beta^{4k+1}+\beta^{4k+1}\otimes 1
\end{equation*}
to endow $\cfs$ with the structure of a graded cocommutative Hopf algebra. This coproduct encodes the restrictions of canonical forms to the boundary faces at infinity of $\FP_G$.
We use Sweedler's notation $\sum_{(\omega)}\omega'\otimes\omega''$ to denote the reduced coproduct, such that
\begin{equation*}
    \Delta \omega = 1\otimes \omega + \omega\otimes 1 + \sum_{(\omega)} \omega'\otimes \omega''
    .
\end{equation*}
\begin{thm}\label{thm:intro-Stokes}
Let $\MCdip\in\cGC_3$ denote the following sum of all even loop number dipoles:
\begin{equation*}
    \MCdip = \sum_{k=1}^{\infty} \frac{D_{2k+1}}{2\cdot(2k+1)!}.
\end{equation*}
Write $\abs{\omega}$ for the degree of $\omega$. Then, for any canonical form $\omega\in\cfs$, we have the relation
\begin{equation}\label{eq:intro-Stokes}
	0 = \delta I(\omega) + [I(\omega),\MCdip] + \frac{1}{2}\sum_{(\omega)} (-1)^{|\omega'|} [I(\omega''),I(\omega')].
\end{equation}
\end{thm}
The dipole sum $\MCdip$ arises as the integral $I(1)$ of the bare Pfaffian form in the trivial piece $1\otimes \omega+\omega\otimes 1$ of the coproduct $\Delta\omega$.
Specializing to $\omega=1$, Stokes' relation is equivalent to the Maurer-Cartan equation
\begin{equation*}
    \delta \MCdip + \frac{1}{2}[\MCdip,\MCdip] = 0
\end{equation*}
which was discovered combinatorially in \cite{diff}. We thus obtain a geometric explanation for the origin of the corresponding twisted differential $\delta+[\cdot,\MCdip]$ on $\cGC_3$, tracing it back to the structure of the boundary of the compactifications $\FP_G$.

Specializing to primitive canonical forms $\omega=\beta^{4i+1}$, the Stokes' relation \eqref{eq:intro-Stokes} shows that the corresponding cocycle is closed with respect to the twisted differential:
\begin{equation*}
    \delta I(\beta^{4i+1})=-[I(\beta^{4i+1}),\MCdip].
\end{equation*}
To obtain \emph{bona fide} cocycles for the original differential $\delta$ in $\cGC_3$, let $\gr_\ell I(\omega)$ denote the restriction of the cochain $I(\omega)$ from $\hGC_3=\bigoplus_\ell \gr_\ell \hGC_3$ to the part in loop number $\ell$.
\begin{lemma}\label{lem:intro-gr-canint=0}
    Write $\lceil x\rceil$ for the smallest integer above or equal to $x$. For a canonical form $\omega\in\cfs[d]$ of degree $d$, the integrals $\gr_\ell I(\omega)=0$ vanish in all loop orders $\ell<2+2\lceil d/4 \rceil$.
\end{lemma}
For example, this result shows that $I_G(1)=0$ for graphs with one loop, and $I_G(\beta^5)=0$ for graphs with $\ell<6$ loops. Since the Lie bracket respects the grading by $\ell$, Stokes' relations simplify drastically in the minimal loop order where $I(\omega)$ might be non-zero.
\begin{corollary}\label{cor:intro-cocycles}
    For every canonical form $\omega\in\cfs[d]$ of degree $d$, the cochain $I(\omega)$ restricts in loop number $2+2\lceil d/4\rceil$ to a cocycle $\gr_{2+2\lceil d/4\rceil} I(\omega)\in \cGC_3\otimes \R$ with respect to $\delta$.
\end{corollary}

We have thus explicitly defined infinitely many classes in the cohomology of $\cGC_3$, one for each monomial in the forms $\beta^{4i+1}$. They lie in degrees $-6\leq k\leq -3$ related to the degree $d$ of the form $\omega$ by $k\equiv d+1 \mod 4$. We do not know in general which of these classes are zero, but we give examples of zero- and non-zero cases.

The simplest such cocycle is $\gr_2 I(1)$, that is, the integral of the Pfaffian form $\pff^2$, in degree $k=-3$. By $I_{D_3}(1)=1$, the corresponding cohomology class
\begin{equation*}
    [\gr_2 I(1)] \in H^{-3}(\cGC_3)
\end{equation*}
is non-trivial and dual to the class $[D_3]\in H_{-3}(\hGC_3)$ of the theta graph in homology.

For primitive $\omega=\beta^{4i+1}$, the cocycles lie in degree $k=-6$ and define classes
\begin{equation*}
    \big[\gr_{4+2i} I(\beta^{4i+1})\big] \in H^{-6}(\cGC_3)\otimes \R.
\end{equation*}
We computed a cycle $X\in\hGC_3$ whose class spans $\gr_6 H_{-6}(\hGC_3)\cong \Q$. It consists of a linear combination of 20 graphs with 12 edges, 6 loops, and 7 vertices each (\cref{fig:cycle6}). Calculating explicitly each of the corresponding integrals $I_G(\beta^5)$, we found that $I_X(\beta^5)\neq 0$:
\begin{thm}\label{thm:intro-gr6I(beta5)<>0}
    The class $[\gr_6 I(\beta^5)] \in H^{-6}(\cGC_3)\otimes\R$ is non-zero.
\end{thm}
Since $\gr_6 H^{-6}(\cGC_3)\cong\Q$ is known to have dimension one (see \cref{table:GC3}), it is spanned by this class. We have thus explained the first non-trivial cohomology group of $\cGC_3$ in degrees below $-3$ in terms of the non-compactly supported orientation form $\beta^5\wedge\pff^6$. Explicitly, we can represent this graph cohomology class using only two graphs:\footnote{%
Some integrals were  calculated numerically, see \cref{sec:integration6}. The identification of the numerical factor $2\pi^2\log 2-13\zeta(3)\approx 4.112$ is thus strictly confirmed only to 40 digits of precision, which we consider ample for this purpose. To establish the non-vanishing (\cref{thm:intro-gr6I(beta5)<>0}), two digits would have sufficed.}
\begin{equation*}
    \left[\gr_6 I(\beta^5)\right]
    = 10\cdot\Big(2\pi^2\ln 2-13\zeta(3)\Big)\cdot\left[2\times \Graph[0.8]{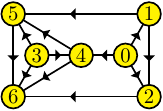} + 3\times \Graph[0.8]{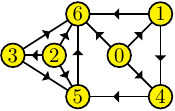}\right]
    .
\end{equation*}
This class is in the image of the Lie bracket $[H^{-3},H^{-3}]\subseteq H^{-6}$ on the cohomology of $\cGC_3$. It is proportional to the bracket $[ [K_4],[K_4] ]=-2[ [Y_3],[D_3]]$ of the complete graph $K_4$ with itself, as well as the bracket of the triangular prism $Y_3$ with the theta graph $D_3$.

The first product, $\omega=\beta^5\wedge\beta^9$, yields a cocycle $\gr_{10} I(\beta^5\wedge\beta^9)$ supported in $10$ loops and degree $k=-5$. This cocycle is exact, since the cohomology in this bidegree is trivial (see \cref{table:GC3}). However, we will explain how this exactness can be exploited to adjust the cochain $\gr_{12} I(\beta^5\wedge\beta^9)$ so as to obtain a cocycle in bidegree $(\ell,k)=(12,-9)$. While we cannot presently infer if this class is non-trivial, this kind of mechanism suggests that out of the cochains $I(\omega)$, one can construct a large number of cocycles (going beyond the restrictions of $I(\omega)$ to graphs of  smallest loop order as in  \cref{lem:intro-gr-canint=0}). We hope that those will prove useful to construct further classes in $H^\bullet(\cGC_3)$.

\subsection{Comments}
The orientation forms $\Omega_c[\pff]$ of compact type, and those of non-compact type $\Omega_{nc}[\pff]$, play opposite roles in our construction of cohomology classes for $\GL_{2n}(\Z)$ on the one hand, and for the odd graph complex $\cGC_3$ on the other hand:
\begin{itemize}
\item Compact type forms $\omega\wedge\beta^{4n-3}\wedge\pff^{2n} \in\Omega_c[\pff](2n)$ provide lots of classes in $H^{\bullet}_c(\SPD{2n}/\GL_{2n}(\Z);\ob)$ by \cref{thm:introSymembeds}, but the corresponding integrals $\gr_{2n} I(\omega\wedge\beta^{4n-3})$ in $\cGC_3$ are all zero by \cref{lem:intro-gr-canint=0}. 
\item Non-compact type forms $\omega\wedge\pff^{2n}\in\Omega_{nc}[\pff](2n)$ map to zero in $H^\bullet(\SPD{2n}/\GL_{2n}(\Z);\ob)$ by \cite{Sullivan:EulerSLZ,BismutCheeger:TransgressedEuler}, but their integrals $\gr_{2n} I(\omega)$ can give rise to non-trivial classes in $H^\bullet(\cGC_3)$, see \cref{thm:intro-gr6I(beta5)<>0}.
\end{itemize}

This picture is starkly different from the case of the \emph{even} graph complex $\cGC_2$.
There, an infinite family of explicit non-zero classes in $H^0(\cGC_2)$ was constructed using integrals over vertex positions in \cite{RossiWillwacher:Etingof}, and independently as integrals $\int_{\FS_G} \omega_G$ over edge lengths in \cite{invariant,BrownSchnetz:WheelCan}.
In the latter construction, the integrands $\omega$ used for graphs with $\ell=2i+1$ loops are the primitive forms of \emph{compact} type $\beta^{4i+1}\in\Omega_c(2i+1)$. These cocycles pair non-trivially with the wheel cycles \cite{BrownSchnetz:WheelCan} and conjecturally generate the whole of $H^0(\cGC_2)\otimes \R$ as a Lie algebra. One expects that the entirety of $H^0(\cGC_2)$ is pulled back from the compactly supported cohomology of $\GL_\ell(\Z)$, for $\ell$ odd \cite{BrownChanGalatiusPayne:HopfAGLSL}.

The even and odd cases also seem to differ in the kind of numbers that emerge from the integrals. As the example above shows, integrals of graphs in $\hGC_3$ can produce polylogarithms at 4th roots of unity and mix different weights (see \cref{sec:integration6}). In contrast, all canonical integrals computed to date for $\hGC_2$ are multiple zeta values, of homogeneous weight. However, in the odd case, we observed cancellations such that the canonical integral of the \emph{cycle} $X\in\hGC_3$ is much simpler, with homogeneous weight 3 and no 4th roots (see \cref{thm:tau1(X)}).

\subsection*{Acknowledgements}
Erik Panzer is funded as a Royal Society University Research Fellow through grant {URF{\textbackslash}R1{\textbackslash}201473}. Francis Brown thanks the University of Geneva and Trinity College Dublin for hospitality, where some of this work was carried out.
Simone Hu acknowledges the support of the Natural Sciences and Engineering Research Council of Canada (NSERC), [funding reference number 578060-2023]. Cette recherche a \'et\'e financ\'ee par le Conseil de recherches en sciences naturelles et en g\'enie du Canada (CRSNG), [num\'ero de r\'ef\'erence 578060-2023]. Francis Brown is grateful to Thomas Willwacher for suggesting the problem of defining differential forms for the odd graph complex,  to Peter Patzt for explanations  regarding \cite{AMP}, and M.\ Chan, S.\ Galatius and S.\ Payne for discussions on related topics. 

For the purpose of open access, the authors have applied a CC BY public copyright licence to any author accepted manuscript arising from this submission.
The data and code for the calculations of the cocycle in \cref{sec:gr6H-6}, as summarized in \cref{sec:integration6}, are available under
\href{https://dx.doi.org/10.5287/ora-ngnborbr4}{\texttt{doi:10.5287/ora-ngnborbr4}}.

\section{The Pfaffian form on positive definite matrices}

We denote $\SPD{\ell}\subset \R^{\ell(\ell+1)/2}$ the subset of symmetric $\ell\times\ell$ matrices $X=X^\Transpose$ which are positive definite. This space is homeomorphic to the space of right cosets of the general linear group $\GL_\ell(\R)$ by its compact subgroup $\OG_\ell(\R)$ of orthogonal matrices, via
\begin{equation*}
\OG_{\ell}(\R)  \backslash  \GL_{\ell}(\R)  \cong \SPD{\ell},\qquad
    \OG_\ell(\R) g \mapsto  g^\Transpose g. 
\end{equation*}
Thus $\SPD{\ell}$ is equipped with a transitive right action $ X\mapsto g^\Transpose X g$ by $\GL_\ell(\R)$.
The \emph{primitive canonical forms} $\beta^{4k+1}\in\Omega^{4k+1}(\SPD{\ell})$ studied in \cite[\S 4]{invariant} are smooth differential forms on $\SPD{\ell}$ of degree $4k+1$ defined for every integer $k\geq 0$ by
\begin{equation}\label{eq:primcan}
    \beta^{4k+1} \defas \tr\left( (X^{-1} \td X)^{4k+1} \right).
\end{equation}
They are closed and invariant under $\GL_\ell(\R)$. In fact, by the Hopf-Koszul-Samelson theorem
they generate the entire subalgebra of invariant forms, which is the graded exterior algebra
\begin{equation*}
    \Omega^\bullet(\SPD{\ell})^{\GL_\ell(\R)}
    = \ExtA^\bullet\left(
        \R\beta^1\oplus\R\beta^5\oplus\ldots\oplus\R\beta^{4\lfloor\frac{\ell-1}{2}\rfloor+1}
    \right).
\end{equation*}
Except for $\beta^1=\td\log\det X$, these forms are also projective, by which we mean the vanishing $\IntProd{\Evf}{\beta^{4k+1}}=0$ of the interior product with the Euler vector field
\begin{equation}\label{eq:Evf}
    \Evf = \sum_{1\leq i\leq j\leq \ell} X_{ij} \frac{\partial}{\partial X_{ij}}
\end{equation}
for all $k\geq 1$. This vector field is tangent to the orbits of the centre $\R^\times \subset\GL_\ell(\R)$, which consists of multiples $\lambda\IdMat$ of the identity and acts by positive scalings $X\mapsto \lambda^2 X$. The forms $\beta^{4k+1}$ with $k\geq 1$ therefore descend to (i.e.\ are pullbacks of) forms on the quotient
\begin{equation}\label{eq:link}
    L\SPD{\ell} \defas \SPD{\ell}/\R^\times
    \cong \SL_\ell(\R)/\SO_\ell(\R)
    .
\end{equation}
We call this quotient ``the \emph{link} of $\SPD{\ell}$''. It follows that its invariant forms are
\begin{equation}\label{eq:LGL-invariants}
    \Omega^\bullet(L\SPD{\ell})^{\GL_\ell(\R)}
    = \ExtA^\bullet\left(
        \R\beta^5\oplus\R\beta^9\oplus\ldots\oplus\R\beta^{4\lfloor\frac{\ell-1}{2}\rfloor+1}
    \right)
    .
\end{equation}

The purpose of this section is to enlarge this picture by allowing forms which are invariant under $X\mapsto g^\Transpose X g$ if $\det g>0$, but change sign if $\det g<0$.

\subsection{The invariant Pfaffian}\label{sect: InvPfaff}

The Pfaffian of a skew-symmetric $2n\times 2n$ matrix $M$  with entries in a commutative ring $R$ is  the following homogeneous polynomial of degree $n$ in the entries of $M$:
\begin{equation} \label{eqn: PfaffDefn}
    \Pf(M) = \frac{1}{2^n n!}
    \sum_{\pi
        \in\Perms{2n}
    }
    \sgn \pi
    \cdot M_{\pi(1)\pi(2)}\cdots M_{\pi(2n-1)\pi(2n)}
\end{equation}
where the sum is over the group $\Perms{2n}$ of permutations of $\{1,\ldots,2n\}$.
For a square matrix of odd dimension $(2n+1)\times(2n+1)$, we set $\Pf(M)=0$.

We will use several well-known properties of the Pfaffian:
\begin{enumerate}[label=(P\arabic*)]
  \item $\det(M)=\Pf(M)^2$, \label{P:pf2}
  \item $\Pf(\lambda M) = \lambda^n \Pf(M)$ for any $\lambda \in R$, \label{P:scale}
  \item $\Pf(P^{\Transpose}MP) = \det(P)\Pf(M)$ for any $2n \times 2n$ matrix $P$ with entries in $R$, \label{P:conj}
  \item $\Pf(M_1 \oplus M_2) = \Pf(M_1)\Pf(M_2)$, where $M_1\oplus M_2=\left(\begin{smallmatrix} M_1 & 0 \\ 0 & M_2 \end{smallmatrix}\right)$ denotes a block diagonal matrix. \label{P:block}
\end{enumerate}

Since positive definite matrices $X\in\SPD{2n}$ are invertible, we can consider on $\SPD{2n}$ the matrix
\[ \td X \cdot X^{-1} \cdot \td X. \]
Its entries are smooth $2$-forms, so they commute. This matrix is also skew-symmetric (which follows from the symmetry $X=X^\Transpose$). Therefore, its Pfaffian is defined.

\begin{definition}\label{def:pff}
Let $\SPD{n}$ denote the space of positive definite symmetric $n\times n$ matrices $X$ with real entries. We define a smooth differential form $\pff^n\in\Omega^n(\SPD{n})$ of degree $n$ on $\SPD{n}$, called the \textbf{Pfaffian form}, as follows: If $n$ is odd, we set $\pff^n=0$. For $n$ even, set
  \begin{equation}\label{eq:pff}
      \pff^n_X = \pff_X \defas \frac{1}{\sqrt{\det{X}}} \Pf\left( \td X \cdot X^{-1} \cdot \td X \right).
  \end{equation}
  Here $\sqrt{\det X}$ denotes the positive square root of $\det X>0$.
\end{definition}

\begin{lemma}\label{lem:pfprops}
  The Pfaffian form satisfies the following properties:
  \begin{enumerate}
    \item $\pff_X = 0$ if $X$ is odd-dimensional.
    \item $\pff_X \wedge \pff_X = 0$.
    \item $\pff_{P^\Transpose XP} = \sgn(\det{P})\cdot\pff_{X}$ for any matrix $P \in \GL_{n}(\R)$ with real  entries.
    \item $\pff_{X^{-1}} = \pff_{X}$.
    \item $\pff_{X_1 \oplus X_2} = \pff_{X_1} \wedge \pff_{X_2} = \pff_{X_2} \wedge \pff_{X_1}$.
    \item $\pff_X=\pff_{\lambda X}$ is basic for the scaling action $X\mapsto\lambda X$ of $\lambda\in\R_+$ on $\SPD{n}$. \label{pfprop-proj}
    \item $\mathrm{d} \pff = 0$.
  \end{enumerate}
\end{lemma}
\begin{proof}
  Most of these properties follow directly from those of the Pfaffian.
  \begin{enumerate}
    \item  By definition.
    \item As the Pfaffian squares to the determinant \ref{P:pf2},
      \[
      \pff_X \wedge \pff_X
      = (\det{X})^{-1} \det \left( \td X \cdot X^{-1} \cdot \td X\right)
      = \det \left( X^{-1} \cdot \td X \cdot X^{-1} \cdot \td X\right).
      \]
      Let $M = X^{-1} \cdot \td X \cdot X^{-1} \cdot \td X$ which has commuting entries.
      Since $\det M$ is a polynomial in the traces of powers of $M$, it suffices to show that $\tr(M^k) = 0$ for all $k$.
      This is \cite[Lemma~4.3]{invariant} with $\tr(M^k ) = \beta_X^{2k}$.
    \item  Such an invertible real matrix  $P$  is in particular constant ($\td P=0$). Therefore
      \[ \td (P^{\Transpose}XP) \cdot (P^{\Transpose}XP)^{-1} \cdot \td(P^{\Transpose}XP) = P^{\Transpose}\left(\td X \cdot X^{-1} \cdot \td X\right) P \]
      and so by applying \ref{P:conj} we obtain
      \[ \pff_{P^{\Transpose}XP} = \frac{\det{P}}{\sqrt{(\det{P})^2\det{X}}} \Pf\left(\td X \cdot X^{-1}\cdot \td X\right) = \sgn(\det{P}) \cdot \pff_X. \]
    \item Notice that $0 = \td(X\cdot X^{-1}) = \td X\cdot X^{-1} + X\cdot \td X^{-1}$ and thus
      \[ \td X^{-1} \cdot X \cdot \td X^{-1} = X^{-1} \left(\td X \cdot X^{-1} \cdot \td X \right) X^{-1}. \]
      As $X$ is symmetric and $\sqrt{(\det{X})^{-1}} = (\det{X})^{-1} \sqrt{\det{X}}$, \ref{P:conj} implies that
      \[ \pff_{X^{-1}} = \frac{\det{X^{-1}}}{\sqrt{\det{X^{-1}}}} \Pf\left(\td X \cdot X^{-1} \cdot \td X\right) = \pff_X. \]
    \item It follows immediately from \ref{P:block} that
      \[
      \pff_{X_1 \oplus X_2}
      = \frac{\Pf(\td X_1 \cdot X_1^{-1} \cdot \td X_1 \oplus \td X_2 \cdot X_2^{-1} \cdot \td X_2)}{\sqrt{\det{X_1}\det{X_2}}}
      = \pff_{X_1} \wedge \pff_{X_2}.
      \]
    \item  Let $\Phi_X=\td X\cdot X^{-1} \cdot \td X$. Then for any $\lambda>0$, we have
    \[ \Phi_{\lambda X} = (X\, \td\lambda  + \lambda\, \td X) \lambda^{-1} X^{-1} (X\, \td\lambda  + \lambda \, \td X) =\lambda \,  \Phi_X + \td\lambda \cdot  \td X + \td X \cdot \td\lambda \ . \]
Since $\td X$ is a matrix of one forms, $\td X\, \td\lambda = - \td\lambda\, \td X$ and hence $\Phi_{\lambda X} = \lambda\, \Phi_X$. It follows from property \ref{P:scale} that $\pff_{\lambda X} = (\lambda^{n}/ \sqrt{\lambda^{2n}} )\pff_X = \pff_X$.

The absence of $\td\lambda$ (thus $\frac{\partial}{\partial \lambda} \pff_{\lambda X}=0$) means that the form $\pff_X$ on $\SPD{n}$ is horizontal for the projection $\SPD{n}\longrightarrow \SPD{n}/\R_+$ with the vertical vector field $\sum_{1\leq i\leq j\leq n} X_{ij} \frac{\partial}{\partial X_{ij}}$.

    \item To show that the Pfaffian form is closed, we use
      \begin{equation*}
      \td\pff_{X} = \frac{1}{\sqrt{\det{X}}}\left[
        \td\Pf\left(\td X \cdot X^{-1} \cdot \td X\right) - \Pf\left(\td X \cdot X^{-1} \cdot \td X\right) \wedge \frac{\td\det{X}}{2\det{X}}
        \right]
    \end{equation*}
      and prove an analogue to Jacobi's formula.
      Let $M^{ij}$ be the matrix $M$ with both rows and columns $i,j$ deleted and let  $\theta(i-j) = 1$ if $i > j$ and $0$ otherwise.
      The Pfaffian expands along row $i$ (analogous to the cofactor expansion of the determinant) as
      \begin{equation}\label{eq:pff-cofactors}
      \Pf(M) = \sum_{j\neq i} (-1)^{i+j+1+\theta(i-j)} M_{ij} \Pf(M^{ij}).
      \end{equation}
      In terms of the Pfaffian adjugate matrix $M'$ defined by
      \[ M'_{ij} = \begin{cases}(-1)^{i+j+\theta(i-j)} \Pf(M^{ij}) & \text{if $i \neq j$ and} \\ 0 & \text{if $i = j$,} \end{cases}\]
      the expansion shows that the diagonal entries of $M'\cdot M$ are $\Pf(M)$. Off-diagonal entries $\sum_j M'_{ij}M_{jk}$ vanish (consider replacing the $i$th row and column of $M$ with a copy of the $k$th). Hence we get a multiple of the identity matrix:
      \[ M' \cdot M = \Pf(M) \cdot \IdMat_n. \]
      It follows already from \eqref{eq:pff-cofactors} that
      $\partial (\Pf{M})/\partial M_{ij} = M'_{ji}$
      and thus
      \[ \td \Pf(M) = \sum_{i < j} M'_{ji} \td M_{ij} = \frac{1}{2} \tr(M'\cdot\td M). \]
      Letting $M = \td X\cdot X^{-1} \cdot\td X$ and using $\td X^{-1} = -X^{-1}\cdot \td X \cdot X^{-1}$ (see above), we note
      $
      \td M = -\td X \cdot \td X^{-1} \cdot \td X
      = M \cdot X^{-1} \cdot \td X
      $
      and therefore
      \begin{equation*}
        \td\Pf(M)
          =\tr(M'\cdot M \cdot X^{-1}\cdot \td X)/2
          = \Pf(M) \wedge \tr(X^{-1} \cdot \td X)/2.
      \end{equation*}
      With Jacobi's formula $\tr(X^{-1} \cdot\td X) = (\td \det{X})/(\det X)$, this proves $\td \pff_X = 0$.
  \end{enumerate}
\end{proof}
The multiples $g=\lambda\IdMat_n\in\GL_n(\R)$ of the identity matrix constitute the central subgroup $\R^\times\subset\GL_n(\R)$. They act as scalings on $\SPD{n}$, defining a map
\begin{equation*}
    \SPD{n}  \times \R^\times  \longrightarrow \SPD{n}, \qquad
    (X,g)\mapsto \lambda^2 X.
\end{equation*}
Part \ref{pfprop-proj} above says that under this scaling action, the pullback of the Pfaffian is independent of $\lambda$ (i.e.\ no $\td \lambda$ differential). Hence $\pff$ descends to a form on the quotient
\begin{equation*}
    L\SPD{n}=\SPD{n}/\R^\times
    \quad\subset\quad
    \Pro\big(\R^{n(n+1)/2}\big)
    .
\end{equation*}

\begin{example}\label{ex:pf_2}
The $2\times 2$ symmetric, positive definite matrices are
\begin{equation*}
\SPD{2}=\left\{
X = \begin{pmatrix} a & b \\ b & c \end{pmatrix}
\colon
a,c>0\ \text{and}\ b^2<ac
\right\}.
\end{equation*}
In terms of $\Omega_3 = -a \cdot \td b \wedge \td c + b \cdot \td a \wedge \td c - c \cdot \td a \wedge \td b$, we have
\begin{equation*}
  \td X \cdot X^{-1} \cdot \td X = \frac{1}{ac - b^2}
    \begin{pmatrix}
      0 & -\Omega_3 \\
      \Omega_3  & 0
    \end{pmatrix}
\quad\text{and thus}\quad
  \pff^2_X = \frac{-\Omega_3}{(ac - b^2)^{3/2}}.
\end{equation*}
\end{example}
Notice that $\det X=ac-b^2$ in the denominator of $\pff^2$ becomes the elementary symmetric polynomial $x_1x_2+x_1x_3+x_2x_3=\det X$ in the coordinates $(a,b,c)=(x_1+x_3,x_3,x_2+x_3)$, and $\Omega_3=x_1\td x_2\wedge\td x_3-x_2\td x_1\wedge\td x_3+x_3 \td x_1\wedge \td x_2$. This formula generalizes to higher $n$, if we restrict to the subspace of matrices where all off-diagonal entries are equal:
\begin{lemma} \label{lem:pff(diag+const)}
For any $i \geq 1$, consider the $(2i+1)$-dimensional subspace of $\SPD{2i}$ consisting of matrices that can be written as the sum of a diagonal matrix with positive entries and a matrix with all entries equal and positive. This subspace is parametrized by
\begin{equation*}\setlength{\arraycolsep}{2pt}
    \Laplacian\colon \R_{>0}^{2i+1} \hookrightarrow \SPD{2i},\qquad
    (x_1,\ldots,x_{2i+1})\mapsto \begin{pmatrix}
        x_1 & & \\[-4pt]
            & \ddots & \\[-4pt]
        & & x_{2i} \\
    \end{pmatrix}
    +x_{2i+1} \begin{pmatrix}
        1 & \cdots & 1 \\[-3pt]
        \vdots &  & \vdots \\
        1 & \cdots & 1 \\
    \end{pmatrix}.
\end{equation*}
Set $\Omega_{2i+1}=\sum_{a=1}^{2i+1} x_a(-1)^{a-1} \td x_1\wedge\ldots\wedge \td x_{a-1}\wedge \td x_{a+1} \wedge \ldots\wedge \td x_{2i+1}$ 
and furthermore denote the $2i$-th elementary symmetric polynomial as $\ESP_{2i}(x)=\sum_{a=1}^{2i+1} \prod_{b\neq a} x_b$.

Then the restriction of the Pfaffian form $\pff^{2i}$ to this  subspace equals
\begin{equation}\label{eq:pff(diag+const)}
  \Laplacian^\ast \pff^{2i} = (-1)^{i}\frac{(2i)!}{2^{i}\cdot i!}\frac{(x_1\cdots x_{2i+1})^{i-1}}{[\ESP_{2i}(x)]^{i+1/2}}
\Omega_{2i+1}
  .
\end{equation}
\end{lemma}
\begin{proof}
    The determinant and inverse of the matrix $\Laplacian_{a,b}=x_{2i+1}+x_a\delta_{a,b}$ are
    \begin{equation*}
        \det \Laplacian=\ESP_{2i}=\sum_{a=1}^{2i+1} \prod_{b\neq a} x_b
        \quad\text{and}\quad
        \left(\Laplacian^{-1}\right)_{a,b} = \frac{1}{\ESP_{2i}}
        \begin{cases}
            \partial \ESP_{2i}/\partial {x_a} & \text{if $a = b$,} \\[3pt]
            -\displaystyle\frac{x_1 \cdots x_{2i+1}}{x_a x_b} & \text{if $a\neq b$}
      \end{cases}
    \end{equation*}
    and from there one computes that
    \begin{equation*}
      \left(\td \Laplacian \cdot \Laplacian^{-1} \cdot \td \Laplacian\right)_{a,b} = -\frac{x_1\cdots x_{2i+1}}{\ESP_{2i}} \left[
        \frac{\td x_a}{x_a}\wedge\frac{\td x_b}{x_b} - \frac{\td x_a}{x_a}\wedge \frac{\td x_{2i+1}}{x_{2i+1}} + \frac{\td x_b}{x_b}\wedge \frac{\td x_{2i+1}}{x_{2i+1}}
        \right].
    \end{equation*}
    Since $\pff$ is projective, we only need to establish \eqref{eq:pff(diag+const)} in the dense affine chart $x_{2i+1}=1$. So apart from the factor $(-x_1\cdots x_{2i+1}/\ESP_{2i})^i$ due to the homogeneity property \ref{P:scale} of the Pfaffian, we only need to apply \eqref{eqn: PfaffDefn} to obtain 
    \begin{equation*}
        \Pf\left( \frac{\td x_a}{x_a}\wedge\frac{\td x_b}{x_b} \right)_{1\leq a,b\leq 2i}
        =\frac{1}{2^i\cdot i!} \sum_{\pi\in\Perms{2i}} (\sgn\pi) \frac{\td x_{\pi(1)}}{x_{\pi(1)}}\wedge\ldots\wedge\frac{\td x_{\pi(2i)}}{x_{\pi(2i)}}
    \end{equation*}
    and note that all $(2i)!$ summands are the same, namely $(\td x_1)/x_1\wedge\ldots\wedge (\td x_{2i})/x_{2i}$.
\end{proof}
\begin{corollary}\label{cor:pff<>0}
    For even $n$, the differential form $\pff^{n}$ on $\SPD{n}$ does not vanish anywhere.
\end{corollary}
\begin{proof}
    By \cref{lem:pff(diag+const)}, $\pff^n$ is non-vanishing on a non-empty subset of $\SPD{n}$. Using part (iii) of \cref{lem:pfprops}, the claim follows from transitivity of the $\GL_n(\R)$-action on $\SPD{n}$.
\end{proof}

\subsection{Asymptotics} \label{sect: Asymptotics}

We now consider the behaviour of the invariant Pfaffian $\pff_X^n$ when the matrix $X$ becomes singular.
Specifically, we study degenerations parametrized by block matrices
\begin{equation}\label{eq:SPD-asy-chart}
    X=\begin{pmatrix} z A & z B \\ z B^\Transpose & C \end{pmatrix}
\end{equation}
with positive definite $A\in\SPD{n_1}$, $C\in\SPD{n_2}$, and some $n_1\times n_2$ matrix $B$ with real entries. For $0<z$ sufficiently small, such a matrix $X$ is positive definite.\footnote{The matrix $X$ is positive definite if and only if both $zA$ and $C-z B^\Transpose A^{-1} B$ are positive definite, hence for $0<z<z_0(A,B,C)$ where $z_0$ denotes the smallest positive zero of $\det(C-z B^\Transpose A^{-1}B)$.}
Therefore, $\pff^n_X$ is a smooth form on an open domain in $(A,C,B,z)\in \SPD{n_1}\times\SPD{n_2}\times\R^{n_1 n_2}\times \R$. We may study its behaviour as $z\rightarrow 0$.
\begin{example}\label{ex:pff2-asy}
    For $n_1=n_2=1$, the blocks are scalar ($1\times 1$ matrices), and the form from \cref{ex:pf_2} becomes, in the chart \eqref{eq:SPD-asy-chart} where $a=zA$, $b=zB$, $c=zC$:
    \begin{equation*}
        \pff^2_{X}=\frac{\td z}{\sqrt{z}}\wedge \frac{C(A\td B-B\td A)}{(AC-zB^2)^{3/2}}+\sqrt{z}\cdot\frac{A\td B\wedge\td C-B\td A\wedge\td C+C\td A\wedge\td B}{(AC-zB^2)^{3/2}}
        .
    \end{equation*}
\end{example}
\begin{lemma}\label{lem:pff-asy}
    Let $n_1,n_2>0$ be integers with $n=n_1+n_2$ even and let $X$ as in \eqref{eq:SPD-asy-chart}. If $n_1$ and $n_2$ are even, then $\pff_{X}^n$ extends to a holomorphic form in a neighbourhood of $\{z=0\}$, and its pullback to $\{z=0\}$ is the product $\pff_A^{n_1}\wedge\pff_C^{n_2}$. If $n_1$ and $n_2$ are odd, then there exist forms $\alpha$ and $\beta$, holomorphic in a neighbourhood of $\{z=0\}$, such that
\begin{equation*}
    \pff^n_{X}=\sqrt{z}\cdot\alpha+(\td \sqrt{z})\wedge\beta.
\end{equation*}
\end{lemma}
\begin{proof}
    Since $C$ is invertible, the matrix $S=C-zB^\Transpose A^{-1} B$ is also invertible (for small $z$) and $z\mapsto S^{-1}$ is a holomorphic function of $z$ at $z=0$. The inverse of $X$ can be written as
    \begin{equation*}
        X^{-1} = \begin{pmatrix}
            z^{-1}A^{-1}+A^{-1} B S^{-1}B^{\Transpose} A^{-1} & - A^{-1} B S^{-1} \\
            -S^{-1} B^\Transpose A^{-1} & S^{-1} \\
        \end{pmatrix}
        .
    \end{equation*}
    For matrices $T,U$ and an ideal $I$, write $T\equiv U\mod I$ to indicate that all entries of $T-U$ lie in $I$. For example, $S\equiv C \mod (z)$ and $S^{-1}\equiv C^{-1} \mod (z)$. Let $I=(z,\td z)$ denote the ideal  generated by $z$ and $\td z$, inside the ring of differential forms that are holomorphic at $z=0$. Note that $I^k=(z^k,z^{k-1}\td z)$. Multiplying the matrices, one finds
    \begin{equation*}
        M\defas(\td X)X^{-1}(\td X)
        \equiv
        \begin{pmatrix}
            z (\td A)A^{-1}(\td A) \mod I^2 & 0 \mod I \\
            0 \mod I & (\td C)C^{-1}(\td C) \mod I \\
        \end{pmatrix}
        .
    \end{equation*}
    The lowest order of $\Pf(M)$ in $I$ arises by taking as many factors of the summand in \eqref{eqn: PfaffDefn} as possible from the bottom-right block. If $n_1$ and $n_2$ are even, we can take $n_2/2$ such factors from the bottom-right, and are then forced to take all remaining factors from the top-left. Using property \ref{P:block} of the Pfaffian, this shows that
    \begin{equation*}
        \Pf(M)
        \equiv z^{n_1/2} \Pf(\td A\cdot A^{-1}\cdot \td A)\wedge \Pf(\td C\cdot C^{-1}\cdot \td C)
        \mod I^{n_1/2+1}.
    \end{equation*}
    Analogous reasoning shows that  $\det X\equiv \det (zA) (\det C)\mod z^{n_1+1}$, thus $\sqrt{\det X}=z^{n_1/2}R$, where $R=\sqrt{(\det X)/z^{n_1}}\equiv \sqrt{(\det A)(\det C)} \mod z$ is analytic and invertible at $z=0$. Therefore, dividing $\Pf(M)$ by $\sqrt{\det X}$ shows that
\begin{equation*}
    \pff^n_{X}\equiv \pff_A^{n_1}\wedge\pff_C^{n_2} \mod I.
\end{equation*}

    For $n_1$ and $n_2$ odd, we can take at most $(n_2-1)/2$ factors from the bottom-right block of $M$. Hence we must have at least $(n_1+1)/2$ factors from elsewhere, so that $\Pf(M)\equiv 0 \mod I^{(n_1+1)/2}$. Thus, there are holomorphic forms $\alpha',\beta'$ such that
    \begin{equation*}
        \Pf(M)=z^{(n_1+1)/2} \alpha' + z^{(n_1-1)/2} \td z \wedge \beta'.
    \end{equation*}
    Dividing by $\sqrt{\det X}= z^{n_1/2}R$ as before, we are left with $\pff^n_{X}=\sqrt{z}\cdot \alpha+(\td \sqrt{z})\wedge \beta$ where the forms $\alpha= \alpha'/R$ and $\beta=2\beta'/R$ are holomorphic at $z=0$.
\end{proof}

For the asymptotics of canonical forms \eqref{eq:primcan}, we quote from \cite[Theorem~7.4]{Brown:FeynmanAmplitudesGalois}:
\begin{lemma}\label{lem:can-asy}
    For integers $i,n_1,n_2>0$ and $X$ as in \eqref{eq:SPD-asy-chart}, the form $\beta_{X}^{4i+1}$ extends holomorphically over $z=0$, and its pullback to $\{z=0\}$ is $\beta^{4i+1}_A + \beta^{4i+1}_C$.
\end{lemma}

\subsection{Hopf algebra of forms}
\label{sec:Hopf-can}

By the \emph{algebra of canonical forms} we mean the graded $\Q$-algebra
\begin{equation*}
    \cfs
    = \bigoplus_{d \geq 0} \cfs[d]
    =\ExtA^\bullet\left(\bigoplus_{k\geq 1} \Q\beta^{4k+1}\right)
\end{equation*}
with one generator $\beta^{4k+1}$ of degree $4k+1$ for each integer $k\geq 1$.
Let $\SPD{n}$ denote the space of positive definite, symmetric $n\times n$ matrices $X$ over $\R$. The algebra $\Omega^{\bullet}(\SPD{n})$ of smooth differential forms on $\SPD{n}$ receives a homomorphism
\begin{equation*}
    \rho_n\colon \cfs\longrightarrow\Omega^{\bullet}(\SPD{n})
    \quad\text{defined by}\quad
    \beta^{4k+1}\mapsto \beta^{4k+1}_X\defas \tr\left( (X^{-1}\td X)^{4k+1} \right).
\end{equation*}
It is known  (e.g. \cite[Proposition~4.5]{invariant}) that $\beta^{4k+1}_X$ vanishes if $n<2k+1$. In particular, for $g>1$ odd and  all $n\leq g+1$,  $\rho_{n}$ factors through the sum $\Omega(g)= \Omega_c(g) \oplus \Omega_{nc}(g)$ of the spaces defined in \cref{sec:formspaces}. This sum can be identified with the quotient of $\cfs$ by setting $\beta^{4k+1}$ to zero for all $4k+1>2g-1$.

We write $\Delta\colon \cfs\rightarrow \cfs\otimes\cfs$ for the (graded) cocommutative coproduct for which every generator is primitive, that is,
\begin{equation*}
    \Delta\left(\beta^{4k+1}\right) = 1 \otimes \beta^{4k+1} + \beta^{4k+1} \otimes 1.
\end{equation*}
In Sweedler notation $\Delta(\omega)=\sum_{(\omega)}\omega_{(1)}\otimes \omega_{(2)}$, graded cocommutativity means that
\begin{equation}\label{eq:graded-cocom}
	\sum_{(\omega)} \omega_{(1)}\otimes\omega_{(2)} = \sum_{(\omega)} (-1)^{|\omega_{(1)}|\cdot|\omega_{(2)}|} \omega_{(2)}\otimes \omega_{(1)}
\end{equation}
where $|\omega|$ denotes the degree.

Under the representation on $\Omega^{\bullet}(\SPD{n})$, this coproduct encodes the restriction onto subspaces of block-diagonal matrices. Namely, under the map
\begin{equation} \label{eqn: blocksumonpositivedefn}
    \iota_{m,n}\colon \SPD{m}\times\SPD{n} \longrightarrow \SPD{m+n},\qquad
    (A,B) \mapsto
    A\oplus B \defas
    \begin{pmatrix}
       A & 0 \\
       0 & B \\
    \end{pmatrix},
\end{equation}
the additivity of the trace implies that $\beta_{A\oplus B}^{4k+1}=\beta_A^{4k+1}+\beta_B^{4k+1}$. In other words, the following diagram commutes:
\begin{equation*}\xymatrix{
    \cfs \ar[rrr]^{\rho_{m+n}} \ar[d]^{\Delta} & & & \Omega^{\bullet}(\SPD{m+n}) \ar[d]^{\iota_{m,n}^*} \\
    \cfs\otimes\cfs \ar[rr]^{\rho_m\otimes \rho_n} & & \Omega^{\bullet}(\SPD{m})\otimes\Omega^{\bullet}(\SPD{n}) \ar[r] & \Omega^{\bullet}(\SPD{m}\times\SPD{n})
}\end{equation*}

\begin{remark}
	The Pfaffian of a block-diagonal matrix factorizes as $\pff_{A\oplus B} = \pff_A\wedge \pff_B$ (see \cref{lem:pfprops}). To encode this property in the Hopf algebra formalism, one could add one generator $\pff^{2k}$ in each positive even degree to $\cfs$, subject to the relations $\pff^{2k}\pff^{2l}=0$ and $\beta^{4k+1}\pff^{2l}=\pff^{2l}\beta^{4k+1}$ for all $k,l>0$. The realizations and coproduct  extend via
\begin{equation*}
    \rho_n(\pff^{2k})=\begin{cases}
        \pff_X & \text{if $n=2k$ and} \\
        0 & \text{otherwise,} \\
    \end{cases}
    \qquad\qquad
    \Delta \pff^{2k} = \sum_{i=1}^{k-1} \pff^{2i} \otimes \pff^{2(k-i)}.
\end{equation*}
These relations can also be phrased as $\Delta\pff=\pff\otimes\pff$ for the generating series $\pff=\sum_{k>0} \pff^{2k}$ in the degree completion of this Hopf algebra.
\end{remark}

\subsection{Order of the pole}
For a positive definite symmetric $2n \times 2n$ matrix $X$, it follows from \eqref{eq:pff} that $\pff_X$ can be written in terms of some homogeneous polynomial $Q(X)$ of degree $n(2n+1)$ in the variables $X_{ij}$ of the entries of $X$ and their differentials $\td X_{ij}$, such that
\begin{equation}\label{eq:pff-pole}
    \pff_X = \frac{Q(X)}{(\det{X})^{n+1/2}}
\end{equation}
since each entry of the matrix $\td X \cdot X^{-1} \cdot \td X$ is a polynomial $2$-form of homogeneity degree $2n+1$, divided by $\det{X}$.
The exponent in the denominator of \eqref{eq:pff-pole} is best possible, by \cref{lem:pff(diag+const)}.
Similarly, from \eqref{eq:primcan} it is clear that $\beta_X^{4k+1}$ is a polynomial divided by $(\det{X})^{4k+1}$. In fact, the pole order is lower \cite[Theorem~5.2]{invariant} and one has:
\begin{equation}\label{eq:primcan-pole}
  \beta_X^{4k+1} = \frac{Q_k(X)}{(\det{X})^{k+1}}
\end{equation}
for polynomial forms $Q_k(X)$ homogeneous of degree $2n(k+1)$. By \cite[Proposition~4.5]{invariant} we know that $\beta_X^{4k+1}=0$ when $k\geq n$, so the order of the pole of $\beta_X^{4k+1}$ never exceeds $n$. This bound also  persists  for arbitrary polynomials in $\beta_X^{4k+1}$ and the Pfaffian form:
\begin{thm}\label{thm:pfbeta-order}
  For symmetric $2n\times 2n$ matrices $X$ and any canonical form $\omega \in \cfs$, the pole of the form $\omega_X \wedge \pff_X$ at $\det(X)=0$ has order no more than $n+1/2$.
  That is,
\begin{equation*}
\omega_X \wedge \pff_X =\frac{Q_\omega(X)}{(\det X)^{n+1/2}}
\end{equation*}
  for some polynomial form $Q_\omega(X)$ in the matrix entries $X_{ij}$ and their differentials $\td X_{ij}$.
\end{thm}
\begin{proof}
Following \cite[\S 5.2]{invariant}, we parametrize $X = U^\Transpose D U$ in terms of block matrices
\begin{equation*}
U = \begin{pmatrix} \IdMat & u \\ 0 & 1 \end{pmatrix},\qquad
D = \begin{pmatrix} B & 0 \\ 0 & b \end{pmatrix},\qquad
b = \frac{\det(X)}{\det(B)},
\end{equation*}
where $\IdMat$ denotes the $(2n-1)\times(2n-1)$ identity matrix, $u$ is a column vector of size $(2n-1)$, and the matrix $B$ is $X$ without the last row and column. For each $i\geq 0$, define
\begin{equation*}
\nu_i = b^{-1} \cdot \td u^\Transpose \cdot B \cdot\big(B^{-1}\cdot\td B\big)^i \cdot \td u.
\end{equation*}
These are $(i+2)$-forms in the entries of $X$. Indeed, since $u=B^{-1} (X_{j,2n})_{1\leq j < 2n}$, we can express $\td u$ in terms of $\td X_{j,2n}$ and $\td (B^{-1})=-B^{-1}\cdot\td B\cdot B^{-1}$. Therefore, we can write
\begin{equation*}
    \nu_i = \frac{Q_{\nu_i}(X)}{(\det B)^{i+2}(\det X)}
\end{equation*}
for some polynomial $Q_{\nu_i}(X)$ in the entries of $X$ and their differentials. It was shown in \cite[\S 5.2]{invariant} that $\nu_0=\nu_1=0$ and that a primitive canonical form can be expressed as
\begin{equation}\label{eq:beta-expn}
  \beta^{4k+1}_X = \beta_B^{4k+1} + \text{words in}\ \Big\{\nu_i, (\td b)/b\Big\}.
\end{equation}
The form $\beta_B^{4k+1}$ has no pole on $\det X=0$, and each $\nu_i$ and $(\td b)/b$ has a simple pole on $\det X=0$. Since the $\nu_i$ have degree $\geq 4$ (as $i\geq 2$), each word can have at most $k$ of them. Also, $(\td b)/b$ can appear at most once since $\td b\wedge \td b=0$. The bound \eqref{eq:primcan-pole} follows.

Now consider a product $\omega = \beta^{4k_1+1} \wedge \cdots \wedge \beta^{4k_m+1}$ and expand each primitive form via \eqref{eq:beta-expn}. Let $[\td[r] u] \omega_X$ denote the part of $\omega_X$ that is homogeneous of degree $r$ in the differentials $\td u$. Each $\nu_i$ is quadratic in $\td u$, thus $\omega_X$ has only even degrees in $\td u$ and
\begin{equation*}
    [\td[2r] u] \omega_X \in \frac{1}{b^r} \Q\Big[X_{ij},\td X_{ij},(\det B)^{-1},(\td b)/b \Big].
\end{equation*}

Turning to the Pfaffian, for $X=U^\Transpose D U$ we find $\td X \cdot X^{-1}\cdot \td X=U^\Transpose M U$ and therefore $\pff_X=\Pf(M)/\sqrt{\det X}$ for the skew-symmetric matrix
\begin{equation*}
  M=\left(
  \renewcommand{\arraystretch}{1.5}
  \begin{array}{c|c}
    \td B\cdot B^{-1} \cdot\td B + b^{-1} \left(B \cdot \td u\cdot\td u^\Transpose \cdot B\right) &
    \td B\cdot \td u + \left(B\cdot \td u\right) b^{-1}\td b
    \\
    \hline
    \td u^\Transpose \cdot\td B + b^{-1}\td b\left(\td u^\Transpose\cdot B\right) & 0
  \end{array}\right)
  .
\end{equation*}
Expanding $\Pf(M)$ via \eqref{eqn: PfaffDefn} as a polynomial in the entries of $M$, note that the last row of $M$ has degree one in $\td u$. All remaining factors of $\Pf(M)$ come from the top-left block, where $\td u$'s come in pairs and each pair is accompanied by a factor $b^{-1}$. Therefore,
\begin{equation*}
    [\td[2r+1] u] \Pf(M) \in \frac{1}{b^r} \Q\Big[X_{ij},\td X_{ij},(\det B)^{-1},(\td b)/b\Big].
\end{equation*}

In total, $\td u$ has only $2n-1$ entries. Therefore, the maximal degree of $\omega_X \wedge \Pf(M)$ in $\td u$, and thereby the highest possible power of $b^{-1}$, comes from terms of the form
\begin{equation*}
    ([\td[2n-2-2r] u] \omega_X) \cdot ([\td[2r+1]u] \Pf(M)) 
    \in \frac{1}{b^{n-1}} \Q\Big[X_{ij},\td X_{ij},(\det B)^{-1},(\td b)/b\Big].
\end{equation*}
Since there can be at most one factor of $(\td b)/b$, the form $\omega_X \wedge \Pf(M)$ has a pole of order at most $n$ on $\det X=0$. Dividing by $\sqrt{\det X}$ finishes the proof.
\end{proof}

\begin{example}
Consider the canonical form $\omega=\beta^5\wedge\beta^9\wedge\ldots\wedge\beta^{4n-3}$. Then the form $\omega_X \wedge \pff^{2n}_X$ on $\SPD{2n}$ has degree $2n^2+n-1$. The denominator in \cref{thm:pfbeta-order} has degree $2n^2+n$, so by projectivity it follows that $\omega_X \wedge \pff^{2n}_X=c_n \eta^{2n}$ is a multiple of the volume form \eqref{eq:SPD-vol} on $L\SPD{2n}$, for some $c_n\in\Q$. We prove in \cref{prop: gammawedgestoeta} that $c_n\neq 0$.
\end{example}

\section{Cohomology classes of Pfaffian invariant forms}
Let $\SPD{n}$ denote the space of positive-definite symmetric $n\times n$ matrices $X$ with real entries $X_{ij}=X_{ji}\in\R$.
The invariant Pfaffian forms $\pff^{2n}\in\Omega^{2n}(\SPD{2n})$ descend to differential forms on the locally symmetric space $\SPD{2n}/\GL_{2n}(\Z)$ with coefficients in the orientation bundle. We shall use them to represent compactly supported cohomology classes on $\SPD{2n}/\GL_{2n}(\Z)$ with coefficients in $\ob$.

\subsection{Volume form}
We first need some preliminary results on the exterior products of the Pfaffian and canonical differential forms.
Let  $\varepsilon\colon \GL_{2n}(\R) \rightarrow   \Z^{\times}$ be the sign of the determinant:
\begin{equation*}
\varepsilon(g)= \det(g) \, |\det(g)|^{-1}\ . \nonumber
\end{equation*}
Consider the volume form on $\SPD{2n}$ given by
\begin{equation*}
\frac{ 1 }{\det(X)^{\frac{2n+1}{2}} } \,  \bigwedge_{i\leq j}  \td X_{ij}
\in\Omega^{n(2n+1)}(\SPD{2n})
\end{equation*}
where the square root is  positive  and the exterior product is ordered lexicographically in $(i,j)$.
Contraction with the Euler vector field \eqref{eq:Evf} produces a differential form
\begin{equation}\label{eq:SPD-vol}
\eta^{2n}(X) =
\IntProd{\Evf}{\left(  \frac{ 1 }{\det(X)^{\frac{2n+1}{2}} } \,  \bigwedge_{i\leq j}  \td X_{ij}\right) }
\end{equation}
which is closed,  projectively invariant (i.e. passes to the quotient $\SPD{2n}/\R^{\times}_{>0}$),  and  satisfies
\begin{equation}\label{etatransformationlaw}
\eta^{2n}(g^\Transpose X g) = \varepsilon(g) \eta^{2n}(X) \quad 
\text{for all}\quad g \in \GL_{2n}(\R),    X \in  \SPD{2n}\ .
\end{equation}
Let $\ob$ denote the orientation bundle on the orbifold $\SPD{2n}/\GL_{2n}(\Z)$. It is the rank one local system of $\R$  vector spaces defined by the representation $\varepsilon\colon \GL_{2n}(\Z)\rightarrow \R^{\times}$.  It descends to a local system on the quotient $L  \SPD{2n}/\GL_{2n}(\Z)$, where
$L\SPD{2n} = \R_{>0}^{\times} \setminus \SPD{2n}$.
We shall call an \emph{orientation form} any differential form with coefficients in the orientation bundle. Such forms satisfy the transformation law $\omega(g^\Transpose Xg) = \varepsilon(g) \omega(X)$  for all $X\in \SPD{2n}$ and $g\in \GL_{2n}(\Z)$. They   are called `odd' forms in the terminology of de Rham.

With this definition, \eqref{etatransformationlaw} implies that $\eta^{2n} $ defines   a closed  orientation form on $L  \SPD{2n}/\GL_{2n}(\Z)$. We first show that $\eta^{2n}$
factorises into Pfaffian and canonical forms.

\begin{proposition}  \label{prop: gammawedgestoeta} Let $n\geq1$.
There is a non-zero constant $c_n\in \Z$ such that
\[
   \beta^5 \wedge \ldots \wedge \beta^{4n-3} \wedge \pff^{2n} = c_{n} \, \eta^{2n}.
\]
In particular, the map  $\omega \mapsto  \omega \wedge \pff^{2n} $ from $\Omega^{\bullet}_{nc}(2n-1)$ to  the space of differential  forms on $L\SPD{2n}$ satisfying \eqref{etatransformationlaw}, is injective.
\end{proposition}
\begin{proof}
Consider the form $\delta^{2n}= \star \pff^{2n}$ where $\star$ is the Hodge star operator on $L \SPD{2n}$ relative to the orientation bundle $\ob$ on $L \SPD{2n}$.
   It is uniquely defined by the property that
\begin{equation} \label{defn: delta}
\omega \wedge \delta^{2n} = \langle \omega, \pff^{2n} \rangle\,  \eta^{2n}
\end{equation}
for all smooth differential forms $\omega$ on $L \SPD{2n}$ of degree $2n$, and has  complementary degree $\dim L\SPD{2n}-2n=d_{2n}-1-2n$. In the above, 
$\langle  \  , \  \rangle $ is the pairing:
\begin{equation*}
\Big(\textstyle{ \ExtA^{m}}\,  T^* L\SPD{2n}\Big)
\otimes
\Big(\textstyle{\ExtA^{m}} \, T^* L\SPD{2n}\Big)
\rightarrow  \R
\end{equation*}
induced  by taking exterior powers of any $\GL_{2n}(\R)$-invariant inner product on the cotangent space $ T^* L\SPD{2n}$.
  For any  $g\in \GL_{2n}(\R)$, we deduce that
\begin{align*}
g^* \omega \wedge  g^* \delta^{2n}
&= g^*(\langle \omega, \pff^{2n} \rangle ) \,  g^* \eta^{2n} \\
&= \langle g^*\omega, g^*\pff^{2n} \rangle\,  g^* \eta^{2n} \\
&= \langle g^*\omega,  \varepsilon(g) \pff^{2n} \rangle\,  \varepsilon(g) \eta^{2n} \\
&= \langle g^*\omega,  \pff^{2n} \rangle\,  \eta^{2n}
\end{align*}
where the second equality  follows from $\GL_{2n}(\R)$-invariance of the inner product, the third from the transformation properties of $\pff^{2n}$ and $\eta^{2n}$ \eqref{etatransformationlaw}, and the fourth line from the fact that $\varepsilon^2=1$. Since this equation is true for all $\omega$, we
may replace $g^*\omega$ with $\omega$ to deduce that
$ \omega \wedge g^* \delta^{2n}   =  \langle \omega,  \pff^{2n} \rangle\,  \eta^{2n} $
for all $\omega$, which  on comparison with   \eqref{defn: delta} implies that $g^* \delta^{2n} = \delta^{2n}$ for all $g\in \GL_{2n}(\R)$. Thus $\delta^{2n}$ is $\GL_{2n}(\R)$-invariant.
Now the space of $\GL_{2n}(\R)$-invariant differential forms on $L\SPD{2n}$ is known, as a consequence of the Hopf-Koszul-Samelson theorem, to be equal to the graded exterior algebra on the canonical forms $\beta^5,\ldots, \beta^{4n-3}$. By comparing degrees, we conclude that
\[ \delta^{2n} = \alpha_{n} \,  \beta^5 \wedge \ldots \wedge \beta^{4n-3}\]
for some $\alpha_n \in \R$. Substituting  into  \eqref{defn: delta} with $\omega = \pff^{2n}$, gives
\[  \alpha_n  \,    \beta^5 \wedge \ldots \wedge \beta^{4n-3} \wedge \pff^{2n} = \langle \pff^{2n} , \pff^{2n} \rangle\,  \eta^{2n} \ .\]
Note that the forms $\pff^{2n}$ and $\eta^{2n}$ are nowhere zero, by \cref{cor:pff<>0}  and by definition, respectively.
Consequently, $\langle \pff^{2n}, \pff^{2n} \rangle> 0$ and the right-hand side of the previous equation is non-zero.   We deduce that $\alpha_n \neq 0$, and hence the equation in the proposition is true for the constant
\[  c_n = \alpha_n^{-1} \langle \pff^{2n} , \pff^{2n} \rangle \ . \]
Finally, note that the forms $\pff^{2n}, \beta^{5}, \ldots, \beta^{4n-3}$ are  defined integrally. In other words,  at the tangent space at the identity matrix $\IdMat_{2n}$, they lie in the graded exterior algebra  on the  integral Lie algebra generated by the entries $\td X_{ij}$. It follows from this fact, and the definition of $\eta^{2n}$, that $c_n$ is necessarily an integer.

The last statement concerning injectivity follows from the fact that, for any non-zero element $\omega$ in the exterior algebra generated by $\beta^5,\ldots, \beta^{4n-3}$, there exists an element $\omega'$ in the same algebra such that $\omega \wedge \omega' = \beta^5 \wedge \ldots \wedge \beta^{4n-3}$.  Therefore  $\omega \wedge \pff^{2n}$ is a factor of  $\pff^{2n} \wedge \omega \wedge \omega'= c_n \eta^{2n}\neq 0$ and therefore does not vanish.
\end{proof}

\begin{example}\label{ex:vol-c}
    We saw that $c_1=1$ in \cref{ex:pf_2}. With computer calculations we also found the values $c_2 = -180$ and $c_3 = - 5\times 10!$.
\end{example}

\subsection{Twisted cohomology of \texorpdfstring{$\GL_{2n}(\Z)$}{GL\_2n(Z)}}
For all $g> 1$ odd, let  $\Omega^{\bullet}(g)$ be the graded exterior  algebra over $\Q$ generated by  the canonical forms $\beta^5,\ldots, \beta^{2g-1}$.  One has a direct sum decomposition into forms of compact and non-compact types:
\[ \Omega(g) = \Omega_c(g) \oplus \Omega_{nc}(g)\ , \]
 where recall that $\Omega_c(g)\subset \Omega(g)$ is  the ideal spanned by $\beta^{2g-1}$, and $\Omega_{nc}(g)\subset \Omega(g)$ is the subalgebra generated by $\beta^5,\ldots, \beta^{2g-5}$.
By \cref{prop: gammawedgestoeta}  the Hodge star operator for $L\SPD{2n}$ relative to the orientation bundle $\ob$  induces a linear  involution of  vector spaces
 \[\star \colon       \Omega_c(2n-1)   \wedge \pff^{2n}  \ \overset{\sim}{\To} \   \Omega_{nc}(2n-1) \  , \]
 where the invariant forms on the right-hand side are interpreted as invariant differential forms on $L \SPD{2n}$.
 By rescaling the star operator, we may assume that
 \[  \star \alpha   \wedge  \alpha   = \beta^{5} \wedge \ldots \wedge \beta^{4n-3} \wedge \pff^{2n} \]
 for any $\alpha \   \in \ \Omega_c(2n-1) \wedge \pff^{2n}$ which is  a simple monomial in
 the generators.
 \begin{example} For $n=4$, the Hodge star operator interchanges the forms on $L \SPD{8}$:
 \begin{align*}
 \beta^{13}   \wedge \pff^{8}                           &\ \longleftrightarrow\ \,\beta^5\wedge\beta^9 &
 \beta^5\wedge\beta^{13}   \wedge \pff^{8}                &\ \longleftrightarrow\ -\beta^9  & \\
 \beta^9\wedge\beta^{13}    \wedge \pff^{8}               &\ \longleftrightarrow\ \,\beta^5  &
 \beta^5\wedge\beta^9\wedge\beta^{13} \wedge \pff^{8}    &\ \longleftrightarrow\ \,1 &
 \end{align*}
 where the forms $\alpha$ on the left-hand columns are orientation forms (satisfy the transformation law \eqref{etatransformationlaw}), and the forms $\star\alpha$ on the right are $\GL_{2n}(\R)$-invariant.
 \end{example}

\subsection{Vanishing of the Euler class in cohomology}

Here we explain how the invariant Pfaffian can be interpreted as the Euler form of a vector bundle in the context of Chern-Weil theory, see e.g.\ \cite[\S 25]{Tu:DiffGeom}. From this we deduce the vanishing of the class of the Pfaffian form in cohomology.

The space $\SPD{2n}$ parametrizes Euclidean metrics on $\R^{2n}$: let $E\defas \SPD{2n}\times \R^{2n}$ denote the trivial rank $2n$ vector bundle on $\SPD{2n}$, equipped with the metric $(\cdot,\cdot)$ defined by
\begin{equation*}
    (v,w)= v^\Transpose X w
\end{equation*}
for vectors $v,w\in \R^{2n}$ in the fiber over $X\in\SPD{2n}$. We furthermore endow $E$ with a right action of $\GL_{2n}(\Z)$: lifting the action on the base $\SPD{2n}$, we define
\begin{equation*}
     (X,v) \mapsto  \big(g^\Transpose X g , g^{-1} v\big).
\end{equation*}
The metric $(\cdot,\cdot)$ is invariant under this action, and so $E$ descends to a Euclidean orbibundle $E/\GL_{2n}(\Z)\rightarrow \SPD{2n}/\GL_{2n}(\Z)$. Since the action also preserves the lattice $\SPD{2n}\times\Z^{2n}$ inside $E$, following \cite[\S I.(b)]{BismutCheeger:TransgressedEuler} there is a distinguished metric connection $\nabla$ on $E$.\footnote{This connection, which we denote $\nabla$, is denoted $^0\nabla^E$ in \cite{BismutCheeger:TransgressedEuler}.}
In the global frame $\fr{e}=(\fr{e}_1,\ldots,\fr{e}_{2n})$ of $E$ given by the unit (column) vectors $\fr{e}_i$ of $\Z^{2n}$, this connection reads $\nabla\fr{e}=\fr{e} \theta$, that is $\nabla \fr{e}_i=\sum_{j=1}^{2n} \theta_{ji}\fr{e}_j$, for the matrix of one-forms
\begin{equation*}
    \theta = \frac{1}{2} X^{-1}\td X.
\end{equation*}
\begin{remark}This connection is the average $\nabla=(^0\nabla+{\iota^*} ^0\nabla^\vee)/2$ of the \emph{flat} connection $^0\nabla \fr{e}_i=0$ on $E$, and the pullback of the flat connection $^0\nabla^\vee (\fr{e}_i^\vee) = 0$ on the dual bundle $E^\vee$, under the isomorphism $\iota\colon E\cong E^\vee$ furnished by the metric $(\cdot,\cdot)$. In terms of the frame $\fr{e}^\vee$ of $E^\vee$ dual to $\fr{e}$, i.e.\ $\fr{e}_i^\vee(\fr{e}_j)=\delta_{ij}$, this isomorphism is $\iota(\fr{e}_i)=\sum_{j=1}^{2n} X_{ij} \fr{e}_j^\vee$.
It follows that $\nabla$ is compatible with the metric $(\cdot,\cdot)$ and with the $\GL_{2n}(\Z)$ action; hence it descends to a metric connection on the orbibundle.
\end{remark}

To compute the Euler form of the connection $\nabla$, we change from $\fr{e}$ to an orthonormal frame $\fr{e}'$ of $E$.
Recall that every matrix $X\in\SPD{2n}$ has a unique positive definite square root $\sqrt{X}\in\SPD{2n}$ such that $X=(\sqrt{X})^2$.
Setting $\fr{e}'= \fr{e}(\sqrt{X})^{-1}$ gives an orthonormal frame, $(\fr{e}'_i,\fr{e}'_j)=\delta_{ij}$. In this frame, the connection matrix $\theta'$ such that $\nabla\fr{e}'=\fr{e}'\theta'$ reads
\begin{equation*}
    \theta' =
    \frac{1}{2} \left(\sqrt{X}^{-1}\cdot \td \sqrt{X} - \td \sqrt{X} \cdot \sqrt{X}^{-1} \right)
    .
\end{equation*}
This matrix is skew-symmetric (confirming that $\nabla$ is compatible with the metric on $E$), and so is the associated curvature matrix
\begin{equation*}
    \Theta'
    = \td \theta'+\theta'\wedge\theta'
    = -\frac{1}{4} \sqrt{X}^{-1}\cdot \td X \cdot X^{-1} \cdot \td X \cdot \sqrt{X}^{-1}
    .
\end{equation*}
Comparing with \cref{def:pff} and using  property (iii) from \cref{lem:pfprops}, we conclude that the Pfaffian form is indeed the Euler form of the bundle $E$ with the connection $\nabla$:
\begin{equation*}
    \pff^{2n} = (-4)^n \Pf(\Theta').
\end{equation*}
By the general theory of characteristic classes, this gives an alternative proof of the closedness $\td \pff^{2n}=0$. In cohomology, the differential form $\Pf(\Theta')/(2\pi)^n$ represents the topological Euler class of the vector bundle $E$, see \cite{Bell:GaussBonnetVB}. This class is known to vanish rationally \cite{Sullivan:EulerSLZ}, and an explicit primitive is constructed in \cite[Theorem~1.27]{BismutCheeger:TransgressedEuler}.\footnote{In terms of the differential form denoted $\gamma(0)$ in \cite[Theorem~1.27]{BismutCheeger:TransgressedEuler}, we have $\pff^{2n}=(-8\pi)^n \td\gamma(0)$.}
\begin{example}
    For $n=2$, the link $L\SPD{2}\cong\HC=\{\tau\in\C\colon \tau_2>0\}$ is isomorphic to the upper half-plane with coordinate $\tau=\tau_1+\iu\tau_2$, via the $\SL_2(\Z)$-equivariant map\footnote{Here we let $P=\big(\begin{smallmatrix} A & B \\ C & D\end{smallmatrix}\big)\in\SL_2(\Z)$ act on $\SPD{2}$ and $\HC$ from the left via $X\mapsto P X P^\Transpose$ and $\tau\mapsto \frac{A\tau+B}{C\tau+D}$.}
    \begin{equation*}
    \tau\colon \SPD{2}\longrightarrow \HC,\qquad
    X=\begin{pmatrix}
        a & b \\
        b & c \\
    \end{pmatrix}
    \mapsto
    \frac{b+\iu\sqrt{ac-b^2}}{c}
    \ .
\end{equation*}
    In this coordinate, the Pfaffian form (\cref{ex:pf_2}) is the $\SL_2(\Z)$-invariant volume form 
    $\pff^2_X=
    -2(\td\tau_1\wedge\td\tau_2)/\tau_2^2
    $.
    An explicit $\SL_{2}(\Z)$-invariant primitive is given by the real part of the almost holomorphic modular form $E_2^*(\tau)=E_2(\tau)-3/(\pi\tau_2)$, where $E_2(\tau)=1-24\sum_{k,m=1}^\infty ke^{2\ipi km\tau}$, see for example \cite[\S 11.2, p.~355]{BergeronCharalloisGarcia:EulerEisenstein}:
\begin{equation*}
    \pff_X^2=\td \frac{2\pi}{3}\Re\Big( E_2^*(\tau)\td \tau \Big).
\end{equation*}
\end{example}
\begin{corollary}
   For all $n\geq 1$,  the cohomology class
   \begin{equation*}
   [\beta\wedge \pff^{2n}] \in H^{\bullet} (L\SPD{2n}/\GL_{2n}(\Z);\ob)
   \end{equation*}
   vanishes for any  form $\beta \in \Omega(2n-1)$.
\end{corollary}
\begin{proof}
The case $\beta=1$ implies the general case, since $[\beta\wedge \pff^{2n}]=[\beta] \wedge [\pff^{2n}]$. The fact that $\pff^{2n}$ is trivial in cohomology follows from \cite{BismutCheeger:TransgressedEuler,Sullivan:EulerSLZ}.
\end{proof}

In particular, the cohomology classes of the odd graph complex constructed in this paper (\cref{sec:cohom-classes}) are \emph{not} pulled back from the locally symmetric space $\SPD{2n}/\GL_{2n}(\Z).$

\subsection{Proof of theorem \ref{intro: thmcohomology}}
The proof of the theorem requires a number of technical steps. We first define the notion of orientation forms on the moduli space of tropical abelian varieties, and its compactification (which is homeomorphic to the Borel-Serre compactification for $\GL_n(\Z)$).   We then show that the Pfaffian forms can be extended to the boundary. The final step involves a computation of the cup product in relative cohomology. The key point is that since the volume form $\eta^{2n}$ is non-zero in cohomology, so too are its factors.

\subsubsection{The moduli space of tropical abelian varieties}

For background, see \cite{BBCMMW,BrBord}.
Every positive definite quadratic form $Q\in\SPD{g}$ has a finite set $M(Q)\subset\Z^g\setminus\{0\}$ of \emph{minimal vectors}, i.e.\ vectors $v$ which minimize $v^\Transpose Q v$ over all non-zero integer vectors.\footnote{For example, if $Q_{aa}=1$ and $Q_{ab}=-1/g$ for $a\neq b$, then $M(Q)=\{\pm e_1,\ldots,\pm e_g,\pm (e_1+\ldots+e_g)\}$ for $e_i\in\Z^g$ the unit vectors. The interior of the cone $\widehat{\sigma}_Q$ is the subspace of $\SPD{g}$ considered in \cref{lem:pff(diag+const)}.} The matrices $vv^\Transpose$ generate a convex cone
\begin{equation*}
    \widehat{\sigma}_Q=\sum_{v\in M(Q)} \R_{\geq 0} \cdot v v^\Transpose
    \quad\subset\quad \SPD{g}^\rt
\end{equation*}
in the space $\SPD{g}^\rt\subset \R^{g(g+1)/2}$ of positive semi-definite symmetric matrices whose null space is defined over the rationals. We call $\sigma_Q=(\widehat{\sigma}_Q\setminus\{0\})/\R_{>0}$ the \emph{link} of this cone.

Let $\mathcal{D}_g^{\perf}$ be the category whose objects $[Q]$ are equivalence classes of positive definite quadratic forms $Q\in\SPD{g}$, where two forms $Q,Q'$ are equivalent if $M(Q)=M(Q')$. The morphisms in $\mathcal{D}_g^{\perf}$ are $[Q] \rightarrow  [h^\Transpose Q h]$ for any $h \in \GL_g(\Z)$, and whenever $M(Q')\subseteq M(Q)$ there is a `face map' morphism $[Q']\rightarrow [Q]$.
The link of the  moduli space of tropical abelian varieties is the topological space  defined by the  colimit:
\begin{equation*}
    | \LA_g^{\trop} |  = \varinjlim_{[Q]\in \mathcal{D}_g^{\perf} }  \sigma_Q \ .
\end{equation*}
Equivalently, $| \LA_g^{\trop} |$ may be defined as the quotient of $\SPD{g}^{\rt}$ by $\GL_g(\Z)$,  where  the former is equipped with the Satake  topology (the finest topology such that   $\sigma_Q\rightarrow  \SPD{g}^{\rt}$ is continuous for all $Q$).  It has a closed boundary $|\partial \LA_g^{\trop}| \subset  | \LA_g^{\trop}|$ with open complement which we denote by
\begin{equation*}
| \LA_g^{\circ, \trop} |
\defas
| \LA_g^{\trop} |   \ \setminus  \ | \partial  \LA_g^{\trop} |  
\cong L \SPD{g}  / \GL_g(\Z)\ , 
\end{equation*}
where $L\SPD{g}=\R^{\times}_{>0} \setminus \SPD{g}$.
The space $|\partial \LA_g^{\trop}| $ is the colimit of the links $\sigma_Q$  at `infinity',  i.e.,  contained in $L\SPD{g}^{\rt} \setminus L\SPD{g}$, or equivalently, in the vanishing  locus of  the  determinant.

Consider the functor
\[ \ob \colon  \mathcal{D}_g^{\perf}\To \mathrm{Vec}_{\R}  \]
which assigns to every object  $[Q]$ of $ \mathcal{D}_g^{\perf}$ the one-dimensional vector space $\R$, and  whose values on  morphisms  are: the identity on face maps, and  multiplication by $\det(h)^{g-1}$ for $h \in \GL_g(\Z)$. It defines a rank one local system  of  $\R$-vector spaces over $|\LA_g^{\trop}|$:
\[   \varinjlim_{[Q]\in \mathcal{D}_g^{\perf} }  \left( \sigma_Q \times \ob_Q\right) \]
whose restriction to the open $| \LA_g^{\circ, \trop}| $ is the orientation bundle on $L \SPD{g}  / \GL_g(\Z)$.

\subsubsection{Orientation forms on the bordification of \texorpdfstring{$|\LA_g^{\trop}|$}{|LA\_g\^trop|}}

The space  $|\LA_g^{\trop, \BB}|$ is obtained from $|\LA_g^{\trop}|$ as an iterated blow-up. A posteriori, it is homeomorphic to the quotient of the Borel-Serre compactification of $L \SPD{g}$ modulo the action of $\GL_g(\Z)$, but in order to study the asymptotic behaviour of the Pfaffian forms, we must use its description as an iterated  blow-up.

For this, embed the link of each  cone in a projective space:
\[ \sigma_Q   \ \subset  \ \Pro(\Quad(V))(\R)\]
where $V$ is a $\Q$-vector space of dimension $g$, $Q$ is a positive definite quadratic form on $V\otimes_{\Q} \R$, and $\Quad(V)$ is the $\Q$-vector space of quadratic forms on $V$. Following \cite{BrBord}, we  blow up  a finite set    $  \BB_Q$ of   nested chains of subspaces of the form $\Pro(\Quad(V/K)) \subset \Pro(\Quad(V))$, which are defined by the set of  quadratic forms with kernel  $K$ for a specific set of subspaces $K\subset V$ determined by $\sigma_Q$, in increasing order of dimension. We  obtain a scheme
\[ \pi_{\BB_Q}  \colon \Pro^{\BB_Q} \To   \Pro(\Quad(V))\ .  \]
One defines $\sigma^{\BB}_Q \subset   \Pro^{\BB_Q} (\R)$ to be the closure, in the analytic topology, of the inverse image of the interior of $\sigma_Q$ under the map $\pi_{\BB_Q}$. The blow-down $\pi_{\BB_Q}$ induces  a continuous map $\sigma^{\BB}_Q \rightarrow \sigma_Q$. The space $\sigma^{\BB}_Q$ is a manifold with corners which has the combinatorial structure of a polyhedron.

One  defines  a category $\mathcal{D}_g^{\perf,\BB}$   generated by the sets of faces  $F\sigma^{\BB}_Q$ of $\sigma^{\BB}_Q$, for all $[Q]$, whose morphisms are   inclusions of faces, and maps induced by  the action of $\GL_g(\Z)$.  Finally, the space $|\LA_g^{\trop,\BB}|$ is defined by
\[  |\LA_g^{\trop,\BB}| = \varinjlim_{ \mathcal{D}_g^{\perf,\BB} } F\sigma_Q^{\BB}    \]
It is the space obtained by gluing together all (quotients of) polyhedra $\sigma_{Q}^{\BB}$   along faces, modulo the action of $\GL_g(\Z)$.
It has  a closed subspace, $ |\partial \LA_g^{\trop,\BB}| \subset  |\LA_g^{\trop,\BB}|$ which is obtained by gluing all faces whose images under the blow-down $\pi_{\BB_Q}$ lie at infinity, i.e.,  are contained in $|\partial \LA_g^{\trop}|$.  It is proven in \cite{BrBord} that the open complement is again the locally symmetric space:
\[      |\LA_g^{\trop,\BB}| \ \setminus \   |\partial \LA_g^{\trop,\BB}| =   |\LA_g^{\circ, \trop}| =  L \SPD{g}/ \GL_g(\Z)\ . \]

Since each face  $F \sigma_Q^{\BB}$ of  $ \sigma_Q^{\BB}$ is contractible, we may define an orientation bundle on $|\LA_g^{\trop, \BB}|$ as before. In detail, consider the functor
\[ \ob \colon \mathcal{D}_g^{\perf, \BB} \To  \mathrm{Vec}_{\R}\]
which assigns to every object the vector space $\R$, and sends all inclusions of faces to the identity map, and morphisms induced by $h \in \GL_g(\Z)$ to multiplication by $\det(h)^{g-1}$.   It defines a local system on $ |\LA_g^{\trop,\BB}|$ given by:
\[   \varinjlim_{ \mathcal{D}_g^{\perf,\BB} }  \left( F\sigma^{\BB}_Q \times \ob_Q\right) \To   |\LA_g^{\trop,\BB}|  \ .  \]

\subsubsection{Differential forms with integrable square root singularities}

In order to set up a de Rham cohomology theory for  a space defined as a colimit  of polyhedral cells $\tau$ \cite[\S3.3]{BrBord} (where $\tau$ ranges over the set of $F \sigma_Q^{\BB}$), we may use any  complex of differential forms $\AdR(\tau)$ on $\tau$, equipped with natural restriction maps $i^*\colon \AdR(\tau) \rightarrow \AdR(\partial\tau) $ where   $i\colon \partial \tau \subset  \tau$ denotes the inclusion, with the following properties:
\begin{enumerate}
\item   $ H^n(\AdR(\tau))\cong H^n(\tau;\R)$ and $ H^n(\AdR(\partial\tau))\cong H^n(\partial\tau;\R)$ for all $n$. \label{AdR:i}
    \item The  map $i^*\colon \AdR(\tau) \rightarrow \AdR(\partial\tau)  $ is surjective \cite[\S7]{Sullivan:Inf}. \label{AdR:ii}
\end{enumerate}
One example is given by the  complex of differential forms \cite{BrBord}  which are smooth in a neighbourhood of $\tau$. However, Pfaffian  forms do not have this property because of the appearance of square root singularities on the boundary, see \cref{ex:pff2-asy}. These singularities are mild, in that the square roots appear only as $\sqrt{z}$ or $(\td z)/\sqrt{z}=2\td(\sqrt{z})$. These forms are integrable at $z=0$, and in fact they become smooth after changing variables from $z$ to $x=\sqrt{z}$. We formalize  this larger complex of forms as follows.

We first define a complex of sheaves $\mathfrak{A}^{\bullet}$ on $\tau$, for  every such polyhedron  $\tau$, whose sections are differential forms on $\tau$ which may have integrable square root singularities on the  boundary $\partial \tau$.  In  detail, every such $\tau$ is a compact real manifold with corners with the property that it is embedded in the real points $P(\R)$ of a smooth scheme $P$ over $\R$ which we may assume to have the same dimension as $\tau$ (more precisely, $P$ is an iterated blow up  of a projective space) and the Zariski closure of its facets is a normal crossing divisor. In particular, every point $p\in \tau$ admits a neighbourhood $U_p \subset P(\R)$ with local coordinates $z_1,\ldots, z_n$ such that for some $0\leq r\leq n$,
\begin{equation}\label{eq:corner-chart}\begin{aligned}
\tau \cap U_p &= \{(z_1,\ldots, z_n) \in U_p\colon z_1,\ldots, z_r \geq 0\}
\quad  \text{and} \\
\partial \tau \cap U_p &= \tau \cap   U_p \cap \left\{z_1\cdots z_r=0\right\}  \ . 
\end{aligned}\end{equation}
Define $\mathfrak{A}^m(\tau \cap U_p)$ to be the $\R$ vector space of  forms $\omega$ of degree $m$  which are  smooth in $\sqrt{z_1},\ldots, \sqrt{z_r}, z_{r+1},\ldots, z_n$ in an open neighbourhood of $\tau \cap U_p$ inside $U_p$.   Equivalently, consider the ramified map of degree $2^r$ given by $\pi_r\colon \R^{n}\rightarrow \R^n$
which, in standard coordinates, maps $(x_1,\ldots, x_n)\mapsto (x_1^2,\ldots, x_r^2,x_{r+1},\ldots, x_n)$. Let $\pi^{-1}_{r,+}(\tau \cap U_p)\subset \R_{\geq 0}^n$ be the unique preimage of $\tau \cap U_p$ in the region $x_1,\ldots, x_r\geq 0$ (i.e., given by non-negative square roots of $z_1,\ldots, z_r$). It is homeomorphic to $\tau \cap U_p$ via the continuous map $\pi_r$. Then 
$\mathfrak{A}^m(U_p\cap \tau) = \Omega^m(\pi^{-1}_{r,+}(\tau \cap U_p))$ is   the space of $m$-forms $\omega$ such that $\pi_r^* \omega$ is smooth on  an open neighbourhood of $\pi^{-1}_{r,+}(\tau \cap U_p)$. This is nothing other than the direct image under $\pi_r$ of the sheaf of  smooth forms on a certain open subset of  $\R^n$. 
Since,  away from $z=0$, the function $\sqrt{z}$ is smooth in $z$, the collection of $\mathfrak{A}^{\bullet}(\tau \cap U_p)$ defines a complex of sheaves on  $\tau$.

We denote the algebra of global sections to be $\AdR(\tau)=\Gamma(\tau, \mathfrak{A}^\bullet)$. The restriction of an element of $\AdR(\tau)$ to the interior of $\tau$ is a smooth form in the usual sense. For every face $i\colon F\tau\hookrightarrow \tau$, the space $\AdR(F\tau)$ is defined in the same way, and the restriction of differential forms (in the coordinates $x_i=\sqrt{z_i}$) induces natural maps $i^*\colon \AdR(\tau) \rightarrow \AdR(F\tau)$.
  
More precisely, for $1\leq\mu\leq r$, the restriction $i_\mu^*(\omega)\in \AdR(\partial_\mu\tau)$ of a form $\omega\in\AdR(\tau)$ to a facet $\partial_\mu\tau\cap U_p=\tau\cap U_p\cap\{z_\mu=0\}$ of $\partial\tau$ is defined locally on $U_p$ by setting $\td\sqrt{z_\mu}=\td x_\mu$ and $\sqrt{z_\mu}=x_\mu$ in $\omega$ to zero. In other words, we have a commutative diagram
\begin{equation*}\xymatrix{
\Omega^\bullet(\pi_{r,+}^{-1}(\tau\cap U_p)) \ar@{=}[d] \ar[r]^{\iota_\mu^*} & \Omega^\bullet(\pi_{r,+}^{-1}(\partial_\mu\tau\cap U_p))\ar@{=}[d] \\
\AdR(\tau\cap U_p) \ar[r]^{i_\mu^*} & \AdR(\partial_\mu\tau\cap U_p)
}\end{equation*}
where $\iota_\mu^*$ in the top row is the ordinary restriction of smooth forms to the subspace $\iota\colon \{x_\mu=0\}=\pi_{r,+}^{-1}(\partial_\mu\tau\cap U_p)\subset \pi_{r,+}^{-1}(\tau\cap U_p)$. The  induced map $i^*$ in the coordinates $z_i=x_i^2$ has the effect of formally setting $\sqrt{z_{\mu}}=0$ and $(\td z_{\mu})/\sqrt{z_{\mu}}=0$. Somewhat abusively, we will call this map $i^*$ also ``restriction'' (to the facet $\{z_\mu=0\}$).

Restrictions to faces $\{z_{\mu_1}=\cdots=z_{\mu_c}=0\}$ of higher codimension are defined analogously by setting $\td\sqrt{z_{\mu_k}}$ and $\sqrt{z_{\mu_k}}$ to zero for all $1\leq k\leq c$. These restrictions are functorial: given two face inclusions $i\colon \sigma\subset \partial\tau$ and $j\colon \nu\subset\partial \sigma$, the pullback under the composed inclusion $ij\colon \nu\subset\partial\tau$ of $\nu$ as a face of $\tau$ is $(ij)^*=j^*i^*\colon \AdR(\tau)\rightarrow \AdR(\nu)$.

We also define algebras $\AdR(\partial \tau)$  to be the projective limit over the $\AdR(F \tau) $ for all strict faces $F\tau$ of $\tau$.

\begin{lemma} The complexes $\AdR(\tau)$ and $\AdR(\partial\tau)$ satisfy \ref{AdR:i} and \ref{AdR:ii} above.
\end{lemma}

\begin{proof} Property \ref{AdR:i} follows since the complex  of sheaves $\mathfrak{A}^{\bullet}$  is acyclic by the usual `Poincar\'e lemma' (due to Volterra) and therefore defines a resolution of the constant sheaf $\underline{\R}$ on $\tau$.
The statement for $\partial \tau$ is  not actually  necessary, but follows by induction on the dimension of cells and repeated application of \cite[Theorem 3.3]{BrBord} to $\partial \tau$, which is the inductive limit over all faces of $\tau$.

For property~\ref{AdR:ii} we show that a section of $\AdR[m](\partial \tau)$ on the boundary $\partial \tau$ may be extended to $\tau$. Using a partition of unity subordinate to a cover by charts $U_p$ as above (with $p\in\partial\tau$), it suffices to show the extendability of forms $\omega\in\AdR(\partial\tau)$ that have compact support contained in some $U_p$. So let $z_1,\ldots,z_n$ be coordinates as in \eqref{eq:corner-chart}. For any non-empty subset $K\subseteq\{1,\ldots,r\}$, let $\pi_K\colon\tau\cap U_p\rightarrow \partial\tau\cap U_p$ denote the orthogonal projection of $z$ onto the face $\partial_K\tau=\tau\bigcap_{\mu\in K}\{z_\mu=0\}$. Recall that $\omega$ is by definition a collection of forms $\omega_K\in\AdR(\partial_K\tau)$ such that $\omega_{K\cup\{\mu\}}=i^*_\mu \omega_K$ for $\mu\notin K$, where we abuse notation $i_\mu^*$ for the ``restriction'' to $\{z_\mu=0\}$ irrespective of the domain. Let $\psi\colon\R\rightarrow \R$ be a smooth function with compact support and equal to $1$ in a neighbourhood of zero. Now define a form on $\tau$ with compact support by
\begin{equation*}
    \widetilde{\omega} \defas \sum_{\substack{K\subseteq\{1,\ldots,r\} \\ K\neq\emptyset}} (-1)^{|K|-1} (\pi_K^* \omega_K) \prod_{\mu\in K}\psi(z_\mu)
    \quad\in\quad \AdR(\tau).
\end{equation*}
Since $i_\mu^* (\psi(z_\mu)\pi_K^*\omega_K)=\pi_{K\setminus\{\mu\}}^*\omega_K$ for $\mu\in K$ and $i^*_\mu(\pi^*_K\omega_K)=\pi^*_K \omega_{K\cup\{\mu\}}$ for $\mu\notin K$, all contributions to $i_\mu^*\widetilde{\omega}$ from $K$ not containing $\mu$ cancel with the contribution from the summand $K\cup\{\mu\}$. The only term left over is $K=\{\mu\}$, proving $i_\mu^*\widetilde{\omega}=\omega_{\{\mu\}}$ for all $\mu\leq r$. It follows that $i^* \widetilde{\omega}=\omega$, since $\omega_K=i^*_{K\setminus\{\mu\}}\omega_{\{\mu\}}=i^*_{K\setminus\{\mu\}}i^*_\mu\widetilde{\omega}=i^*_K\widetilde{\omega}$ for any $\mu\in K$.
\end{proof}

For every face $F \sigma_Q^{\BB}$ of $\sigma_Q^{\BB}$,  we define $\AdR(F \sigma_Q^{\BB};\ob)$ to be the complex of differential forms with integrable square root singularities as considered above. They are to be interpreted as taking values in the trivial bundle $\ob_Q=\R$ on $F \sigma_Q^{\BB}$.

\begin{definition}
Define the complex of  smooth orientation forms  on  $  | \LA_g^{\trop,\BB} |$  to be  the projective limit:
\begin{equation*}
\AdR\left(   | \LA_g^{\trop,\BB} | ; \ob      \right)= \varprojlim_{\mathcal{D}_g^{\perf,\BB} }  \AdR\left(     F\sigma^{\BB}_Q, \ob   \right) \ . 
\end{equation*}
By taking the limit over the faces contained in $|\partial \LA_g^{\trop,\BB}|$, we may similarly define    smooth orientation forms on the boundary $  |\partial \LA_g^{\trop,\BB} |$. The relative de Rham complex  $\AdR( | \LA_g^{\trop,\BB} | ,| \partial\LA_g^{\trop,\BB} | ; \ob )$  is defined to be the mapping cone of the restriction map  $  \AdR( | \LA_g^{\trop,\BB} | ; \ob )\rightarrow  \AdR(| \partial\LA_g^{\trop,\BB} | ; \ob )$.
\end{definition}

An element of  $\AdR (   | \LA_g^{\trop,\BB} | ; \ob  )$ is a compatible collection of differential forms on every face of every polyhedron $\sigma_Q^{\BB}$, allowing integrable square root singularities along its boundary. The compatibility means that they glue together consistently with respect to inclusions of faces and transform with the correct sign rule with respect to elements of $\GL_g(\Z)$.

One may show,  as in \cite[\S3.4-3.5]{BrBord},  that the complex $ \AdR(|\LA_g^{\trop,\BB} | ; \ob )$ computes the cohomology of  $ | \LA_g^{\trop,\BB} | $ with coefficients in $\ob$, and we deduce that
\begin{equation*}
 H^n_{\dR, c}    ( L \SPD{g}  / \GL_g(\Z); \ob) \cong   H^n(  \AdR( | \LA_g^{\trop,\BB} | ,| \partial\LA_g^{\trop,\BB} | ; \ob ))\ 
\end{equation*}
by interpreting relative (twisted) cohomology as compactly supported cohomology.
Closed elements in $ \AdR(|\LA_g^{\trop,\BB} | ; \ob )$ pair with twisted cellular chains on $|\LA_g^{\trop,\BB} |$
in the usual manner via integration: the key point being that the square root singularity $(\td z)/\sqrt{z}$ is integrable for $z\in [0,1]$.\footnote{In particular, after pulling back along $\pi_r$, the forms are smooth in the coordinates $x$ and the usual Stokes' theorem for smooth forms applies. In the coordinates $z$ on $\tau$ it takes the form $\int_\sigma \td \omega=\int_{\partial \sigma} i^* \omega$ where $i^*\colon \AdR(\tau)\rightarrow \AdR(\partial \tau)$ is the ``restriction'' which locally sets $\sqrt{z_i}$ and $(\td z_i)/\sqrt{z_i}$ to zero for $i\leq r$, and $\sigma$ is any singular chain on $\tau$ with boundary contained in $\partial\tau$.}

\subsubsection{Extending the Pfaffian invariant forms to the boundary}

The Pfaffian invariant forms $\pff^{2n}$  define  closed orientation forms on $L \SPD{2n}/\GL_{2n}(\Z)$.
We  first show that they extend  to the compactification.

\begin{proposition}  \label{prop: gammaextendstoboundaryvanishing} Let $n\geq 1$. The collection $\{\pi^*_{\BB_Q} \pff^{2n}\}$ of the pullbacks of the Pfaffian invariant forms defines a closed orientation form
\begin{equation} \label{gammaextends}
\widetilde{\pff}^{2n}  \ \in \    \AdR\left(   | \LA_{2n}^{\trop,\BB} | ; \ob      \right) .
\end{equation}
Furthermore, the restriction\footnote{namely, applying the  restriction map  $i^*$ to facets of each polyhedron in the sense defined above}  of the form   $\widetilde{\pff}^2$ to the boundary $|\partial \LA_{2}^{\trop, \BB}|$, or of
\[ \widetilde{\pff}^{2n} \wedge \widetilde{\beta}^{4n-3} \ , \ \text{for}\ n>1 \]
 to the boundary $|\partial \LA_{2n}^{\trop, \BB}|$, vanishes.
\end{proposition}
\begin{proof}
 The proof follows \cite[Theorem 13.10]{BrBord}. By an inductive argument, it will  suffice in the first instance to consider a single blow-up
\[ \pi \colon \Pro \To \Pro(\Quad(V))\]
along a subspace $\Pro(\Quad(V/K))$, where $V$ is a vector space of dimension $2n$ and $K\subset V$ a non-trivial subspace. The exceptional divisor  $\mathcal{E}$ is canonically isomorphic to $\Pro(\Quad(V/K)) \times \Pro( \Quad(V)/\Quad(V/K))$.
Following \emph{loc. cit.} Proposition 13.7,  we verify that the form $\pi^* \pff^{2n} $ has no poles along the exceptional divisor. To see this, note that  since the orientation bundle is trivial on the contractible region $\det(M)>0$ for $M \in   \Pro(\Quad(V))(\R)$, we may treat $\pff^{2n}$ as an ordinary differential form for the purposes of the following argument.
Now consider the local coordinates for $\pi$  given in  \emph{loc. cit.}, Proposition 9.10, relative to a choice of splitting
$V\cong  K \oplus C$.
The map $\pi^*$   acts on matrix entries of  matrices in block matrix form as follows:
\[ \pi^* \begin{pmatrix} M_1 & M_{12} \\ M^\Transpose_{12} & M_{2} \end{pmatrix} =  \begin{pmatrix} z M_1 &  z M_{12} \\ z M^\Transpose_{12} & M_{2} \end{pmatrix} \ \]
where $M_i$ is an $n_i \times n_i$ square matrix for $i=1,2$, and $n_1+n_2=2n.$
 The exceptional locus  $\mathcal{E}$ is given by  $z=0$.  Taking $z>0$, and invoking \cref{lem:pff-asy}, we deduce that
\begin{equation*}
\pi^* \pff^{2n} =  \begin{cases}
\pff^{n_1}(M_1) \wedge \pff^{n_2} (M_2)  + \asyO(z) \ , &\quad \hbox{ if } \dim K = n_1 \hbox{ is even} \\
\sqrt{z} \wedge \alpha  + (\td \sqrt{z}) \wedge \beta, & \quad  \hbox{ if } \dim K = n_1 \hbox{ is odd}
\end{cases}
\end{equation*}
where  $\alpha$, $\beta$ are smooth. Thus, $\pi^*\pff^{2n}$ is either smooth on $z=0$ (for $n_1$ even), or becomes smooth on the cover $z=x^2$ (for $n_1$ odd). 

 Let  $X_{V/K} \subset \Pro(\Quad(V/K)) (\R)$ denote the locus consisting of quadratic forms with positive determinant, and let
 $X\big|_K \subset   \Pro( \Quad(V)/\Quad(V/K)) $ denote the subspace of quadratic forms whose restriction to $K$ is positive definite.
 The computation  above implies that  the restriction of $\pi^* \pff^{2n}$ to the set
 \[  X_{\mathcal{E}}\defas X\big|_K   \times  X_{V/K}      \quad  \subset         \quad     \Pro( \Quad(V)/\Quad(V/K))(\R)  \times  \Pro(\Quad(V/K)) (\R)  \cong    \mathcal{E}(\R)   \]  vanishes if   $\dim K$ is odd, and equals
 \[ \pi^*\pff^{2n}\big|_{  X_{\mathcal{E}} } =\pi_K^* (\pff^{n_1}) \wedge  \pff^{n_2}   \]
 if $\dim K= n_1$ is even, 
 where $\pi_K \colon X\big|_K\rightarrow \Pro(\Quad(K)) (\R)$ denotes restriction of quadratic forms to $K$ (on block matrices of the form described above, it sends $M\mapsto M_1$). This completes the description of the behaviour of $\pff^{2n}$ with respect to a single blow-up.

 Recall that $\pi_{\BB_Q}$ is a composition of blow-ups.
 By   \cite[Remark 5.3]{BrBord} it is enough to consider a sequence of blow ups of the type considered above in order to understand the behaviour of $\widetilde{\pff}^{2n}$ along exceptional divisors. By iterating,
  we deduce that the restriction of $\widetilde{\pff}^{2n} = \pi_{\BB_Q}^*
 \pff^{2n}$ to any face of $\sigma_Q^{\BB}$ either vanishes, or equals an exterior   product of a certain number of forms $\pff^{2k}$ for $k<n$. In particular, it is finite and $\widetilde{\pff}^{2n} $ extends  to the boundary of $\sigma_Q^{\BB}$. It defines a section of the trivial bundle $\sigma_Q^{\BB} \times \ob_Q = \sigma_Q^{\BB} \times \R $ over $\sigma_Q^{\BB}$ and similarly for all its faces.
   The functoriality of   $\pff^{2n}$ implies that the forms $\widetilde{\pff}^{2n} $ glue together consistently for all such $\sigma^{\BB}_{Q}$ and are compatible with both face maps, and morphisms induced by $\GL_{2n}(\Z)$ with the transformation rule \eqref{etatransformationlaw}.
    It follows from these calculations and   \cref{lem:pff-asy}  that $\widetilde{\pff}^{2n}$ defines an element of 
$ \AdR (   | \LA_g^{\trop,\BB} | ; \ob )$. 
   The claim \eqref{gammaextends} follows.

 For the last part, consider the  restriction of the form $\widetilde{\pff}^{2n} \times \widetilde{\beta}^{4n-3}$ to $|\partial \LA_{2n}^{\trop,\BB}|$. The boundary facets of $|\partial \LA_g^{\trop,\BB}|$ are isomorphic to  products of blow-ups $\sigma^{\BB}_1 \times \ldots \times \sigma_{r}^{\BB}$.  By iterating,  it suffices to consider the case  $r=2$. By computing a single blow-up of $V$ along $K$,  as above, the   restriction of this differential form satisfies
\[   \left( \widetilde{\pff}^{2n} \times \widetilde{\beta}^{4n-3} \right)\Big|_{\sigma^{\BB}_1 \times \sigma^{\BB}_2}  =    \left( \widetilde{\pff}^{2n_1} \wedge \widetilde{\beta}^{4n-3} \right)\Big|_{\sigma^{\BB}_1}  \wedge \widetilde{\pff}^{2n_2}\Big|_{\sigma^{\BB}_2}  +   \widetilde{\pff}^{2n_1}\Big|_{\sigma^{\BB}_1} \wedge     \left( \widetilde{\pff}^{2n_2} \wedge \widetilde{\beta}^{4n-3} \right) \Big|_{\sigma^{\BB}_2}
 \]
 if the dimension  $\dim K=2n_1$ is even, and vanishes if not (since if $\dim K$ is odd, so too is $\dim (V/K)$).      In the former case, $\widetilde{\beta}^{4n-3}$ already vanishes on matrices of rank   $ 2n_1, 2n_2 \leq n-2$ because it is of compact type \cite[Theorem 13.10]{BrBord}. Therefore  in all cases the restriction of $\widetilde{\pff}^{2n} \times \widetilde{\beta}^{4n-3}$ to faces vanishes.  The case   $\widetilde{\pff}^{2}$ is similar.  \end{proof}

\begin{corollary}  Let $n\geq 2$. The orientation forms $\pff^{2n}\wedge \beta^{4n-3}$ admit  compactly supported representatives on $L \SPD{2n} /\GL_{2n}(\Z)$ which we denote by
\begin{equation*}
\left(\pff^{2n} \wedge \beta^{4n- 3}\right)_c 
\end{equation*}
The same holds for $\pff^2$ on $L\SPD{2}/\GL_2(\Z)$.
\end{corollary}

\begin{proof}Use the fact that
$H_{\dR, c}^{\bullet} (  L\SPD{2n} /\GL_{2n}(\Z); \ob ) \cong   H_{\dR}^{\bullet}  ( | \LA_{2n}^{\trop, \BB} | ,  | \partial \LA_{2n}^{\trop, \BB} | ,\ob)$. By \cref{prop: gammaextendstoboundaryvanishing} the forms $\widetilde{\pff}^{2n} \wedge \widetilde{\beta}^{4n-3}$ define relative cohomology classes on $| \LA_{2n}^{\trop, \BB} | $ since they vanish along $| \partial \LA_{2n}^{\trop, \BB} |$.
\end{proof}

\subsubsection{Proof of theorem \ref{intro: thmcohomology}}
The strategy follows that of \cite{BrBord}.
The relative cohomology class of the volume  form $\eta^{2n}$ is non-zero in the group
\begin{equation*}
H^{d_{2n}-1}_{\dR} \Big( |\LA_{2n}^{\trop,\BB} |, |\partial\LA_{2n}^{\trop,\BB} | , \ob\Big)
\cong
H^{d_{2n}-1}_{c,\dR} \Big( |\LA_{2n}^{\circ, \trop} | ,\ob  \Big)
\end{equation*}
  (where $d_{2n}$ defined in \cref{thm:introSymembeds} is the dimension of $\SPD{2n}$),  since it pairs non-trivially with the relative fundamental class.
 Let $\omega \in \Omega_c(2n-1)$ be of compact type and let $\beta= \star (\pff^{2n} \wedge \omega)$. By \cref{prop: gammawedgestoeta}, the form $\beta$ is the image of an element of $\Omega(2n-1)$.  By definition of the Hodge star operator, we have
\begin{equation} \label{betagammaeta} \beta \wedge (\pff^{2n} \wedge \omega) = \lambda\, \eta^{2n}\end{equation}
where $\lambda=    \langle \pff^{2n} \wedge \omega,   \pff^{2n} \wedge \omega\rangle  $ is non-zero  by the last part of \cref{prop: gammawedgestoeta}. The identity \eqref{betagammaeta} remains true after pulling back to the blow-up $|\LA_{2n}^{\trop,\BB} |$ and extending to the boundary, so we may view the pullbacks as elements in:
\[  \widetilde{\beta} \in \AdR \left(|\LA_{2n}^{\trop,\BB} | \right) \quad , \quad   \widetilde{\pff}^{2n} \wedge \widetilde{\omega} \in  \AdR \left(|\LA_{2n}^{\trop,\BB} | ,  |\partial \LA_{2n}^{\trop,\BB} |   ,\ob\right)    \]
 where the fact that $\widetilde{\beta}$ extends to $|\LA_{2n}^{\trop,\BB} |$ was established in \cite{BrBord}, and the statement for $ \widetilde{\pff}^{2n} \wedge \widetilde{\omega}$ follows from \cref{prop: gammaextendstoboundaryvanishing}, as does the fact that
 \begin{equation*} 
 \widetilde{\eta}_{2n}  \in \AdR\left( |\LA_{2n}^{\trop,\BB} | ,  |\partial \LA_{2n}^{\trop,\BB} |   ,\ob \right) .
 \end{equation*}

The exterior product  of  forms induces a cup product in  relative   de Rham cohomology
\begin{equation*}
H^{m}_{\dR} \Big( |\LA_{2n}^{\trop,\BB} |\Big)    \otimes   H^{\ell}_{\dR} \Big( |\LA_{2n}^{\trop,\BB} |, |\partial\LA_{2n}^{\trop,\BB} | , \ob \Big) \rightarrow
H^{m+\ell}_{\dR} \Big( |\LA_{2n}^{\trop,\BB} |, |\partial\LA_{2n}^{\trop,\BB} | , \ob \Big)
\end{equation*}
Identity \eqref{betagammaeta} implies the relation on cohomology classes:
\begin{equation*}
     \big[\widetilde{\beta} \big] \wedge  \big[ (\pff^{2n} \wedge \omega)_c\big] = \lambda \big[(\eta^{2n})_c\big]  \ .
\end{equation*}
Since $\lambda [(\eta^{2n})_c]$ is non-zero, this implies that both
\begin{equation*}
\big[\widetilde{\beta}\big] \in   H^{m}_{\dR} ( |\LA_{2n}^{\trop,\BB} |)
\qquad \text{and} \qquad  
\big[(\pff^{2n} \wedge \omega)_c\big]  \in     H^{\ell}_{\dR} ( |\LA_{2n}^{\trop,\BB} |, |\partial\LA_{2n}^{\trop,\BB} | , \ob)
\end{equation*}
are non-zero. Identifying the latter group with $H_c^{\ell}(|\LA_{2n}^{\circ, \trop}|;\ob)$ completes the proof.

\begin{remark}
  This argument proves in passing that the natural map $\Omega^{\bullet}_{nc}(2n-1) \rightarrow  H^{\bullet} (\SPD{2n}/\GL_{2n}(\Z);\R) $ is injective,
which was also proven in \cite{BrBord}. Whilst the existence of the classes \eqref{intro:gammaformsinject} could be deduced from \emph{loc.\ cit.} using Poincar\'e duality, their description using Pfaffian invariant forms is new.
\end{remark}

\subsection{Primitivity  and proof of theorem \ref{thm:introSymembeds}}
Block direct sum of matrices \eqref{eqn: blocksumonpositivedefn} defines a map
\begin{equation*}
\bds_{m,n}\colon  \SPD{2m}/\GL_{2m}(\Z) \times \SPD{2n}/\GL_{2n}(\Z) \longrightarrow \SPD{2m+2n}/\GL_{2m+2n}(\Z)\ , 
\end{equation*}
which induces an isomorphism of orientation bundles  $\bds^*_{m,n}\,  \ob =  \ob \boxtimes \ob $
where $\boxtimes$ denotes the external tensor product. Since $\bds_{m,n}$ is proper\footnote{Block direct sum defines a continuous map $\mathcal{A}_{g}^{\trop}\times \mathcal{A}_h^{\trop} \rightarrow \mathcal{A}_{g+h}^{\trop}$, which  restricts to $\bds_{m,n}$ when $g=2m,h=2n$. Since the $\mathcal{A}^{\trop}_{r}$ are compact, the former is a proper map with compact fibers. The fiber of a point in the interior   $\SPD{g+h}/\GL_{g+h}(\Z)$ of $\mathcal{A}_{g+h}^{\trop}$ is contained in  $\SPD{g}/\GL_{g}(\Z) \times \SPD{h}/\GL_{h}(\Z) $ since $\det(X\oplus Y)\neq 0$ implies  $\det(X), \det(Y)\neq  0$. This implies that  $\bds_{m,n}$ has compact fibers. } 
 it  induces a morphism
\begin{multline*}
b_{m,n}^*\colon    H^{\bullet}_c\left( \SPD{2m+2n}/\GL_{2m+2n}(\Z) ;\ob\right) \To   \\ H^{\bullet}_c\left( \SPD{2m}/\GL_{2m}(\Z);\ob\right)  \otimes_{\R}    H^{\bullet}_c\left( \SPD{2n}/\GL_{2n}(\Z);\ob \right)  \ .
\end{multline*}
\begin{proposition}\label{prop:vanish}
For all $m,n\geq 1$, and $\omega \in \Omega_c(2m+2n-1)$ of compact type the class
$b^*_{m,n}  [(\pff^{2m+2n} \wedge \omega)_c]$  vanishes in $ H^{\bullet}_c\left( \SPD{2m}/\GL_{2m}(\Z);\ob\right)  \otimes_{\R}    H^{\bullet}_c\left( \SPD{2n}/\GL_{2n}(\Z);\ob \right) $.
\end{proposition}

\begin{proof} It suffices to consider the case  $\omega = \beta^{4m+4n-3}$.
   The statement follows from the fact that, for any $2m\times 2m$ matrix $X$ and $2n\times 2n$ matrix $Y$, the  form
    \[ \left(\pff^{2m+2n} \wedge  \beta^{4m+4n-3}\right)_{X\oplus Y} =  \left( \pff^{2m}\wedge \beta^{4m+4n-3}\right)_X \wedge \pff^{2n}_Y  +  \pff^{2m}_X \wedge  \left( \pff^{2n}\wedge \beta^{4m+4n-3}\right)_Y   \]
    vanishes since $\beta^{4m+4n-3}_X=\beta^{4m+4n-3}_Y =0$, see \cite[Proposition~4.5]{invariant}.
\end{proof}

By \cite{AMP}, the  bigraded vector space
\[ H_{\det, \R} = \bigoplus_{n\geq 0}     H^{\bullet}_c\left( \SPD{2n}/\GL_{2n}(\Z) ;\ob\right)\]
admits a structure of   cocommutative bigraded Hopf algebra over $\R$.
\begin{corollary} \label{cor: primitives} For all $\omega \in \Omega_c(2n-1)$, the class
$ [(\pff^{2n} \wedge \omega)_c]  \in H_{\det,\R} $
    is primitive.
\end{corollary}
\begin{proof}
 By \cite[Definition~2.1]{AMP} the multiplication on $H^{\det}$ is induced by block direct sum of matrices, and therefore the $(m,n-m)$th graded components of the coproduct
$\Delta\colon H_{\det,\R} \rightarrow H_{\det,\R} \otimes H_{\det,\R}$ dual to it are given by $b_{m,n-m}^*$. By \cref{prop:vanish} all components with $0<m<n$ vanish, hence only the trivial pieces $b_{0,n}$ and $b_{n,0}$ remain such that $\Delta [(\pff^{2n} \wedge \omega)_c] = 1\otimes  [(\pff^{2n} \wedge \omega)_c]+ [(\pff^{2n} \wedge \omega)_c]\otimes 1$.
\end{proof}
We deduce the existence of a  large amount of new unstable cohomology for $\GL_n(\Z)$ for even $n$. Indeed, the Milnor-Moore theorem  states that for a connected, cocommutative Hopf algebra $H$ over a field of characteristic zero which is finite-dimensional in each degree, there is a canonical isomorphism  $\UE(\Prim(H))\overset{\sim}{\rightarrow} H$ with the universal enveloping algebra of the Lie algebra of primitive elements. The Poincar\'e-Birkhoff-Witt theorem states that $\SymA(\Prim(H))\cong \Gr \UE(\Prim(H))$ is an isomorphism of algebras, where $\Gr$ is the grading associated to the  filtration by length of tensors. By choosing a splitting of this grading, it follows that there exists an isomorphism of graded vector spaces $\SymA(\Prim(H))\cong \UE(\Prim(H))$. There is a choice of splitting for which it is a morphism of coalgebras. It follows from \cref{intro: thmcohomology} and
\cref{cor: primitives} that
\begin{equation*}
\SymA \left( \Omega_{c}[\pff] \right) \otimes \R  \To   H_{\det,\R}
\end{equation*}
is injective. The rest of \cref{thm:introSymembeds} follows by duality.

\begin{remark} \cite{AMP} have previously shown for dimension reasons that  the volume form classes $\eta^{2n} = \beta^5 \wedge \ldots \wedge \beta^{4n-3} \wedge \pff^{2n}$ are primitive. They are denoted by $t_{2n}$ in \emph{loc. cit.} since they are dual to the trivial classes in $H^0(\GL_{2n}(\Z);\R)$.
\end{remark}

\section{The odd commutative graph complex}\label{sec:GC}

The odd graph complex as introduced by Kontsevich \cite[\S 2]{kontsympl} and studied further e.g.\ in \cite{onatheorem} is a chain complex. We denote it by $\hGC_3$. Its dual cochain complex is a differential graded Lie algebra, which we denote as $\cGC_3$, following \cite{Willwacher:GCgrt}. For further details and in particular the underlying operad of graphs, see also \cite[\S 4]{Morand:KontsevichGCuniv}.

A graph $G$ in this paper consists of finite sets of vertices $V(G)$, undirected edges $E(G)$, and half-edges $H(G)$. Each edge $e$ consists of two distinct half-edges $H(e)=\{e',e''\}$, providing a partition $H(G)=\bigsqcup_e H(e)$. Each half-edge is incident to exactly one vertex, amounting to a map $H(G)\rightarrow V(G)$.
Thus, multiple edges may connect the same pair of vertices, and edges with both ends at the same vertex (self-loops) are allowed (\cref{fig:gc-half-edges}).
\begin{figure}\centering
	$G=\Graph{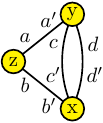}$ \quad
	\begin{tabular}{c@{\ }c@{\ }c}
		$H(G)$ & $\rightarrow$ & $V(G)$ \\
	\midrule
		$b',c',d'$ & $\mapsto$ & $x$ \\
		$a',c,d$ & $\mapsto$ & $y$ \\
		$a,b$ & $\mapsto$ & $z$ \\
	\bottomrule
	\end{tabular}
	\qquad
	$G/c=\Graph{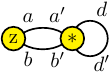}$
	\quad
	\begin{tabular}{c@{\ }c@{\ }c}
		$H(G/c)$ & $\rightarrow$ & $V(G/c)$ \\
	\midrule
		$a',b',d,d'$ & $\mapsto$ & $\ast$ \\
		$a,b$ & $\mapsto$ & $z$ \\
	\bottomrule
	\end{tabular}
	\caption{A graph with its half-edges, incidences, and its quotient by edge $c$.}%
	\label{fig:gc-half-edges}%
\end{figure}

\begin{definition}\label{def:GCor}
An \textbf{orientation} of a graph $G$ is one of the two invertible elements in
\begin{equation*}
    \det \Z^{V(G)}\otimes \bigotimes_{e\in E(G)} \det \Z^{H(e)}\cong\Z.
\end{equation*}
An \textbf{oriented graph} is a pair $(G,o)$ of a combinatorial graph and an orientation $o$ of $G$.
\end{definition}
To pick an orientation $o=v_1\wedge\ldots\wedge v_n\otimes \bigotimes_e o_e$, we can specify a total order $v_1<\ldots<v_n$ of the vertices and choose a direction $o_e=e'\wedge e''$ or $o_e=e''\wedge e'$ for each edge. In figures, we indicate orientations by labelling vertices with integers (implying the increasing order) and directing edges $v\rightarrow w$ to mean that in $o_e$ the half-edge at the tail ($v$) comes first and the head ($w$) last. Swapping two vertices, or reversing an edge, changes the sign of $o$.

An isomorphism $\varphi\colon G\cong G'$ is a pair of bijections $H(G)\leftrightarrow H(G')$ and $V(G)\leftrightarrow V(G')$ that preserves the edge partitions and incidence relations. It induces a bijection of edges $E(G)\leftrightarrow E(G')$ and identifies each orientation $o$ of $G$ with an orientation $\varphi_*(o)$ of $G'$.
\begin{definition}\label{def:hGC}
As a vector space, the \textbf{odd commutative graph complex} is a quotient
\begin{equation*}
\hGC_3 \defas \bigoplus_{(G,o)} \Q(G,o)\bigg/{\sim}
\end{equation*}
spanned by all oriented graphs which are connected and have at least 3 half-edges at every vertex. The relations imposed are $(G,o)\sim(G',\varphi_*(o))$ for every isomorphism $\varphi\colon G\cong G'$, and furthermore that $(G,o)\sim-(G,-o)$.
\end{definition}
Any automorphism $\alpha\colon G\cong G$ with $\alpha_*(o)=-o$ forces $(G,o)\sim -(G,o)\sim 0$. We call such automorphisms \emph{odd}. For example, swapping the two half-edges $e'\leftrightarrow e''$ of a self-loop constitutes an odd automorphism. Graphs with self-loops are therefore zero in $\hGC_3$.
We typically do not distinguish an oriented graph from its equivalence class in $\hGC_3$.
\begin{example}\label{ex:dipoles}
The $m$-fold \textbf{multiedge} or \textbf{dipole} is the oriented graph consisting of two vertices, ordered $0<1$, and $m$ edges in between---all directed from $0$ to $1$. For example,
    \begin{equation*}
    D_1=\Graph{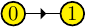},\quad
    D_2=\Graph{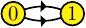},\quad
    D_3=\Graph{gcD3},\quad
    D_4=\Graph{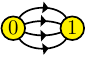}.
\end{equation*}
Reversing all edges and swapping the vertex order $0\leftrightarrow 1$ is an automorphism of $D_m$, which acts on its orientation as $o\mapsto (-1)^{m+1} o$. Therefore $D_{m}=0\in\hGC_3$ for $m$ even. Note that the graphs $D_1$ and $D_2$ do not belong to $\hGC_3$ (vertices with degree less than $3$).
\end{example}

The vector space $\hGC_3=\bigoplus_{\ell,k} \gr_{\ell,k} \hGC_3$ is bigraded by the \textbf{loop number} $\ell(G)$ and the \textbf{degree} $k=\gcdeg{G}$. For a connected graph with $n$ vertices and $m$ edges, they are
\begin{equation*}
    \ell(G) = m - n + 1
    \qquad\text{and}\qquad
    \gcdeg{G}
    = m - 3\ell(G)
    = 3n-2m-3
    .
\end{equation*}
The dimension of each piece $\gr_{\ell,k} \hGC_3$ is finite and equal to the number of isomorphism classes of graphs in this bidegree that do not have any odd automorphisms.
As each vertex has valence at least 3, necessarily $2m\geq 3n$ and so $\gr_{\ell,k}\hGC_3=0$ if $k>-3$ or $\ell<2$.

\begin{definition}\label{def:cGC}
    We denote the degree completion of $\hGC_3$ as
    \begin{equation*}
\cGC_3 = \prod_{\ell,k} \gr_{\ell,k} \hGC_3.
\end{equation*}
\end{definition}
Accounts on graph complexes typically discuss either the (homological) chain complex $\hGC_3$ \emph{or} the (cohomological) cochain complex $\cGC_3$, but not both. In the following, we first review the relevant combinatorial structures in these two settings separately. Then in \cref{sec:GC-duality} we make explicit the identification $\cGC_3\cong\Hom(\hGC_3,\Q)$ that relates them.

\subsection{Homology}\label{sec:GC-hom}
Given a graph $G$ and a subgraph $\gamma\subset G$, the quotient $G/\gamma$ is a graph with vertex set $V(G/\gamma)=(V(G)\setminus V(\gamma))\sqcup\{v_*\}$, and it is obtained from $G$ by removing all edges in $\gamma$ and identifying, in the remaining edges, all endpoints in $V(\gamma)$ with a single vertex $v_*$.

If $\gamma=e$ is a single edge with distinct endpoints $v_1$ and $v_2$, consider the half-edges $e'$ and $e''$ of $e$ at $v_1$ and $v_2$, respectively. We get a canonical isomorphism
\begin{alignat*}{2}
    \det \Z^{H(e)} &\otimes \det \Z^{V(G)}         & \quad & \cong \quad \det \Z^{V(G/e)} \\
    e'\wedge e''   &\otimes v_1\wedge v_2\wedge w  & \quad & \mapsto\quad v_*\wedge w
\end{alignat*}
where $w=v_3\wedge\ldots\wedge v_n$ for any order of the remaining vertices. Since all edges other than $e$ are in both graphs, we conclude that orientations of $G$ and $G/e$ are in natural bijection with each other. Let $(G,o)/e=(G/e,o/e)$ denote the \emph{oriented} quotient graph obtained in this way from an \emph{oriented} input graph $(G,o)$. Pictorially, $o/e$ is obtained by first moving the tail ($v_1$) and head ($v_2$) of $e$ to the beginning of the vertex order $v_1<v_2<\ldots<v_n$ (which may introduce a sign), then replacing $v_1<v_2$ by $v_*$, and keeping the directions of all edges other than $e$.

\begin{definition}
    The \textbf{boundary} $\partial\colon \hGC_3\longrightarrow \hGC_3$ is the linear map such that, for every oriented graph $(G,o)$ without self-loops,
    \begin{equation*}
  \partial(G, o)
    = \sum_{e \in E(G)} (G,o)/e.
\end{equation*}
\end{definition}
The boundary is homogeneous of bidegree $(0,-1)$ in the grading by $(\ell,k)$, and it turns $\hGC_3$ into a chain complex ($\partial^2=0$). Thus the \textbf{graph homology}
  \begin{equation*}
  H_{\bullet}(\hGC_3) = \frac{\ker \partial}{\im \partial}
  =\bigoplus_{\ell,k} \gr_\ell H_k(\hGC_3)
  \end{equation*}
is bigraded by the loop number $\ell$ and the degree $k$.

\begin{example}\label{ex:dipole}
    For all $i\geq 1$, the dipole $D_{2i+1}$ is closed (contracting any edge creates self-loops). For $i > 1$, it is the boundary of a triangle graph with a simple edge $0\rightarrow 1$, a double edge $0\rightarrow 2$, and $(2i-1)$-fold multiedge $1\rightarrow 2$:
  \begin{equation*}
  \partial\left( \Graph[0.8]{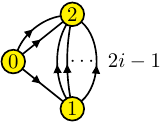} \right) = D_{2i+1}.
  \end{equation*}
  The theta graph $D_3$ has degree $k=-3$ and is therefore not exact, since $\hGC_3$ is concentrated in degrees $\leq -3$ (no graphs in degree $-2$ implies no boundaries in degree $-3$).
  Hence the theta class is non-zero in homology:
  \begin{equation*}
  [D_3] \neq 0 \in \gr_{2} H_{-3}(\hGC_3).
  \end{equation*}
\end{example}

For an arbitrary subgraph $\gamma\subset G$, an orientation $o$ of $G$ does not single out preferred orientations $o'$ and $o''$ of $\gamma$ and $G/\gamma$. But we can define a canonical \emph{pair}, so that the product $(\gamma,o')\otimes (G/\gamma,o'')$ is well-defined in $\hGC_3\otimes \hGC_3$. We use the isomorphism
\begin{alignat*}{2}
    \det \Z^{V(G)} \quad &\cong & \det \Z^{V(\gamma)} &\otimes \det \Z^{V(G/\gamma)} \\
    v_1\wedge\ldots\wedge v_n \quad &\mapsto \quad & v_1\wedge\ldots\wedge v_{n'} &\otimes v_*\wedge v_{n'+1}\wedge\ldots\wedge v_n
\end{alignat*}
where $v_1,\ldots,v_{n'}$ and $v_{n'+1},\ldots,v_n$ are the vertices in $V(\gamma)$ and $V(G)\setminus V(\gamma)$, respectively, in any order. Thus if any $o$ is given, move the vertices of $\gamma$ to the beginning of the vertex order of $G$ (which may produce a sign), and then $o'$ and $o''$ are defined by the induced vertex orders just described, and by keeping all edge directions as in $o$.
In other words, if both $G$ and the subgraph $\gamma$ are oriented, we have an induced orientation on $G/\gamma$.\footnote{The quotient $G/e$ by a single edge, discussed at the beginning of \cref{sec:GC-hom}, is the special case where we endow an edge $e$ with the orientation of $D_1\cong e$ (directing $e$ from the first to the second vertex).}
\begin{example}\label{ex:tri-quotient}
    Let $G$ denote the oriented triangle graph from \cref{ex:dipole} and consider the subgraph $\gamma\cong D_{2i-1}$ consisting of vertices $1$ and $2$ and the multiedge in between. Then
    \begin{equation*}
    G=\Graph[0.8]{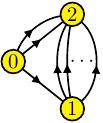}=(-1)^2\cdot \Graph[0.8]{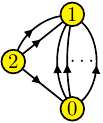}
    ,\quad\text{so}\quad
        G/\gamma = \Graph[0.8]{gcTri12kb}\Bigg/\Graph[0.8]{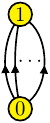}
        = \Graph{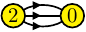}
        = -D_3.
    \end{equation*}
\end{example}
For an oriented graph $G$ and a subgraph $\gamma\subset G$, the expression $\gamma\otimes G/\gamma$ is therefore well-defined (independent of the orientation put on $\gamma$). We thus have a linear map
\begin{equation}\label{eq:GC-cop}
    \hGC_3\longrightarrow \hGC_3\otimes \hGC_3,\qquad
    G\mapsto \sum_{\gamma\subset G} \gamma\otimes G/\gamma
\end{equation}
by summing only over subgraphs $\gamma$ with the property that both $\gamma$ and $G/\gamma$ are in $\hGC_3$ (i.e.\ are connected and have each vertex of valency 3 or more).
This map is similar, but different, to the coproduct of the core Hopf algebra of Feynman graphs \cite{BlochKreimer:MixedHodge,Kreimer:Core}.\footnote{In the latter, graphs have no orientations, may have 2-valent vertices, may not have bridges, may be disconnected, and the quotient identifies only vertices within each connected component of $\gamma$.}

\subsection{Cohomology}\label{sec:GC-cohom}

Dual to the contraction of a subgraph, we can insert one oriented graph $G_2$ into a vertex $v$ of another oriented graph $G_1$. Let $H(v)$ denote the half-edges of $G_1$ at $v$. Any map $\rho\colon H(v)\rightarrow V(G_2)$ can be interpreted as a rule to re-attach the half-edges $h\in H(v)$ to vertices $\rho(h)$ in $G_2$. Afterwards, the vertex $v$ is left isolated. Starting from the disjoint union $G_1\sqcup G_2$, re-attaching $H(v)$, and then deleting $v$, produces a graph denoted $G_1\circ_{\rho} G_2$. Its vertex set is $V(G_2)\sqcup V(G_1)\setminus\{v\}$. The orientation of $G_1\circ_\rho G_2$ is determined by keeping all edge directions as in $G_1$ and $G_2$, and merging vertex orderings by putting the vertices from $G_2$ before those of $G_1$; i.e.\ exploiting the isomorphism
\begin{alignat*}{2}
	\det \Z^{V(G_1)} &\otimes \det \Z^{V(G_2)}
	& & \quad\cong\quad \det \Z^{V(G_1\circ_{\rho} G_2)}
	\\
	v\wedge w &\otimes w'  & & \quad\mapsto \quad w'\wedge w
\end{alignat*}
where $w\in\det \Z^{V(G_1)\setminus\{v\}}$ is any order of the remaining vertices from $G_1$.

The sum over all such insertions yields a bilinear operation $\circ\colon \cGC_3\otimes\cGC_3\rightarrow\cGC_3$ with
\begin{equation*}
G_1 \circ G_2 = \sum_{v \in V(G_1)}\ \sum_{\rho\colon H(v)\rightarrow V(G_2)} G_1 \circ_\rho G_2.
\end{equation*}
This insertion is not associative, but it gives rise to the graded\footnote{with signs $[x,y]+(-1)^{\gcdeg{x}\cdot\gcdeg{y}}[y,x]=0$ and $(-1)^{\gcdeg{x}\cdot\gcdeg{z}}[x,[y,z]]+(-1)^{\gcdeg{y}\cdot\gcdeg{x}}[y,[z,x]]+(-1)^{\gcdeg{z}\cdot\gcdeg{y}}[z,[x,y]]=0$} \textbf{Lie bracket} of $\cGC_3$,
\begin{equation*}
	[G_1, G_2] = G_1 \circ G_2 - (-1)^{\gcdeg{G_1}\cdot\gcdeg{G_2}} G_2 \circ G_1.
\end{equation*}
In this combination, contributions from attachments $G_1\circ_\rho G_2$ where $\rho(h)=v'$ glues all half-edges $h\in H(v)$ to the same vertex $v'$ of $G_2$, cancel with $G_2\circ_\rho G_1$ where $\rho(h)=v$ for all $h\in H(v')$. So to compute brackets, we only keep $\rho$ that take at least 2 values.

\begin{example}\label{ex:[D3prism]}
  Let $Y_3$ denote the triangular prism graph with the following orientation:
  \begin{equation*}
      Y_3=\Graph[0.8]{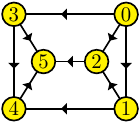}.
  \end{equation*}
  The dipole $D_3$ and $Y_3$ both have odd degree $\gcdeg{Y_3}=\gcdeg{D_3}=-3$, so their bracket is the sum $[Y_3,D_3] = Y_3 \circ D_3 + D_3 \circ Y_3$.
  Both graphs are vertex-transitive, so to compute $D_3\circ Y_3$ (first 5 terms below) and $Y_3\circ D_3$ (last 2 terms below), it suffices to insert only into vertex $0$, and then multiply by the number of vertices ($2$ respectively $6$).
  Explicitly, this gives\footnote{The terms from $D_3\circ Y_3$ where the three edges of $D_3$ are re-attached to vertices like $\{1,3,5\}$ of the prism graph are missing because they have an odd automorphism and thus are $0$ in $\cGC_3$.}
  \begin{align*}
    [Y_3,D_3]
    &= 24\times\Graph[0.65]{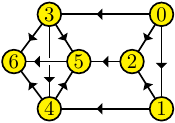}
    + 144\times\Graph[0.65]{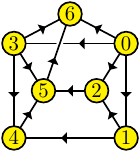}
    +  72\times\Graph[0.65]{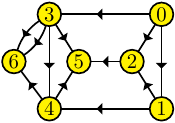}
    \\ &
    +  36\times\Graph[0.65]{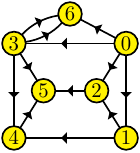}
    +  72\times\Graph[0.65]{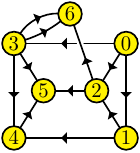}
    +  12\times\Graph[0.65]{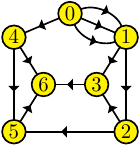}
    +  24\times\Graph[0.65]{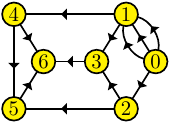}
    .
  \end{align*}
\end{example}

\begin{definition}\label{def:coboundary}
    The \textbf{coboundary} $\delta\colon \cGC_3\longrightarrow \cGC_3$ is the Lie bracket with an edge $D_1$:\footnote{The Lie bracket defined above for $\cGC_3$ extends without any modifications to arbitrary graphs, so we can input $D_1\notin\cGC_3$. One checks easily that when $G\in\cGC_3$, the output will land in $\cGC_3$ as well.}
\begin{equation*}
\delta G
= \frac{1}{2} \left[G, \Graph{gcD1} \right]
.
\end{equation*}
\end{definition}
Combinatorially, $\delta$ splits each vertex $v$ of $G$ in turn into two vertices $\{0,1\}=V(D_1)$, linked by an extra edge.
Flipping the role of $0$ and $1$ in a bipartition $H(v)=\rho^{-1}(0)\sqcup \rho^{-1}(1)$ produces isomorphic attachments $G\circ_\rho D_1\cong G\circ_{1-\rho} D_1$, only one of which is counted due the prefactor $1/2$.
New vertices of degree 2 or less are never created by $\delta$ (those appear twice but with opposite signs) and so it restricts to a well-defined map on $\cGC_3$.

The coboundary is homogeneous of bidegree $(0,+1)$ in the grading $(\ell,k)$, and it turns $\cGC_3$ into a cochain complex ($\delta^2=0$). Thus the \textbf{graph cohomology}
\begin{equation*}
  H^{\bullet}(\cGC_3) = \frac{\ker \delta}{\im \delta}
  =\prod_{\ell,k} \gr_\ell H^k(\cGC_3)
\end{equation*}
is bigraded by loop number ($\ell$) and degree ($k$). The triple $(\cGC_3,\delta,[\cdot,\cdot])$ forms a differential graded Lie algebra\footnote{with our conventions, the graded Leibniz rule reads $\delta[x,y]=[x,\delta y]+(-1)^{\gcdeg{y}}[\delta x,y]$} and so the graded Lie (super-)bracket descends to cohomology.
\begin{example}
    Every connected 3-regular graph has degree $k=-3$ and is therefore a cocycle (its coboundary has degree $-2$, where $\cGC_3$ is zero).
    In small loop orders $2\leq \ell\leq 4$, the cohomology groups are supported in degree $k=-3$ and have rank one, see \cref{table:GC3}. They are generated by the classes of the theta graph $D_3$ ($\ell=2$), the complete graph $K_4$ ($\ell=3$), and the triangular prism $Y_3$ ($\ell=4$). The Lie bracket of the classes of $D_3$ and $Y_3$ is a class in $\gr_6 H^{-6}(\cGC_3)$, for which a representative is given in \cref{ex:[D3prism]}.
\end{example}

\subsection{Duality}\label{sec:GC-duality}

From \cref{def:hGC,def:cGC} it is clear that $\cGC_3$ is the linear dual of $\hGC_3$. In order for the operations $\partial,\delta,G/\gamma,\circ$ recalled above to be exact duals of each other, we identify $\cGC_3\cong\Hom(\hGC_3,\Q)$ through the following pairing $\inner{\cdot,\cdot}\colon \cGC_3\otimes \hGC_3\rightarrow \Q$.\footnote{In \cite{Willwacher:GCgrt,Morand:KontsevichGCuniv}, graphs have a fixed vertex set $\{1,\ldots,n\}$ and cochains $\cGC_3$ are defined as \emph{invariants} for vertex permutations and coinvariants for edge reversals. Instead, we write cochains like the chains $\hGC_3$ \cite{kontsympl,onatheorem}, i.e.\ as \emph{coinvariants} (equivalence classes under oriented graph isomorphisms $\sim$) with respect to vertex permutations \emph{and} edge reversals. The factor of $A=\Aut(G)$ in $\inner{\cdot,\cdot}$ thus comes from identifying $(\Gra_A)^\vee\cong (\Gra^\vee)^A$ with $(\Gra^\vee)_A$, for $\Gra$ the span of labelled graphs.}
\begin{definition}
For any two oriented graphs $G$ and $G'$, set
\begin{equation}\label{eq:GC-duality}
\inner{G, G'} \defas \begin{cases}
    \phantom{-}\abs{\Aut G} & \text{if $G \sim G'$ and $G\not\sim 0$,} \\
    -\abs{\Aut G} & \text{if $G \sim -G'$ and $G\not\sim 0$,} \\
    \phantom{-}0 & \text{otherwise.} \\
\end{cases}
\end{equation}
In other words, $\inner{G,G'}=\inner{G',G}=\epsilon|\Aut G| =\epsilon|\Aut G'|$ where $\epsilon=\pm 1$ if $G$ and $G'$ are isomorphic and don't have any odd automorphisms, and $\epsilon=0$ otherwise.
\end{definition}

Here we denote $\Aut G$ the group of automorphisms of the combinatorial half-edge graph $\underline{G}$, ignoring the orientation $o$ of $G=(\underline{G},o)$. Such an automorphism induces a permutation of the vertices. For a simple graph $\underline{G}$ (no self-loops or multiedges), this permutation determines the automorphism. In general, there is a short exact sequence
\begin{equation*}
    1\rightarrow K_G
    \rightarrow \Aut(G)
    \rightarrow \Aut_V(G)
    \rightarrow 1
\end{equation*}
of groups where $\Aut_V(G)\subseteq\Perms{V(G)}$ is a subgroup of vertex permutations and the kernel $K_G$ is generated by self-loop reversals and multiedge permutations. It has
\begin{equation*}
    \left|K_G\right|= \bigg(\prod_v d_v!\cdot 2^{d_{v}}\bigg) \prod_{v<w} (d_{vw}!)
\end{equation*}
elements, where $d_v$ and $d_{vw}$ denote the number of self-loops at $v$ and the number of edges between two different vertices $v$ and $w$, respectively.
\begin{example}
    Dipole graphs $D_m$ with $m$ edges (\cref{ex:dipoles}) have $|\Aut(D_m)|=m!\cdot 2$ automorphisms, where $m!=d_{01}!$ and $2=|\Aut_V(D_m)|=|\{\id,0\leftrightarrow 1\}|$.
\end{example}
\begin{lemma}\label{lem:inner=sum sgn}
    Given two oriented graphs $G=(\underline{G},o)$ and $G'=(\underline{G}',o')$, write $\varphi\colon G\cong G'$ for isomorphisms of the (un-oriented) graphs $\underline{G}\cong\underline{G}'$. Let $\sgn\varphi\defas (\varphi_* o)/o'\in\{1,-1\}$ denote the relative sign between the orientations. Then the sum over all isomorphisms is
\begin{equation*}
    \inner{G,G'} = \sum_{\varphi\colon G\cong G'} \sgn \varphi.
\end{equation*}
\end{lemma}
\begin{proof}
    If $G$ and $G'$ are not isomorphic, the sum is trivially zero. So assume there exists at least one isomorphism $\varphi$, hence $G\sim(\sgn\varphi)G'$. The set of all isomorphisms is a torsor over $\Aut(G)$. Since $\sgn(\varphi\circ\alpha)=(\sgn\varphi)(\sgn\alpha)$, the sum thus equals
    \begin{equation*}
        (\sgn\varphi)\sum_{\alpha\in\Aut(G)}\sgn \alpha
        =(\sgn\varphi)\left(|\Aut^+(G)|-|\Aut^-(G)| \right)
    \end{equation*}
    where $\Aut^+$ and $\Aut^-$ denote the even and odd automorphisms. If $G$ has an odd automorphism $\beta$, then $\alpha\mapsto\alpha\circ\beta$ gives a bijection $\Aut^+(G)\leftrightarrow \Aut^-(G)$ and thus the sum vanishes---in agreement with $\inner{G,G'}=0$ since $G\sim0$. If $G$ has no odd automorphism, then $\Aut^+(G)=\Aut(G)$ and so the sum reduces to the definition of $\inner{G,G'}$.
\end{proof}

\begin{thm}\label{thm:pairing}
  For any three oriented graphs $G_1, G_2, G$, we have the identity
  \begin{equation*}
    \inner{G_1 \circ G_2, G} = \sum_{\gamma \subset G}
    \inner{G_2, \gamma}\inner{G_1,G/\gamma}
  \end{equation*}
  where the sum runs over all subgraphs $\gamma$ of $G$.
\end{thm}
\begin{proof}
    Expanding both sides using \cref{lem:inner=sum sgn}, the claimed identity has the form
    \begin{equation*}
        \sum_{(v,\rho,\chi)\in L} \sgn \chi
        =
        \sum_{(\gamma,\varphi,\psi)\in R} (\sgn \varphi)(\sgn\psi).
    \end{equation*}
    The left-hand side sums over the set $L$ of triples consisting of a vertex $v\in V(G_1)$, a map $\rho\colon H(v)\rightarrow V(G_2)$, and an isomorphism $\chi\colon G_1\circ_\rho G_2\cong G$.
    The right-hand side sums over the set $R$ of triples consisting of a subgraph $\gamma\subset G$, an isomorphism $\varphi\colon G_2\cong \gamma$, and an isomorphism $\psi\colon G_1\cong G/\gamma$.

    The sets $L$ and $R$ are in bijection: $G_2$ is canonically a subgraph of $G_1\circ_\rho G_2$, and the quotient is canonically isomorphic to $(G_1\circ_\rho G_2)/G_2\cong G_1$. So given $(v,\rho,\chi)$, we can set $\gamma=\chi(G_2)$, $\varphi=\chi|_{G_2}$, and $\psi$ the induced isomorphism $G_1\cong\chi(G_1\circ_\rho G_2)/\chi(G_2)=G/\gamma$. Conversely, given $(\gamma,\varphi,\psi)$, we can recover the special vertex $v=\psi^{-1}(v_*)$; the vertices in $G$ incident to the half-edges at $v_*$ in $G/\gamma$ determine $\rho$ via $\varphi$, and we can glue $\varphi$, $\psi$, and $\rho$ into $\chi$. Under this correspondence, $\sgn \chi=(\sgn\varphi)(\sgn\psi)$ by the definitions of the orientations on $G_1\circ_\rho G_2$ (vertices of $G_2$ come first) and $\gamma\otimes G/\gamma$ (vertices of $\gamma$ come first).
\end{proof}

\begin{corollary}
    For any cochain $Q \in \cGC_3$ and any chain $G \in \hGC_3$, we have
    \begin{equation*}
        \inner{\delta Q, G} = \inner{Q, \partial G}
        .
    \end{equation*}
\end{corollary}
\begin{proof}
  Recall that $\abs{D_1} = 1$ for $D_1=\Graph[0.8]{gcD1}$. For oriented graphs $Q$ and $G$,
  \begin{align*}
    \inner{Q \circ D_1, G}
      = \sum_{e \in E(G)} \inner{D_1, e}\inner{Q, G/e}
      = 2 \inner{Q, \partial G}
  \end{align*}
  by \cref{thm:pairing}.
  Furthermore, $\inner{D_1 \circ Q, G} = 0$, as having a subgraph $\gamma \subset G$ such that  $G/\gamma \cong D_1$ requires a 1-valent vertex in $G$.
  By \cref{def:coboundary}, the claim follows.
\end{proof}

\begin{example}
  Consider the following graphs $Q$ and $G$ and their (co-)boundaries:
  \begin{align*}
    Q &= \Graph[0.8]{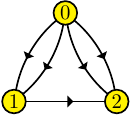},  &   \delta Q &= 2 \times \Graph[0.8]{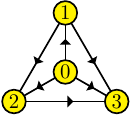} +\;\Graph[0.8]{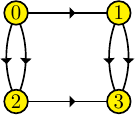},  \\
    G &= \Graph[0.8]{gcK4},    & \partial G &= 6 \times \Graph[0.8]{gcT122}.
\end{align*}
  These graphs have $\abs{\Aut Q} = 8$ and $\abs{\Aut G} = 24$ automorphisms, so the identity reads
  \begin{equation*}
  8 \cdot 6 = \inner{Q, \partial G} = \inner{\delta Q, G} = 24 \cdot 2
  .
  \end{equation*}
\end{example}

\begin{remark}
  The results of this subsection hold in all commutative graph complexes $\hGC_d$. For even $d$, the notion of orientation and hence the induced orientations on $G_1\circ G_2$ and $G/\gamma$ change, but the proof of \cref{thm:pairing} persists verbatim. The contribution of self-loops to $\Aut(G)$ was discussed specifically to clarify the case of even $d$.
\end{remark}

\subsection{Orientations}\label{SS:equivor}

In \cref{def:GCor} we view orientations as equivalence classes of ways to direct edges and order vertices. This convention from \cite{Willwacher:GCgrt} is convenient for computer calculations and allows for simple indication of orientations in figures. To associate integrals of Pfaffian forms to graphs, however, we interpret orientations through canonical isomorphisms
\begin{equation}\label{eq:Hor=Vor}
    \det \Z^{V(G)} \otimes \bigotimes_{e\in E(G)} \det \Z^{H(e)}
    \cong
    \det \Z^{E(G)} \otimes \det H_1(G)
\end{equation}
explained in \cite[Proposition~1]{onatheorem}. Recall that each edge $e\in E(G)$ consists of two half-edges $H(e)=\{e',e''\}$. The 1-simplex corresponding to $e$ is oriented as $e'\wedge e''\in\det\Z^{H(e)}$ or $e''\wedge e'=-e'\wedge e''$ by ordering these half-edges. The group of 1-chains is thus
\begin{equation*}
    C_1(G)=\bigoplus_{e\in E(G)} \det \Z^{H(e)}.
\end{equation*}
This group is non-canonically isomorphic to $C_1(G)\cong \Z^{E(G)}$, by choosing a direction (ordering of the two half-edges) for each edge to identify $\det\Z^{H(e)}\cong\Z$.

The \emph{cycle space} is the subgroup $H_1(G)\subseteq C_1(G)$ given by the kernel of the boundary map $\partial\colon C_1(G)\rightarrow C_0(G)=\Z^{V(G)}$. If $v'$ and $v''$ denote the endpoints of the half-edges $e'$ and $e''$ of edge $e$, respectively, then $\partial (e'\wedge e'')=v''-v'$ (head minus tail). With the augmentation map defined by $\varepsilon(v)=1$ for every vertex $v$, we have the exact sequence
\begin{equation*}
0 \longrightarrow H_1(G) \longrightarrow C_1(G) \overset{\partial}{\longrightarrow} \Z^{V(G)} \overset{\varepsilon}{\longrightarrow} \Z \longrightarrow 0
\end{equation*}
for every connected graph. It implies $\det C_1(G)\cong\det H_1(G)\otimes \det (\ker\varepsilon)$ and $\det \Z^{V(G)}\cong\det(\ker\varepsilon)$, where we fix the latter as
$v_1\wedge\ldots\wedge v_n\mapsto
(v_2-v_1)\wedge(v_3-v_1)\wedge\ldots\wedge(v_n-v_1)
$.
The isomorphism \eqref{eq:Hor=Vor} then arises from the canonical identification
\begin{equation*}
    \det C_1(G)\cong \det \Z^{E(G)} \otimes \bigotimes_{e\in E(G)} \det \Z^{H(e)}
    .
\end{equation*}

Concretely, we can describe the isomorphism \eqref{eq:Hor=Vor} for a connected graph $G$ as follows:
Suppose that generators $o$ of the left- and $o'$ of the right-hand sides of \eqref{eq:Hor=Vor} are given as
\begin{equation*}
o=v_1\wedge\ldots \wedge v_n\otimes \bigotimes_e o_e
\qquad\text{and}\qquad
o'=e_1\wedge\ldots\wedge e_m\otimes C_1\wedge\ldots\wedge C_{\ell}.
\end{equation*}
Here $o_e\in\det\Z^{H(e)}$ are edge directions and $C_1,\ldots,C_\ell\in H_1(G)$ are cycles forming a basis of the cycle space. For each $2\leq i\leq n$, let $P_i\in C_1(G)$ denote a preimage of $v_i-v_1=\partial P_i$, e.g.\ the sum of edges of a directed path from $v_1$ to $v_i$. Since any two such paths differ by a cycle, the element $C_1\wedge\ldots\wedge C_\ell\wedge P_2\wedge\ldots\wedge P_n\in\det C_1(G)$ is independent of the choices of $P_i$. This generator of $\det C_1(G)$ is equal to $(\det A)\cdot o_{e_1}\wedge\ldots\wedge o_{e_m}$, where
\begin{equation*}
A = \begin{pmatrix} C_1 & \cdots & C_{\ell} & P_2 & \cdots & P_n\end{pmatrix}
\in \GL_m(\Z)
\end{equation*}
denotes the square $m\times m$ matrix where 1-chains are written as column vectors, identifying $C_1(G)\cong \Z^m$ such that the $i$th row corresponds to $o_{e_i}\in C_1(G)$.
Thus \eqref{eq:Hor=Vor} is the map
\begin{equation*}
o \mapsto (\det A)\cdot  o'.
\end{equation*}

\begin{example}\label{ex:dipoleor}
The orientation $o=v_0\wedge v_1\otimes\bigotimes_e o_e$ of the dipoles $D_{2i+1}$ from \cref{ex:dipoles} is given by directing all edges from vertex $0$ towards vertex $1$. That is, for every edge, $o_e=e'\wedge e''$ in terms of the half-edges $e'$ and $e''$ at vertex 0 and vertex 1, respectively.

Pick any order of the $m=2i+1$ edges, set $C_j=o_{e_j}-o_{e_m}$ for $1\leq j\leq \ell=2i$ and let $o'=e_1\wedge\ldots\wedge e_m\otimes C_1\wedge\ldots\wedge C_\ell$. Pick the path $P=o_{e_m}$ from $v_0$ to $v_1$. Then clearly
\begin{equation*}
    C_1\wedge\ldots\wedge C_\ell\wedge P=o_{e_1}\wedge\ldots\wedge o_{e_m}
\end{equation*}
and thus $o'\in\det \Z^{E(D_{2i+1})}\otimes\det H_1(D_{2i+1})$ corresponds to $o$ under \eqref{eq:Hor=Vor}.
For the theta graph, this cycle basis and the corresponding matrix (with $\det A=1$) look like
  \begin{equation*}
      D_3=\Graph{gcD3cyc}
      \qquad\text{and}\qquad
      A=\begin{pmatrix}
          C_1 & C_2 & P
      \end{pmatrix}
      =\begin{pmatrix}
           1 &  0 & 0 \\
           0 &  1 & 0 \\
          -1 & -1 & 1
      \end{pmatrix}
      .
  \end{equation*}
\end{example}

\section{Pfaffian forms on graphs}
\label{sect:pff-graphs}

Throughout, let $G$ denote a graph with $n$ vertices and $m$ edges, $E=\{e_1,\ldots,e_m\}$. To each edge, we associate a coordinate $x_e$ of the projective space $\Pro(\R^E)\cong \R\Pro^{m-1}$. The positive part of this projective space is an open simplex, which we denote
\begin{equation*}\label{eq:oFS}
	\oFS_G = \FS_m = \left\{ [x_{e_1}:\cdots : x_{e_m}]\ \text{with all}\  x_e>0 \right\}
	\subset \Pro(\R^E).
\end{equation*}
Its closure $\cFS_G\subset\Pro(\R^E)$ is a compact manifold with corners. Facets $\cFS_G\cap\{x_e=0\}\cong\cFS_{G/e}$ can be identified with closed simplices of quotient graphs.
To fix the fundamental class
\begin{equation*}
    [\cFS_G] \in H_{m-1}\Big(\Pro(\R^E), \bigcup_{e\in E} \{x_e=0\}; \Z\Big),
\end{equation*}
we orient the simplex $\FS_G$ using the $(m-1)$-form on the positive orthant $\R^E_+$ given by
\begin{equation}\label{eq:simplex-volform}
	\Omega_m = \sum_{e=1}^m (-1)^{e-1} x_e \td x_1\wedge\ldots\wedge \td x_{e-1} \wedge\td x_{e+1}\wedge\ldots\wedge\td x_m.
\end{equation}
This form pulls back as $\Omega_m\mapsto \lambda^m\Omega_m$ under the action $x\mapsto \lambda x$ of $\R_+$ on $\R_+^m$, and it is horizontal ($\iota_V\Omega_m=0$) for the Euler vector field $V=\sum_{e=1}^m x_e \partial_e$ since $\Omega_m=\iota_Z (\bigwedge_{e=1}^m \td x_e)$. Thus $\Omega_m$ induces a well-defined orientation on $\sigma_m=\R_+^m/\R_+$. In the affine chart $x_1=1$ of $\R\Pro^{m-1}$, the frame $\partial_2\wedge\ldots\wedge\partial_m$ of the coordinate vector fields is thus positively oriented on $\FS_m$.

Every permutation $\alpha$ of the edges acts on $\R^E$ by a linear automorphism such that
\begin{equation}\label{eq:volform-sign}
    \alpha^* \Omega_m=(\sgn \alpha) \Omega_m.
\end{equation}
Therefore, in particular, the automorphisms of $G$  act on the fundamental class $[\cFS_G]$ via the sign of the corresponding edge permutation.

\subsection{The dual graph Laplacian}
To define differential forms on a graph simplex $\oFS_G$, we map the simplex to the space of positive definite matrices (actually, a quotient of the latter), as follows.

The combinatorial structure of a graph can be encoded by the boundary map
\begin{equation*}
    \partial\colon
    C_1(G)\cong\Z^E \rightarrow \Z^V,
    \qquad
    e\mapsto \text{head}(e)-\text{tail}(e)
\end{equation*}
of the underlying simplicial complex, sending each edge $e\in E$ to the difference of its endpoints.\footnote{The identification $C_1(G)\cong\Z^E$ depends on a choice of direction of each edge, i.e.\ picking which of the endpoints is called head and tail (see \cref{SS:equivor}).} The kernel of this map is the cycle space $H_1(G)$ of $G$.
This abelian group is free with rank $\ell=m-n+1$, the \emph{loop number}. An isomorphism
\begin{equation*}
    H_1(G)\cong\Z^\ell
\end{equation*}
amounts to an integer $m\times\ell$ matrix $\mathcal{C}=(C_1,\ldots,C_{\ell})$ whose columns $C_i\in \Z^E$ constitute an ordered basis of the cycle space.

\begin{definition}\label{def:lapmat}
Let $\mathcal{D}$ denote the $m\times m$ diagonal matrix with diagonal entries $x_1,\ldots,x_m$.
The \textbf{dual Laplacian}\footnote{In graph theory, the Laplacian matrix refers to a similar matrix, where $\mathcal{C}$ is replaced by the edge-vertex incidence matrix. We hence refer to $\Laplacian_\mathcal{C}$ as the \emph{dual} Laplacian. This terminology is reversed in \cite{invariant}.} matrix of a graph $G$, with respect to a basis $\mathcal{C}=(C_1,\ldots,C_{\ell})$ of the cycle space $H_1(G)$, is the symmetric $\ell\times\ell$ matrix
  \begin{equation*}
  \Laplacian_\mathcal{C} \defas \mathcal{C}^\Transpose \mathcal{D} \mathcal{C}.
  \end{equation*}
\end{definition}

More invariantly, the dual Laplacian is the symmetric bilinear form on $H_1(G)$ defined by $(a,b)=\sum_{e=1}^m x_e a_e b_e$ for $a,b\in H_1(G)\subseteq\Z^E$. In the basis $\mathcal{C}$, this form is represented by the matrix $\Laplacian_{\mathcal{C}}$. Any other basis $\mathcal{C}'$ has the form $\mathcal{C}'=\mathcal{C}P$ for some invertible matrix $P\in\GL_\ell(\Z)$, so the dual Laplacian matrix of a graph is defined up to conjugation:
\begin{equation}\label{eq:lapmat-trans}
    \Laplacian_{\mathcal{C}'} = P^\Transpose \Laplacian_{\mathcal{C}} P.
\end{equation}

The entries of the matrix $\Laplacian_\mathcal{C}$ are linear functions of the variables $x\in\R^E$, thus the matrix can be interpreted as a function from $\R^E$ to the space of symmetric matrices. If all coordinates $x_e$ are positive, then the dual Laplacian form is positive definite. Therefore, a cycle basis $\mathcal{C}$ gives rise to a smooth map
\begin{equation*}
    \Laplacian_{\mathcal{C}}\colon \R^E_+ \longrightarrow \SPD{\ell}
\end{equation*}
to the space $\SPD{\ell}$ of positive definite symmetric $\ell\times\ell$ matrices. Under this map, a scaling $x\mapsto\lambda x$ by $\lambda\in\R_+$ amounts to $\Laplacian_{\mathcal{C}}\mapsto \lambda\Laplacian_{\mathcal{C}}=g^\Transpose \Laplacian_{\mathcal{C}} g$ where $g=\sqrt{\lambda}$ is a multiple of the identity matrix. Thus, the above map descends to a smooth map
\begin{equation*}
    [\Laplacian_\mathcal{C}]\colon \oFS_G \longrightarrow L\SPD{\ell}=\SPD{\ell}/\R^\times
\end{equation*}
of the open simplex $\oFS_G\subset\Pro(\R^E)$ to the link of $\SPD{\ell}$. This link, $L\SPD{\ell}$, is the quotient of the symmetric space $\SPD{\ell}=\GL_\ell(\R)/\OG_\ell(\R)$ by the action of the central subgroup $\R^\times\subset \GL_\ell(\R)$ consisting of scalar multiples of the identity matrix.

\begin{definition}\label{def:Symanzik}
    The \textbf{Symanzik polynomial} $\Symanzik_G$ of a graph $G$ is the homogeneous degree $\ell$ polynomial in $(x_e)_{e\in E}$ given by the determinant of any dual Laplacian matrix:
\begin{equation*}
    \Symanzik_G \defas \det \Laplacian_\mathcal{C}.
\end{equation*}
\end{definition}
Since $\det P=\det P^\Transpose=\pm 1$ in \eqref{eq:lapmat-trans}, this determinant is indeed independent of the choice of cycle basis $\mathcal{C}$. By  \cite[Proposition~2.2]{BlochEsnaultKreimer:MotivesGraphPolynomials}, if $G$ is connected, then
\begin{equation*}
    \Symanzik_G = \sum_{T} \prod_{e \notin T} x_e
\end{equation*}
is a sum over all spanning trees $T$ of $G$. Since every connected graph has a spanning tree, $\Symanzik_G$ takes positive values on $x\in\R_+^E$. In particular, $\Symanzik_G$ is not the zero polynomial.
\begin{remark}\label{rem:Symanzik-disconnected}
If $G$ has several connected components $G=G_1\sqcup\ldots\sqcup G_r$, then $H_1(G)=H_1(G_1)\oplus\ldots\oplus H_1(G_r)$. Thus $\Laplacian_\mathcal{C}$ is block diagonal such that
$\Symanzik_G=\Symanzik_{G_1}\!\cdots \Symanzik_{G_r}\neq 0$.
\end{remark}

\subsection{Differential forms associated to graphs}
Pulling back invariant forms along the map $[\Laplacian_\mathcal{C}]\colon \oFS_G\longrightarrow L\SPD{\ell}$ induced by the dual Laplacian, we obtain differential forms on the graph simplex. We denote these pullbacks as $\pff_{\Laplacian_\mathcal{C}}\defas [\Laplacian_\mathcal{C}]^* (\pff)$ and $\beta_{\Laplacian_{\mathcal{C}}}^{4i+1}\defas [\Laplacian_\mathcal{C}]^* (\beta^{4i+1})$, for the Pfaffian and primitive canonical forms, respectively. Whenever $\ell$ is odd, recall that we set $\pff=0$.
\begin{definition}
    For a graph $G$ with $\ell$ loops, and any positive integer $i$, we have smooth closed differential forms on the open simplex $\oFS_G\subset\Pro(\R^E)$, of degrees $\ell$ and $4i+1$, denoted
\begin{equation*}
    \pff_G
    \defas \pff_{\Laplacian_{\mathcal{C}}}
    \in\Omega^\ell(\oFS_G)
    \qquad\text{and}\qquad
    \omega_G^{4i+1}
    \defas \beta_{\Laplacian_{\mathcal{C}}}^{4i+1}
    \in \Omega^{4i+1}(\oFS_G)
    .
\end{equation*}
\end{definition}
Since $\beta^{4i+1}_X$ is $\GL_\ell(\R)$ invariant, it follows from \eqref{eq:lapmat-trans} that $\omega_G^{4i+1}$ is independent of the chosen cycle basis $\mathcal{C}=(C_1,\ldots,C_\ell)$. The Pfaffian $\pff_G$ is only determined up to a sign: for another cycle basis $\mathcal{C}'=(C_1',\ldots,C_\ell') P$ with $P\in\GL_\ell(\Z)$, \cref{lem:pfprops} (iii) shows that
\begin{equation}\label{eq:pff-sign}
    \pff_{\Laplacian_{\mathcal{C}'}} = \pff_{\Laplacian_{\mathcal{C}}} \cdot \det P = \pm \pff_{\Laplacian_{\mathcal{C}}}
\end{equation}
where the ratio $\det P=C_1'\wedge\ldots\wedge C'_\ell / C_1\wedge\ldots\wedge C_\ell=\pm 1$ in $\det H_1(G)$ compares the orientations of both cycle bases. This sign will be fixed later and is in fact crucial to define integrals associated to \emph{oriented} graphs, i.e.\ elements of $\hGC_3$, in \cref{def:canint}.

It follows from \eqref{eq:pff-pole} and \eqref{eq:primcan-pole} that $\pff_G$ and $\omega_G^{4i+1}$ are polynomial forms in $x_e$ and $\td x_e$, divided by $\Symanzik^{(\ell+1)/2}$ and $\Symanzik^{i+1}$, respectively.

As an example, for all odd-fold multiedges (even loop order dipole graphs) we obtain:
\begin{lemma} \label{lem:pff-dipoles}
For any $i \geq 1$, consider the dipole graph $D_{2i+1}$ with edges $\{1,\ldots,2i+1\}$ directed from vertex 0 to vertex 1. For the cycle basis $\mathcal{C}=(e_1-e_{2i+1},\ldots,e_{2i}-e_{2i+1})$,
\begin{equation}\label{eq:pff-dipoles}
  \pff_{D_{2i+1}} = (-1)^{i}\frac{(2i)!}{2^{i}\cdot i!}\frac{(x_1\cdots x_{2i+1})^{i-1} \Omega_{2i+1}}{\Symanzik^{i+1/2}}
\end{equation}
with the $2i$-form $\Omega_{2i+1}$ from \eqref{eq:simplex-volform} and the Symanzik polynomial $\Symanzik=\sum_{a=1}^{2i+1} \prod_{b\neq a} x_b$.
\end{lemma}
\begin{proof}
	As any two cycles in this basis intersect only in edge $2i+1$, the dual Laplacian $2i\times 2i$ matrix $\Laplacian=\Laplacian_{\mathcal{C}}$ has entries $\Laplacian_{a,b}=x_{2i+1}+x_a\delta_{a,b}$. Apply \cref{lem:pff(diag+const)}.
\end{proof}

Combining the Pfaffian and canonical forms, we obtain a host of closed smooth differential forms $\pff_G\wedge\omega_G$ on simplices.
As the Pfaffian vanishes for odd-dimensional matrices, these forms are only interesting for graphs $G$ with even loop number $\ell$.

\begin{lemma}\label{lem:pfbeta-props}
  Given a graph $G$ and any canonical form $\omega\in\cfs$, fix any cycle basis $\mathcal{C}$ and let $\pff_G=\pff_{\Laplacian_\mathcal{C}}$ and $\omega_G=\omega_{\Laplacian_\mathcal{C}}$ denote the corresponding forms.
  We have the properties:
  \begin{enumerate}
    \item \textbf{Automorphism.} Let $\alpha$ be any automorphism of $\,G$ and denote by $\alpha^*$ the pullback under the corresponding edge permutation: $\alpha^*(x_e) = x_{\alpha(e)}$. Then there exists a matrix $P\in\GL_\ell(\Z)$ such that $\alpha^* \Laplacian_\mathcal{C} = P^\Transpose \Laplacian_{\mathcal{C}} P$, and with such a matrix, we have
      \begin{equation}\label{eq:pffcan-sign}
          \pff_G \wedge \omega_G = \det{P} \cdot \alpha^* \left(\pff_G \wedge \omega_G \right).
      \end{equation}
    \item \textbf{Contraction.} Let $e$ be an edge that is not a self-loop. Then $H_1(G)\cong H_1(G/e)$ are canonically isomorphic, so $\mathcal{C}$ determines a cycle basis $\mathcal{C}/e$ of $G/e$. Set $\pff_{G/e}=\pff_{\Laplacian_{\mathcal{C}/e}}$. Then restriction to the hyperplane $\{x_e = 0\}$ can be identified with edge contraction:
      \begin{equation*}
      \left(\pff_G \wedge \omega_G \right)\big|_{x_e=0} = \pff_{G/e} \wedge \omega_{G/e}.
      \end{equation*}
    \item \textbf{Series.} Let $G'$ be the graph obtained from $G$ by subdividing an edge $e$ (with a new two-valent vertex) into edges $e'$ and $e''$.
      Denote $s_e\colon \R^{E'}_+ \to \R^{E}_+$ the map $y=s_e(x)$ such that $y_e=x_{e'}+x_{e''}$ and $y_i=x_i$ for $i\neq e$. Replace each occurrence of $e$ in $\mathcal{C}$ by the path $e'e''$, to obtain a cycle basis $\mathcal{C}'$ of $G'$. Set $\pff_{G'}=\pff_{\Laplacian_{\mathcal{C}'}}$. Then
      \begin{equation*}
      \pff_{G'} \wedge \omega_{G'} = s_e^* \left(\pff_G \wedge \omega_{G} \right).
      \end{equation*}
    \item \textbf{Components.} Suppose that the edges of $G$ can be partitioned into two subgraphs $G_1,G_2$ such that $G_1$ and $G_2$ have at most one vertex in common.
      Let $\Delta\omega =\sum_{(\omega)} \omega^{(1)} \otimes \omega^{(2)}$ denote the coproduct from \cref{sec:Hopf-can}.
      Then
  \begin{equation} \label{pffgraph-blocks}
    \pff_{G} \wedge \omega_G
    = \sum_{(\omega)} \left(\pff_{G_1} \wedge \omega^{(1)}_{G_1}\right) \wedge \left(\pff_{G_2} \wedge \omega^{(2)}_{G_2}\right).
  \end{equation}
  \end{enumerate}
\end{lemma}
\begin{proof}
  By \cite[\S6.2,6.3]{invariant}, the canonical forms $\omega_G$ themselves satisfy $\omega_G=\alpha^*(\omega_G)$, $\omega_G|_{x_e=0}=\omega_{G/e}$, $s_e^*(\omega_G)=\omega_{G'}$, and, for primitive forms, $\omega_G^{4i+1} = \omega_{G_1}^{4i+1} + \omega_{G_2}^{4i+1}$.
  It remains only to establish the corresponding relations for the Pfaffian form $\pff_G$:

  Every automorphism of $G$ induces an  automorphism of the cycle space $H_1(G)$, hence the matrix $P$ exists. Then \eqref{eq:pffcan-sign} follows from \eqref{eq:pff-sign}.

  Cycles of $G$ become cycles of $G/e$ by removing all occurrences of $e$, thus $\Laplacian_{\mathcal{C}/e} = \Laplacian_{\mathcal{C}}|_{x_e = 0}$ and therefore $\pff_G|_{x_e=0}=\pff_{G/e}$.

  Subdividing an edge yields a canonical isomorphism $H_1(G)\cong H_1(G')$, replacing $e$ by the path $e'e''$. Thus for the induced cycle basis, $\Laplacian_{\mathcal{C}'} = s^*_e \Laplacian_{\mathcal{C}}$ and we get $\pff_{G'} = s^*_e \pff_{G}$.

  When two graphs $G_1,G_2$ share at most one vertex, then the cycle space of $G=G_1\cup G_2$ is canonically a direct sum $H_1(G) = H_1(G_1) \oplus H_1(G_2)$. We can choose cycle bases $\mathcal{C}_i$ of $G_i$ such that $(\mathcal{C}_1,\mathcal{C}_2)$ defines the same orientation of $H_1(G)$ as $\mathcal{C}$. Then $P^\Transpose\Laplacian_{\mathcal{C}}P = \Laplacian_{\mathcal{C}_1} \oplus \Laplacian_{\mathcal{C}_2}$ for a matrix $P$ with $\det P=1$.
  The identity $\omega_G^{4i+1} = \omega_{G_1}^{4i+1} + \omega_{G_2}^{4i+1}$ generalizes to polynomials in primitive forms by the very definition of the coproduct (see \cref{sec:Hopf-can}) as
  \begin{equation*}
  \omega_G = \sum_{(\omega)} \omega^{(1)}_{G_1} \wedge \omega^{(2)}_{G_2}.
  \end{equation*}
  Since \cref{lem:pfprops} (v) implies $\pff_G = \pff_{G_1} \wedge \pff_{G_2}$, the claimed relation follows.
\end{proof}
A connected graph is biconnected if every pair of distinct edges is contained in a cycle without repeated vertices. A \emph{biconnected component} is an edge-maximal biconnected subgraph. Each edge belongs to exactly one biconnected component. Any two biconnected components have at most one vertex in common (such a vertex is called a \emph{cut vertex}).
\begin{corollary}\label{cor:pfform}
  If $G$ has a biconnected component with odd loop number
  , then $\pff_G=0$.
\end{corollary}
\begin{proof}
    Let $G_1,\ldots,G_c$ denote the biconnected components of $G$. Then $\pff_G=\pff_{G_1}\wedge\ldots\wedge\pff_{G_c}$ by \eqref{pffgraph-blocks} and induction. If any $G_i$ has odd loop number, then $\pff_{G_i}=0$ and so $\pff_G=0$.
\end{proof}

\subsection{Whitney flip}

Two graphs may give rise to the same dual Laplacian, and hence the same Pfaffian (and canonical) forms, even if they are not isomorphic. It suffices that they are \emph{2-isomorphic}, that is, their edges can be identified in such a way that cycles are mapped to cycles and vice versa.\footnote{In matroid terms, two graphs are 2-isomorphic precisely when their cycle matroids are isomorphic.} For example, a graph $G=G_1\sqcup G_2$ with two connected components is 2-isomorphic to any graph $G'$ obtained from $G$ by identifying a vertex in $G_1$ with a vertex in $G_2$. In this case, the cycle spaces of $G$ and $G'$ are both canonically identified with $H_1(G_1)\oplus H_1(G_2)$. This was exploited in \cref{lem:pfbeta-props}~(iv).
\begin{figure}
    \centering
    $G=\Graph{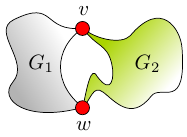} \qquad G'=\Graph{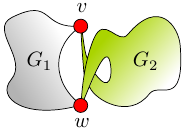}$
    \caption{A Whitney flip of a graph.}%
    \label{fig:Whitney}%
\end{figure}

Whitney showed in \cite{Whitney:2isomorphic} that \emph{all} 2-isomorphisms can be explained with the help of only one further transformation of graphs. This \emph{Whitney flip} or \emph{Whitney switch} is illustrated in \cref{fig:Whitney}:
Suppose that $G=G_1\cup G_2$ is a union of two subgraphs which have only 2 vertices $\{v,w\}=V(G_1)\cap V(G_2)$ in common. Then the Whitney flip of $G$ with respect to this bipartition $E(G) = E(G_1) \sqcup E(G_2)$ is the graph $G'$ obtained from $G$ by re-attaching all half-edges in $G_2$ at $v$ to $w$ and vice-versa.

Since the edges and in fact the half-edges of $G$ and $G'$ are by construction in natural bijection, their 1-chains $C_1(G)= C_1(G_1)\oplus C_1(G_2)=C_1(G')$ are canonically identified. This identification itself does not constitute a 2-isomorphism. For example, a cycle $C=P_1+P_2\in H_1(G)$ consisting of a path $P_1$ in $G_1$ from $v$ to $w$, and a path $P_2$ in $G_2$ from $w$ to $v$, becomes a sum of two paths from $v$ to $w$ in $G'$ so that $C\notin H_1(G')$.

\begin{lemma}\label{lem:pfbeta-flip}
  Let $G'$ be the Whitney flip of $G=G_1\cup G_2$. Define a group homomorphism $\varphi\colon C_1(G)\rightarrow C_1(G')$ as the identity on $C_1(G_1)$ and minus the identity on $C_1(G_2)$, so that
  \begin{equation*}
      C_1(G)\supset \det \Z^{H(e)}\ni e'\wedge e''\mapsto \begin{cases}
          e'\wedge e'' & \text{if $e$ belongs to $G_1$,} \\
          e''\wedge e' & \text{if $e$ belongs to $G_2$.}
      \end{cases}
  \end{equation*}
  Then $\varphi$ is a 2-isomorphism and thus maps cycle bases $\mathcal{C}=(C_1,\ldots,C_\ell)$ of $G$ to cycle bases $\mathcal{C}'=(\varphi(C_1),\ldots,\varphi(C_\ell))$ of $G'$. The respective Pfaffian forms $\pff_{G}=\pff_{\Laplacian_{\mathcal{C}}}$ and $\pff_{G'}=\pff_{\Laplacian_{\mathcal{C}'}}$ are equal, and for any canonical form $\omega$, we have
  \begin{equation}\label{eq:pfbeta-flip}
      \pff_G \wedge \omega_{G} = \pff_{G'} \wedge \omega_{G'}.
  \end{equation}
\end{lemma}
\begin{proof}
    Consider a cycle basis $\mathcal{C} = (\mathcal{C}_1, \mathcal{C}_2, P_1+P_2)$ of $G$ consisting of a cycle basis $\mathcal{C}_1$ of $G_1$, a cycle basis $\mathcal{C}_2$ of $G_2$ and a cycle $P_1+P_2$ built from a path $P_1$ from $v$ to $w$ in $G_1$ and a path $P_2$ from $w$ to $v$ in $G_2$. Then $\mathcal{C}' = (\mathcal{C}_1, -\mathcal{C}_2, P_1-P_2)$, where $-\mathcal{C}_2$ traverses the cycles in $\mathcal{C}_2$ in reverse, and $P_1-P_2$ is indeed a cycle of $G'$. Hence $\mathcal{C}'$ is a cycle basis of $G'$, which shows that $\varphi$ is a 2-isomorphism.

    Now choose any edge directions to identify $C_1(G) =C_1(G') \cong \Z^{E(G_1)} \oplus \Z^{E(G_2)}$. Then $\varphi$ is represented by the diagonal matrix $P = \IdMat_{m_1} \oplus -\IdMat_{m_2}$ where $\IdMat_{m_i}$ is an identity matrix block of size $m_i=|E(G_i)|$. Therefore $\mathcal{C'} = P\mathcal{C}$ and the dual Laplacians
    \begin{equation*}
        \Laplacian_{\mathcal{C}'} = \mathcal{C}^{\Transpose} P^{\Transpose} \cdot \mathcal{D} \cdot P \mathcal{C} = \Laplacian_{\mathcal{C}}
    \end{equation*}
    are identical, since the diagonal matrices $\mathcal{D}$ and $P=P^\Transpose=P^{-1}$ commute.
\end{proof}

By definition, any 2-isomorphism $\varphi$ of graphs $G$ and $G'$ induces isomorphisms $E(G)\cong E(G')$ and $H_1(G)\cong H_1(G')$. It therefore also determines an identification
\begin{equation*}
    \varphi_*\colon \det \Z^{E(G)}\otimes \det H_1(G) \cong \det \Z^{E(G')}\otimes \det H_1(G')
\end{equation*}
such that $e_1\wedge\ldots\wedge e_m\otimes C_1\wedge\ldots\wedge C_\ell \mapsto \varphi(e_1)\wedge\ldots\wedge\varphi(e_m)\otimes \varphi(C_1)\wedge\ldots\wedge \varphi(C_\ell)$.
Following \cref{SS:equivor}, this means we can lift 2-isomorphisms to oriented graphs.
\begin{definition}
  Let $(G, o)$ be a connected oriented graph with even loop number. Let $G'$ denote a Whitney flip of $G$ with respect to some edge-disjoint decomposition $G=G_1\cup G_2$.

  The \textbf{oriented Whitney flip} of $(G,o)$ with respect to $G=G_1\cup G_2$ is the oriented graph $(G', o')$, where $o' = \varphi_* o$ for the 2-isomorphism $\varphi$ from \cref{lem:pfbeta-flip}.
\end{definition}
\begin{remark}
    The Whitney flip is not exactly symmetric under $G_1\leftrightarrow G_2$. Let $G''$ denote the Whitney flip with the roles of $G_1$ and $G_2$ swapped, and call $\psi\colon C_1(G)\cong C_1(G'')$ the corresponding 2-isomorphism: the identity on $C_1(G_2)$ and minus the identity on $C_1(G_1)$. Then $G'\cong G''$ are not identical but only isomorphic (need to swap vertices $v\leftrightarrow w$). Since $\psi=-\varphi$, note $\psi_*=(-1)^{\ell(G)}\varphi_*$. An even loop number therefore ensures that the oriented Whitney flips $(G',\varphi_*o)\sim(G'',\psi_*o)$ are equivalent in $\hGC_3$ via the automorphism $\psi\circ\varphi^{-1}$.
\end{remark}
\begin{lemma}\label{lem:Flip-Or-V}
  If we write the orientation $o = o_V \otimes \bigotimes_e o_e$ in terms of a vertex order $o_V=v_1\wedge\ldots \wedge v_n\in\det\Z^{V(G)}$ and edge directions $o_e\in\det \Z^{H(e)}$ as in \cref{def:GCor}, then the orientation of the Whitney flip is $o' = (-1)^{\ell(G_2)+1} o_V \otimes \bigotimes_e o_e$.
\end{lemma}
\begin{proof}
  Consider a cycle basis $\mathcal{C} = (\mathcal{C}_1, \mathcal{C}_2, P_1+P_2)$ of $G$ and $\mathcal{C}'=(\mathcal{C}_1,-\mathcal{C}_2,P_1-P_2)$ for $G'$ as in the proof of \cref{lem:pfbeta-flip}.
  Let $o_V = v \wedge w \wedge o_{V_1} \wedge o_{V_2}$ where $o_{V_i}$ is any ordering of the vertices of $G_i \setminus \{v,w\}$. Then by \cref{SS:equivor}, $o_V \otimes \bigotimes_e o_e = (\det A ) \cdot o$ where
  \begin{equation*}
    A = \begin{pmatrix} \mathcal{C}_1 & \mathcal{C}_{2} & P_1+P_2 & P_w & \mathcal{Q}_{1} & \mathcal{Q}_{2} \end{pmatrix}
  \end{equation*}
  for any path $P_w$ from $v$ to $w$ and $\mathcal{Q}_i$ denoting matrices of paths from $v$ to the vertices in $G_i \setminus \{v,w\}$, in the order of $o_{V_i}$.
  We are free to choose $P_w=P_1$ and we may assume that the paths in $\mathcal{Q}_i$ are entirely contained in $G_i$.

  For $o'$ the isomorphism from \cref{SS:equivor} gives $o_V \otimes \bigotimes_e o_e = (\det{A'}) \cdot o'$ for the matrix
  \begin{equation*}
    A' = \begin{pmatrix} \mathcal{C}_1 & -\mathcal{C}_{2} & P_1-P_2 & P_1 & \mathcal{Q}_{1} & \mathcal{Q}_{2}' \end{pmatrix}
  \end{equation*}
  where $\mathcal{Q}_2'$ denotes the matrix obtained by adding $P_1$ to every column of $\mathcal{Q}_2$ (in $G'$, paths from $\mathcal{Q}_2$ originate at vertex $w$, not $v$). We conclude that $\det{A'} = (-1)^{\ell(G_2)+1} (\det{A})$.
\end{proof}
\begin{figure}
    \centering
    $ G = \Graph[0.8]{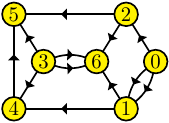} \qquad\qquad G' = \Graph[0.8]{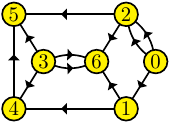} $
    \caption{A graph $G$ that is isomorphic to its Whitney flip $G'$ along the cut $\{1,2\}$.}%
    \label{fig:244-flip}%
\end{figure}
\begin{example}\label{ex:oddflip}
  The graph $G$ denoted $G_{244}$ in \cref{tab:integrals} is illustrated in \cref{fig:244-flip}, together with its Whitney flip $G'$ taking $G_2$ to be the subgraph induced by vertices $\{0,1,2\}$. Since $\ell(G_2)+1=2$, \cref{lem:Flip-Or-V} shows that with the orientations (edge directions and vertex ordering) as drawn, $G'$ is the oriented Whitney flip of $G$.

  In this example, the flip $G'$ happens to be isomorphic to $G$ itself, by swapping the vertices $1\leftrightarrow 2$ and $4\leftrightarrow 5$. This isomorphism reverses the edge $4\rightarrow 5$, so that $G'=-G$ in $\hGC_3$.
  We thus say that $G$ has an odd Whitney flip to itself.
\end{example}

\subsection{Vanishing}
The above graph-theoretic identities imply that under certain conditions, forms $\pff_G\wedge\omega_G$ on the simplex $\oFS_G$ that have top degree (equal to $\dim\oFS_G=m-1$) necessarily vanish:
\begin{lemma}\label{lem:pfbeta-vanish}
  Let $G$ be a graph with $m$ edges, even loop number $\ell$, and no isolated vertices.
  Let $\omega \in \cfs[m - \ell - 1]$.
  Then $\pff_G \wedge \omega_G=0$ if any of the following hold:
  \begin{enumerate}
    \item $G$ has a self-loop
    \item $G$ is disconnected, or $G$ has a vertex whose deletion disconnects $G$
    \item $G$ has a 2-valent vertex
    \item $G$ has a 2-edge cut (two edges whose deletion disconnects $G$)
    \item $G$ has degree $k=\gcdeg{G} > -3$ and is not the single edge ($G\not\cong D_1$)
  \end{enumerate}
\end{lemma}
Note the second part of (ii) includes graphs $G\not\cong D_1$ with bridges and 1-valent vertices.
\begin{proof}
  First note that the form $\pff_G \wedge \omega_G$ has degree $m-1$.

  Case (i) is immediate from \cref{cor:pfform} (every self-loop is a biconnected component).

  For case (ii), suppose that $G=G_1\cup G_2$ is the union of two subgraphs sharing at most one vertex. Then $\ell(G)=\ell(G_1)+\ell(G_2)$, and the edge partition gives $m=m_1+m_2$ when $G_i$ has $m_i$ edges.
  Expand $\pff_{G} \wedge \omega_G$ as in \eqref{pffgraph-blocks} and let $d_i$ denote the degrees of $\omega^{(i)}_{G_i}\in\cfs$.
  Since $\pff_{G_i} \wedge \omega^{(i)}_{G_i}$ is a projective holomorphic form of degree $\ell(G_i) + d_i$ in $m_i$ variables, we must have $d_i \leq m_i-\ell(G_i)-1$ or else this form is trivially zero. But since $d_1+d_2=|\omega|=m-\ell-1$, both inequalities cannot hold simultaneously, and therefore every summand of $\pff_G\wedge\omega_G$ in \eqref{pffgraph-blocks} is zero.

  For case (iii), let $G'$ be the graph such that $G$ is obtained from $G'$ by subdividing the edge $e$.
  By \cref{lem:pfbeta-props} (iii), $\pff_G \wedge \omega_G = s^*_e \left(\pff_{G'} \wedge \omega_{G'}\right)$ is the pullback of the projective form $\pff_{G'} \wedge \omega_{G'}$ of degree $m-1$ in $m-1$ variables, which must thus be zero.

  For case (iv), suppose $G$ has a 2-edge cut given by edges $\{e,f\}$.
  Pick an endpoint of $e$ and an endpoint of $f$, say vertices $\{v,w\}$, such that they lie in different connected components of $G \setminus \{e,f\}$.
  In particular $\{v,w\}$ is a 2-vertex cut of $G$ and taking a Whitney flip with respect to this cut gives a graph $G'$ with a two-valent vertex adjacent to $\{e,f\}$.
  By \cref{lem:pfbeta-flip}, $\pff_G \wedge \omega_G = \pff_{G'} \wedge \omega_{G'}$, which then vanishes by (iii).

  For (v) recall that $\hGC_3$ is supported in degrees $\leq-3$, so $G$ must be disconnected or have some vertex $v$ with degree less than three. By (ii) and (iii) we may assume that $G$ is connected and free of $2$-valent vertices, so $v$ has degree one. If the neighbour $w$ of $v$ has degree larger than one, then its removal disconnects $G$ and we conclude by (ii). The only remaining case is that $w$ has degree one, which by connectedness implies $G\cong D_1$.\footnote{By convention, $\pff_G = 1$ for graphs $G$ with $\ell = 0$.}
\end{proof}

\begin{remark}\label{rem:bicon-qiso}
It was shown in~\cite{cut} that the quotient of $\hGC_3$ by graphs with 1-vertex cuts is quasi-isomorphic to $\hGC_3$.\footnote{Dually, $\cGC_3$ is quasi-isomorphic to the subcomplex of $\cGC_3$ spanned by biconnected graphs \cite[Appendix F]{Willwacher:GCgrt}, i.e.\ graphs without cut-vertices (also called ``one-vertex irreducible'' (1VI) graphs).}
Property (ii) above ensures that canonical integrals factor through this quotient. Hence, in order to compute the pairing between a cohomology class represented by a canonical integral, and a homology class, it suffices to know a representative of the homology class only modulo graphs with cut-vertices.
\end{remark}

\section{Canonical integrals on the odd graph complex}

In this section we integrate Pfaffian forms to define a bilinear pairing
\begin{equation*}
    \hGC_3\otimes \cfs \longrightarrow \R,\qquad (X,\omega) \mapsto I_X(\omega).
\end{equation*}
For each canonical form $\omega\in\cfs$, we can interpret the linear function
\begin{equation}
    I(\omega)\colon \hGC_3 \longrightarrow \R, \qquad X\mapsto I_X(\omega)
\end{equation}
as an element of the cochain complex $\cGC_3\ctp\R=\Hom(\hGC_3,\R)$ via \eqref{eq:GC-duality}. In suitable loop orders, these cochains are closed, due to Stokes' theorem, hence we obtain cohomology classes of the odd graph complex. We prove the convergence (\cref{thm:intro-convergence}) and Stokes' relations (\cref{thm:intro-Stokes}) in \cref{sec:GC-convergence} and \cref{sec:GC-stokes}, respectively. The vanishing (\cref{lem:intro-gr-canint=0}) and cocycles (\cref{cor:intro-cocycles}) are discussed in \cref{sec:cohom-classes}, and the proof of the non-triviality of the class defined by $I(\beta^5)$ (\cref{thm:intro-gr6I(beta5)<>0}) is given in \cref{sec:gr6H-6}.
\begin{definition}\label{def:canint}
    Given an oriented graph $(G,o)\in\hGC_3$ with edges $E=\{1,\ldots,m\}$ and $\ell$ loops, pick any cycle basis $\mathcal{C}=(C_1,\ldots,C_\ell)$ such that $o=e_1\wedge\ldots\wedge e_m\otimes C_1\wedge\ldots\wedge C_\ell$ in $\det \Z^E \otimes \det H_1(G)$. Then for any canonical form $\omega$, we set (with $\pff_G=\pff_{\Laplacian_{\mathcal{C}}}$)
  \begin{equation}\label{eq:canint}
     I_{(G,o)}(\omega) \defas \frac{1}{(-2\pi)^{\ell/2}}\int_{\oFS_G} \pff_{G} \wedge \omega_G.
  \end{equation}
\end{definition}
To justify this definition, we will show in \cref{thm:convergence} that the integrals \eqref{eq:canint} are absolutely convergent. Furthermore, note that if $\alpha$ is an automorphism of $G$, then by pulling back the integration variables under the corresponding edge permutation, we have
\begin{equation*}
    I_{(G,o')}(\omega)=\int_{(\sgn \alpha)\oFS_G} (\det P) \pff_G\wedge\omega_G
    = \frac{o'}{o} I_{(G,o)}(\omega)
\end{equation*}
due to \eqref{eq:volform-sign} and \eqref{eq:pffcan-sign}. Here $\sgn\alpha$ and $\det P$ denote the action of $\alpha$ on the orientations of edges $\det \Z^E$, and cycles $\det H_1(G)$, respectively, so that $o'=o\cdot(\sgn\alpha)\cdot\det P$. Therefore, $I(\omega)$ is indeed well-defined as a function on the graph complex $\hGC_3$.

In the simplest case $\omega=1$, we can compute the integrals for all graphs explicitly:
\begin{example}\label{ex:dipole-integrals}
For dipole graphs $D_{2i+1}$ with even loop number (oriented as in \cref{ex:dipoleor}), the integral of the Pfaffian form is absolutely convergent and evaluates to
\begin{equation}\label{eq:dipole-integrals}
  I_{D_{2i+1}}(1)=\frac{1}{(-2\pi)^i} \int_{\oFS_{D_{2i+1}}} \pff_{D_{2i+1}} = 1.
\end{equation}
\end{example}
\begin{proof}
    By \cref{ex:dipoleor}, the dipole orientation is represented by the edge ordering $e_1\wedge\ldots\wedge e_{2i+1}$ and the cycle basis from \cref{lem:pff-dipoles}. Hence, apart from the factor $(-1)^i (2i)!/(2^i\cdot i!)=(-1)^i\cdot 1 \cdot 3 \cdots (2i-1)$ in front of \eqref{eq:pff-dipoles}, the integral is
    \begin{equation*}
      \int_{\oFS_{2i+1}} \frac{(x_1\cdots x_{2i+1})^{i-1} \Omega_{2i+1}}{\Symanzik^{i+1/2}}
        =
        \int_{\oFS_{2i+1}} \frac{(x_1\cdots x_{2i+1})^{-1/2} \Omega_{2i+1}}{(x_1+\cdots + x_{2i+1})^{i+1/2}}
    \end{equation*}
    by changing variables $x_e\mapsto 1/x_e$. In the chart $x_1+\ldots+x_{2i+1}=1$ we get $\Omega_{2i+1}=\td x_1\wedge\ldots\wedge\td x_{2i}$ and the integral becomes the beta function of several variables,
    \begin{equation*}
    \int_{\substack{x_1,\ldots, x_{2i}>0 \\ x_1+\ldots+x_{2i}<1}} \frac{\td x_1\cdots\td x_{2i}}{\sqrt{x_1\cdots x_{2i}(1-x_1-\ldots-x_{2i})}}
    =
    \frac{\Gamma(1/2)^{2i+1}}{\Gamma(i+1/2)}
        =\frac{(2\pi)^i}{1\cdot 3 \cdots (2i-1)}
    \end{equation*}
    which combines with the factor $(-1)^i\cdot 1\cdot 3\cdots (2i-1)$ as claimed.
\end{proof}

\begin{corollary}\label{lem:pffint=MC}
    For every oriented graph $G$, the integral $I_G(1)$ is absolutely convergent. For all $X\in\hGC_3$, we have $I_X(1)=\inner{X,\MCdip}$, where $\MCdip\in\cGC_3$ denotes the cochain
    \begin{equation}\label{eq:MCdip}
        \MCdip \defas \sum_{i=1}^{\infty} \frac{D_{2i+1}}{2 \cdot(2i+1)!}.
    \end{equation}
\end{corollary}
In other words, $I_G(1)=0$ unless $G$ is isomorphic to a dipole with even loop number, in which case $I_G(1)=\pm 1$ depending on the orientation of $G$.
\begin{proof}
The degree of $\pff_G$ is the loop number $\ell$, and the dimension of $\oFS_G$ is one less than the edge number $m$. Thus the integral is trivially zero if $m\neq \ell+1$ or if $\ell$ is odd. The integrand also vanishes by \cref{lem:pfbeta-vanish} if $G$ is disconnected or has a self-loop. The only remaining graphs have $m-\ell+1=2$ vertices and are thus dipoles (with even loop number). The normalization is given by \cref{ex:dipole-integrals} and $\abs{\Aut(D_{2i+1})}=2\cdot (2i+1)!$.
\end{proof}
More generally, for a canonical form $\omega\in\cfs[d]$ of degree $d$ and a connected graph $G$ with $n$ vertices, $m$ edges and $\ell=m-n+1$ loops, the integrand $\pff_G\wedge\omega_G$ has degree $\ell+d$. This is equal to the dimension $m-1$ of the simplex $\oFS_G$ only if $n=d+2$.
\begin{corollary}\label{cor:deg-match-canint}
    Let $G$ be an oriented graph with $\ell$ loops and $n$ vertices, and let $\omega\in\cfs[d]$ be a canonical form of degree $d$. Then $I_G(\omega)=0$ if $n\neq d+2$, or if $G$ satisfies any of the conditions in \cref{lem:pfbeta-vanish}, or if $G$ has an odd automorphism.
\end{corollary}
\begin{remark}\label{rem:flip-vanishing}
    By \cref{lem:pfbeta-flip} we also have $I_G(\omega)=0$ whenever $G$ is isomorphic to an odd Whitney flip of itself, or when $G$ has a Whitney flip to a graph with an odd automorphism.

    The vanishing of these integrals does not imply the vanishing of the integrands $\pff_G\wedge\omega_G$.

    Indeed, for $G=G_{244}$ from \cref{ex:oddflip}, the odd flip to itself explains why $\tau_1(G_{244}) =I_G(\beta^5)= 0$ in \cref{tab:integrals}. The integrand $\pff_G\wedge\omega_G^5$ of $\tau_1(G_{244})$, however, is non-zero. In the form \eqref{eq:b5pf-psi5/2}, it has numerator $Q_G=10(x_1+x_2)(x_9-x_8)(x_{10}+x_{11})$.

    Similarly, the graph $G_{97}$ in \cref{tab:integrals} has a Whitney flip to a graph with an odd automorphism, which explains why $\tau_1(G_{97})=0$. Its integrand, however, again is non-zero.
\end{remark}
\begin{corollary}
	For a homogeneous form $\omega$, the cochain $I(\omega)$ is supported on graphs whose $\hGC_3$-degree is related to the loop number by $\gcdeg{G}=1+\abs{\omega}-2\ell(G)$. Therefore, the relative signs in the Lie bracket of two such cochains are determined by the forms alone:
	\begin{equation}\label{eq:canint-Lie}
		\left[ I(\omega'), I(\omega'') \right]
		=I(\omega')\circ I(\omega'') - (-1)^{(1+|\omega'|)(1+|\omega''|)} I(\omega'')\circ I(\omega').
	\end{equation}
\end{corollary}

\subsection{Compactification}\label{sec:GC-convergence}
The forms $\pff_G\wedge\omega_G$ are smooth on the open simplex $\oFS_G\subset\Pro(\R^E)$, but they can diverge on its boundary $\partial\cFS_G=\{x_1\cdots x_m=0\}$ in $\Pro(\R^E)$, because this boundary intersects the polar locus where $\Symanzik_G=0$: For any subset of edges $\gamma\subset E$, the linear subspace
\begin{equation*}
    L_\gamma\defas \bigcap_{e\in\gamma} \{x_e=0\} \subset \Pro(\R^E)
\end{equation*}
is entirely contained in $\{\Symanzik_G=0\}$ whenever $\ell(\gamma)>0$, i.e.\ when the subgraph $\gamma$ contains a loop. Therefore, the forms $\pff_G\wedge\omega_G$ do not extend to smooth forms on the closed simplex $\cFS_G$, and hence the convergence of the integrals \eqref{eq:canint} is not obvious.

Following \cite{BlochEsnaultKreimer:MotivesGraphPolynomials}, we associate to a graph $G$ with edges $E$ an iterated blowup
\begin{equation*}
    \pi\colon P^G\rightarrow \Pro(\R^E)
\end{equation*}
of the arrangement of linear subspaces $L_{\gamma}$ indexed by the bridgeless subgraphs $\gamma\subset E$. Bridgeless means that $\ell(\gamma\setminus\{e\})<\ell(\gamma)$ for every edge $e\in\gamma$; in other words, every edge of $\gamma$ lies in a cycle of $\gamma$. Note that bridgeless subgraphs are not necessarily connected.

The exceptional divisors $D_{\gamma}\subset P^G$ of these blowups, together with the strict transforms $D_e$ of the coordinate hyperplanes $\{x_e=0\}\subset \Pro^{E-1}$, form a divisor $D\subset P^G$:
\begin{equation*}
    D=\bigcup_{\gamma\ \text{bridgeless}} D_{\gamma} \quad\cup\quad \bigcup_{e\in E} D_e.
\end{equation*}
This divisor has simple normal crossings. There are canonical isomorphisms
\begin{equation*}
    D_{\gamma}\cong P^{\gamma}\times P^{G\sslash\gamma}
\end{equation*}
where $G \sslash \gamma$ denotes the graph obtained from $G$ by contracting each connected component of $\gamma$ to a single vertex (so if $\gamma$ is a connected graph, then $G\sslash\gamma=G/\gamma)$. This also applies to single edges $\gamma=\{e\}$, in which case $P^{\gamma}$ is a point, so that $D_e\cong P^{G/e}$.

The closure of the preimage of the simplex is a compact manifold with corners
\begin{equation}\label{eq:FeynmanP}
    \FP_G=\overline{\pi^{-1}(\oFS_G)}\subset P^G,
\end{equation}
called the \emph{Feynman polytope} in \cite{Brown:FeynmanAmplitudesGalois}. The facets of its boundary $\partial\FP_G\subset D$ are in bijection with the irreducible components of $D$. The facet in $D^{\gamma}\cong P^\gamma\times P^{G\sslash\gamma}$ is
\begin{equation}\label{eq:FeynmanP-facet}
        \partial \FP_G \cap D_\gamma = \FP_G \cap D_{\gamma} \cong \FP_{\gamma}\times\FP_{G\sslash\gamma}
\end{equation}
for every $\gamma$ that is bridgeless. This identity also holds for single edges $\gamma=\{e\}$, since then the factor $\FP_\gamma$ is a point and indeed we have a canonical isomorphism $\partial\FP_G\cap D_e=\FP_{G/e}$.

\begin{thm}\label{thm:convergence}
    The integral \eqref{eq:canint} converges absolutely for any graph $G$ and any $\omega\in\cfs$.
\end{thm}
\begin{proof}
    It was shown in \cite[\S 7.3]{invariant} that the pullback $\pi^*(\omega_G)$ of a canonical form extends smoothly over the boundary of the polytope $\FP_G$. However, this is not the case for $\pi^*(\pff_G)$, which develops (integrable) square root singularities on $\partial\FP_G$.

    We thus change coordinates to the square roots of edge lengths, by the covering map
    \begin{equation*}
     \sqX\colon \Pro(\R^m)\rightarrow\Pro(\R^m),\quad
    [x_1:\cdots:x_m]\mapsto[x_1^2:\cdots:x_m^2].
    \end{equation*}
    Let $\widetilde{\pff}_G=\pi^*(\sqX^* \pff_G)$ and $\widetilde{\omega}_G=\pi^*(\sqX^* \omega_G)$ denote the pullbacks to the blowup $P^G$. We claim that they extend smoothly over the boundary of the compact polytope $\FP_G$. This implies the absolute convergence of $\int_{\FP_G}\widetilde{\pff}_G\wedge\widetilde{\omega}_G$, which thus equals the improper integrals
    \begin{equation*}
        \int_{\FP_G\setminus\partial\FP_G} \widetilde{\pff}_G \wedge\widetilde{\omega}_G
        = \int_{\oFS_G} \pff_G\wedge\omega_G
    \end{equation*}
    since the change of coordinates $\sqX\circ\pi$ restricts to a diffeomorphism of the interior of the polytope with the open simplex: $\FP_G\setminus\partial\FP_G\cong \oFS_G$.

    To establish smoothness near a facet $\partial\FP_G\cap D_\gamma$, consider coordinates on $P^G$ given by
    \begin{equation}\label{eq:facet-chart}
        x_e=\begin{cases}
            \lambda y_e & \text{if $e\in\gamma$,} \\
            z_e  & \text{else,}
        \end{cases}
    \end{equation}
    such that $D_\gamma=\{\lambda=0\}$. By the exact sequence $0\rightarrow H_1(\gamma)\rightarrow H_1(G)\rightarrow H_1(G\sslash\gamma)\rightarrow 0$, every cycle basis $\mathcal{C}=(\mathcal{C}',\mathcal{C}'')$ of $G$ that begins with a cycle basis $\mathcal{C}'$ of $\gamma$ determines a cycle basis of $G\sslash\gamma$, which we denote $\mathcal{C}''\sslash\gamma$, by forgetting from $\mathcal{C}''$ all rows corresponding to edges in $\gamma$.
    In such a basis $\mathcal{C}$, the dual Laplacian of $G$ has the block form
    \begin{equation*}
        \Laplacian_{\mathcal{C}}(x) =
        \begin{pmatrix}
            \lambda A & \lambda B \\
            \lambda B^\Transpose & C \\
        \end{pmatrix}
    \end{equation*}
    where $A=\Laplacian_{\mathcal{C}'}(y)$ is a dual Laplacian of $\gamma$, and $C\equiv \Laplacian_{\mathcal{C}''\sslash\gamma}(z) \mod \lambda$ becomes the dual Laplacian of the quotient graph $G\sslash\gamma$ on $D_\gamma$. Note that \cref{ex:pff2-asy} shows that $\pi^*\pff_G$ can indeed have singularities of the form $(\td\lambda)/\sqrt{\lambda}=2\td(\sqrt{\lambda})$ on $D_\gamma$.

    These singularities disappear by considering the blowup $P^G\rightarrow \Pro(\R^E)$ not on the space of edge lengths $x_e$, but rather on the space of their square roots. Indeed, by \cref{lem:pff-asy},
    \begin{equation*}
        \sqX^*\Laplacian_{\mathcal{C}}(x) =
        \begin{pmatrix}
            \lambda^2 \sqX^*A & \lambda^2 \sqX^*B \\
            \lambda^2 \sqX^*B^\Transpose & \sqX^* C \\
        \end{pmatrix}
    \end{equation*}
    is smooth at $\lambda=0$. Thus $\sqX^* \pff_G$ extends smoothly to the interior of the facet $\partial\FP_G\cap D_\gamma$. Since $D$ has normal crossings, this already implies that $\sqX^*\pff_G$ extends smoothly over the entire boundary of $\FP_G$, i.e.\ also over faces of higher codimension. Here we use that $\FP_G$ does not meet the strict transform of the polar locus $\{\Symanzik_G=0\}$, see \cite[Theorem~6.7]{Brown:FeynmanAmplitudesGalois}.

    In more detail, for any chain $\emptyset\neq\gamma_1\subsetneq\ldots\subsetneq \gamma_r\subsetneq \gamma_{r+1}=G$ of bridgeless subgraphs, $\FP_G$ has a face of codimension $r$, contained in $D_{\gamma_1}\cap\ldots\cap D_{\gamma_r}\cong P^{\gamma_1}\times P^{\gamma_2\sslash \gamma_1}\times\cdots\times P^{G\sslash\gamma_r}$. This face is visible in the chart of the iterated blowup $P^G$ defined by
    \begin{equation*}
        x_e = \lambda_i\cdots\lambda_r y_e
        \quad\text{for}\quad e\in \gamma_i\setminus\gamma_{i-1}.
    \end{equation*}
    In this chart, $D_{\gamma_i}=\{\lambda_i=0\}$ for each $1\leq i\leq r$. The denominator of $\sqX^*\pff_G$ becomes
    \begin{equation*}
        \sqX^*\sqrt{\Symanzik_G}=\lambda_1^{\ell(\gamma_1)}\cdots\lambda_r^{\ell(\gamma_r)}\cdot \sqX^*\sqrt{
        \Symanzik_{\gamma_1}(y)\Symanzik_{\gamma_2\sslash\gamma_1}(y)\cdots\Symanzik_{G\sslash\gamma_r}(y) + R(\lambda,y)
        }
    \end{equation*}
    where the polynomial $R$ has positive coefficients (recall \cref{def:Symanzik}) and vanishes if any $\lambda_i=0$. Singularities of $\sqX^*\pff_G$ on the interior (all $y_e>0$) of the face $\FP_G\cap D_{\gamma_1}\cap\ldots\cap D_{\gamma_r}$ could therefore only arise from poles in $\lambda_i$. The earlier discussion of facets 
    shows that such poles are absent for each $i$.

    For the smoothness of $\widetilde{\omega}_G$, it suffices to consider the primitive generators $\omega=\beta^{4i+1}$ of the algebra $\cfs$. The smoothness of $\pi^*(\sqX^* \beta_G^{4i+1})$ on $\FP_G$ then follows by the same argument as before, with \cref{lem:can-asy} standing in for \cref{lem:pff-asy}.
\end{proof}

\subsection{Stokes' relations}\label{sec:GC-stokes}

\begin{thm}\label{thm:Stokes}
	In Sweedler notation $\sum_{(\omega)} \omega' \otimes \omega'' = \Delta\omega -1\otimes\omega-\omega\otimes 1$, we have for every canonical form $\omega$ the relation ($\MCdip\in\cGC_3$ is defined in \cref{lem:pffint=MC})
\begin{equation}\label{eq:Stokes}
	  0 = \delta I(\omega)
	    + [I(\omega), \MCdip]
	    + \frac{1}{2}\sum_{(\omega)} (-1)^{|\omega'|} \left[I(\omega''), I(\omega')\right].
\end{equation}
\end{thm}
\begin{proof}
	Above we showed that $\widetilde{\pff}_G\wedge\widetilde{\omega}_G=\pi^*(\sqX^* (\pff_G\wedge\omega_G))$ is a smooth form on the Feynman polytope $\FP_G$. Recall that $\FP_G$ has two types of facets: those labelled by edges (coming from the boundary facets $\partial \cFS_G$ of the original projective simplex), and those coming from blowups, labelled by bridgeless subgraphs.

	Since $\pff_G$ and $\omega_G$ are closed, Stokes' theorem applied to $\FP_G$ thus says
	\begin{equation*}
		0 = \int_{\FP_G} \td (\widetilde{\pff}_G\wedge\widetilde{\omega}_G)
		= \sum_e \int_{\partial\FP_G\cap D_e}(\widetilde{\pff}_G\wedge\widetilde{\omega}_G)\Big|_{D_e}
		+\sum_{\gamma\ \text{bridgeless}} \int_{\partial\FP_G\cap D_\gamma} (\widetilde{\pff}_G\wedge\widetilde{\omega}_G)\Big|_{D_\gamma}
		.
	\end{equation*}
	Topologically, the facet of $\FP_G$ in $D_\gamma$ is the product $\FP_\gamma\times \FP_{G\sslash\gamma}$, see \eqref{eq:FeynmanP-facet}. To determine its orientation, write $m_\gamma=|E(\gamma)|$ for the number of edges in $\gamma$, and suppose that they come first in the edge ordering of $G$. Then in the chart \eqref{eq:facet-chart}, the volume form \eqref{eq:simplex-volform} becomes\footnote{%
	As a polynomial form, $\Omega_G(x)$ also has a term $\lambda^{m_\gamma}\Omega_\gamma(y) \bigwedge_{e\notin\gamma} \td z_e$, but this is zero since the chart \eqref{eq:facet-chart} is understood with $y_i=1$ and $z_j=1$ for some choice of $i$ and $j$. That is, $y$ and $z$ are affine charts on $\Pro\big(\R^{E(\gamma)}\big)$ and $\Pro\big(\R^{E(G\sslash\gamma)}\big)$, respectively. Thus, $\td z_{m_\gamma+1}\wedge\ldots\wedge \td z_{m}=0$.
}
	\begin{equation*}
		\Omega_G(x)=
		\lambda^{m_\gamma-1}\td \lambda \wedge \Omega_\gamma(y)\wedge(-1)^{m_\gamma}\Omega_{G\sslash\gamma}(z).
	\end{equation*}
	Contracting with the outward normal vector field $-\partial/\partial\lambda$, the induced orientation of the facet (which is implied in Stokes' theorem) is represented by the volume form $(-1)^{m_\gamma-1} \Omega_\gamma(y)\wedge\Omega_{G\sslash\gamma}(z)$. So in terms of fundamental classes of oriented manifolds,
	\begin{equation}\label{eq:facet-orientation}
		\big[\partial\FP_G\cap D_\gamma\big]=(-1)^{m_\gamma-1} \cdot \big[\FP_\gamma\big] \times \big[\FP_{G\sslash\gamma}\big].
	\end{equation}
	This identity holds for Feynman polytopes oriented according to any edge orderings $o\in(\det \Z^{E(G)})^\times$, $o'\in(\det\Z^{E(\gamma)})^\times$, and $o''\in(\det\Z^{E(G/\gamma)})^\times$, provided that $o=o'\wedge o''$.

	Restricting to $D_e$ amounts to setting $x_e=0$. So for any cycle basis $\mathcal{C}$ of $G$,
	\begin{equation*}
		\left.\Laplacian_{\mathcal{C}}\right|_{x_e=0}
		=\Laplacian_{\mathcal{C}/e}
	\end{equation*}
	in terms of the cycle basis $\mathcal{C}/e$ of $G/e$ induced by $H_1(G)\cong H_1(G/e)$. Here we may assume $H_1(\{e\})=0$ ($G$ has no self-loops), as otherwise $\widetilde{\pff}_G=0$. We conclude that
	\begin{equation*}
		\int_{\partial\FP_G\cap D_e}(\widetilde{\pff}_G\wedge\widetilde{\omega}_G)\Big|_{D_e}
		=\int_{\FP_{G/e}} \widetilde{\pff}_{G/e} \wedge \widetilde{\omega}_{G/e}
		= I_{G/e}(\omega).
	\end{equation*}
	The sum over edge-type boundaries $\partial\FP_G\cap D_e$ in Stokes' theorem thus produces the first term
	$
		(\delta I(\omega))(G) = I_{\partial G}(\omega)=\sum_e I_{G/e}(\omega)
	$
	in \eqref{eq:Stokes}.

	Restricting to $D_\gamma$ for a bridgeless subgraph $\gamma$, it follows from \cref{lem:can-asy,lem:pff-asy} (see the proof of \cref{thm:convergence}) that
	\begin{equation*}
		\widetilde{\pff}_G\Big|_{D_\gamma} = \widetilde{\pff}_\gamma\wedge\widetilde{\pff}_{G\sslash\gamma}
		\quad\text{and}\quad
		\widetilde{\beta}^{4i+1}_G\Big|_{D_\gamma}=\widetilde{\beta}_\gamma^{4i+1}+\widetilde{\beta}_{G\sslash\gamma}^{4i+1},
	\end{equation*}
	provided that $\pff_G$, $\pff_\gamma$, and $\pff_{G\sslash\gamma}$ are computed with respect to cycle bases $\mathcal{C}=(\mathcal{C}',\mathcal{C}'')$, $\mathcal{C}'$, and $\mathcal{C}''\sslash\gamma$ respectively. The restriction of $\beta_G^{4i+1}$ implies for arbitrary polynomials $\omega\in\cfs$ in primitive canonical forms, that
	\begin{equation*}
		\widetilde{\omega}_G\Big|_{D_\gamma}
		= \widetilde{\omega}_\gamma + \widetilde{\omega}_{G\sslash\gamma}
		+\sum_{(\omega)} \widetilde{\omega}'_{\gamma}\wedge\widetilde{\omega}''_{G\sslash\gamma}
		.
	\end{equation*}
	Using \eqref{eq:facet-orientation}, we can therefore identify the terms from facets $D_\gamma$ in Stokes' formula as
	\begin{equation*}
\int_{\partial\FP_G\cap D_\gamma} (\widetilde{\pff}_G\wedge\widetilde{\omega}_G)\Big|_{D_\gamma}
= \sum_{(\omega)} (-1)^{m_\gamma-1} I_{\gamma}(\omega_{(1)}) I_{G\sslash\gamma}(\omega_{(2)})
	\end{equation*}
	where $\sum_{(\omega)} \omega_{(1)}\otimes \omega_{(2)}=1\otimes\omega+\omega\otimes 1+\sum_{(\omega)} \omega'\otimes \omega''$ denotes the full coproduct. Now recall that $I_\gamma(\omega')=0$ if the dimension $m_\gamma-1$ of $\FP_\gamma$ differs from the degree $\ell(\gamma)+|\omega'|$ of the integrand. Also, $I_\gamma(\omega')$ vanishes if $\ell(\gamma)$ is odd, hence we can replace $(-1)^{m_\gamma-1}$ by $(-1)^{|\omega'|}$.
	Summing over $\gamma$, the contributions from all $D_\gamma$ boundaries to Stokes' formula is
	\begin{equation*}\label{eq:Stokes-as-insertion}\tag{$\ast$}
\sum_{(\omega)} (-1)^{|\omega_{(1)}|} \sum_{\gamma\ \text{bridgeless}} I_{\gamma}(\omega_{(1)}) I_{G\sslash\gamma}(\omega_{(2)})
=\sum_{(\omega)} (-1)^{|\omega_{(1)}|} \Big(I(\omega_{(2)}) \circ I(\omega_{(1)})\Big)(G).
	\end{equation*}
	Here we exploit that $I_{\gamma}(\omega_{(1)})=0$ for subgraphs that have a bridge or are disconnected, hence we may equally sum over \emph{all} (not just bridgeless) subgraphs $\gamma$, and may replace $G\sslash\gamma=G/\gamma$ as if $\gamma$ were always connected, to invoke \cref{thm:pairing}. In fact, we defined the orientations of $\gamma$, $G/\gamma$ and $G$ on the left-hand side of \eqref{eq:Stokes-as-insertion} in terms of edges and cycle bases (subgraph $\gamma$ comes first), whereas on the right-hand side the orientation of the insertion $G/\gamma\circ\gamma$ is defined according to \cref{sec:GC-cohom} by keeping edge directions and ordering the vertices of $\gamma$ to come first (and the insertion vertex $v_*$ is first in $G/\gamma$). Thus tacitly we have used in \eqref{eq:Stokes-as-insertion} that these two ways to relate orientations on $G$ with pairs of orientations of $\gamma$ and $G/\gamma$ are equivalent under our identification from \cref{SS:equivor}.
 
 To see this, let $n'$ denote the number of vertices in $\gamma$. Pick paths $\mathcal{Q}'=(P_2,\ldots,P_{n'})$ from the first vertex $v_1$ in $\gamma$ to the others vertices of $\gamma$, such that each of these paths is contained in $\gamma$ (recall we may assume $\gamma$ connected). Let $\mathcal{Q}''=(P_{n'+1},\ldots,P_n)$ denote paths in $G$ from $v_1$ to the vertices not in $\gamma$. We get an induced list $\mathcal{Q}''/\gamma$ of paths in $G/\gamma$ starting at $v_*$, by skipping from the paths all edges that belong to $\gamma$. Then with respect to an edge ordering of $G$ where the edges of $\gamma$ come first, the determinant
 \begin{equation*}
     \det\begin{pmatrix}
         \mathcal{C}' & \mathcal{C}'' & \mathcal{Q}' & \mathcal{Q}''
     \end{pmatrix}
     = (-1)^{(n'-1)\ell(G/\gamma)}
     \det\begin{pmatrix}
         \mathcal{C'} & \mathcal{Q}'
     \end{pmatrix}
     \cdot
     \det\begin{pmatrix}
         \mathcal{C}''/\gamma & \mathcal{Q}''/\gamma
     \end{pmatrix}
 \end{equation*}
 factorizes because the matrix $(\mathcal{C}'\ \mathcal{Q}'\ \mathcal{C}''\ \mathcal{Q}'')$ is upper block triangular. Since $I_{G\sslash\gamma}(\omega_{(2)})=0$ whenever $G/\gamma$ has odd loop order, the factor $(-1)^{(n'-1)\ell(G/\gamma)}$ relating the orientations on the left- and right-hand sides of \eqref{eq:Stokes-as-insertion} can be omitted.
 
	Exploiting the graded cocommutativity \eqref{eq:graded-cocom}, we can rewrite \eqref{eq:Stokes-as-insertion} with \eqref{eq:canint-Lie} as
	\begin{multline*}
		\frac{1}{2}\sum_{(\omega)}
		(-1)^{|\omega_{(1)}|} I(\omega_{(2)}) \circ I(\omega_{(1)})
		+\frac{1}{2}\sum_{(\omega)}(-1)^{|\omega_{(1)}|\cdot|\omega_{(2)}|}(-1)^{|\omega_{(2)}|} I(\omega_{(1)}) \circ I(\omega_{(2)})
		\\
		=\frac{1}{2}\sum_{(\omega)} (-1)^{|\omega_{(1)}|} \underbrace{\left(
		 I(\omega_{(2)}) \circ I(\omega_{(1)})
		-(-1)^{(1+|\omega_{(1)}|)(1+|\omega_{(2)}|)} I(\omega_{(1)}) \circ I(\omega_{(2)})
	\right)}_{\left[ I(\omega_{(2)}) , I(\omega_{(1)}) \right]}.
	\end{multline*}
	This provides the Lie bracket terms in \eqref{eq:Stokes}; it remains only to identify the primitive terms $1\otimes\omega+\omega\otimes 1$ in the coproduct using
	\begin{equation*}
		\frac{1}{2} \left[ I(\omega),I(1) \right]+\frac{(-1)^{|\omega|}}{2} \left[ I(1),I(\omega) \right]
		= \left[ I(\omega),I(1) \right].
	\end{equation*}
	To conclude, recall from \cref{lem:pffint=MC} that the cochain defined by $I(1)$ is equal to the dipole sum $\MCdip$.
\end{proof}
\begin{example}
	For $\omega=1$ in \cref{thm:Stokes}, the coproduct is $\Delta 1=1\otimes 1$ and thus $\sum_{(\omega)}\omega'\otimes\omega''=-1\otimes 1$. Stokes' relation hence amounts to the Maurer-Cartan equation
\begin{equation*}
    0 = \delta I(1)  + \frac{1}{2}[I(1), \MCdip]
      = \delta \MCdip + \frac{1}{2} [\MCdip,\MCdip]
\end{equation*}
for the dipole sum $\MCdip$ from \eqref{eq:MCdip}. This equation was obtained combinatorially in \cite{diff}, and used there to define the \textbf{twisted differential} $\delta+[\cdot,\MCdip]$.
This twist adds $(2i+1)$-valent vertices and inserts $(2i+1)$-fold multiedges, in all possible ways.

Conversely, without advance knowledge of \cref{lem:pffint=MC}, Stokes' relation for the bare Pfaffian form says $\delta I(1)=-[I(1),I(1)]/2=-I(1)\circ I(1)$. Applied to the triangle graph
\begin{equation*}
	G=\Graph[0.8]{gcTri12k}
	\quad\text{with boundary}\quad
	\partial G=D_{2i+1}
\end{equation*}
and subgraph $\gamma\cong D_{2i-1}$ with quotient $G/\gamma\cong-D_{3}$ (see \cref{ex:tri-quotient}), we find $I_{D_{2i+1}}(1)=I_{D_3}(1) I_{D_{2i-1}}(1)=(I_{D_3})^i$ by induction. Proving \cref{ex:dipole-integrals} then simplifies to computing a single elementary integral: $I_{D_3}(1)=1$ (the integrand is $\pff_{D_3}$ from \cref{sec:intro-gc}).
\end{example}

\begin{corollary}\label{cor:closed-twisted}
	For any integer $i\geq 1$, the cochain $I(\beta^{4i+1})$ is closed with respect to the twisted differential: $\delta I(\beta^{4i+1})=-[I(\beta^{4i+1}), \MCdip]$.
\end{corollary}

\subsection{Classes in graph cohomology}
\label{sec:cohom-classes}

We now show \cref{lem:intro-gr-canint=0} from the introduction. We write $\lceil x \rceil$ for the smallest integer larger than or equal to $x$.
\begin{lemma}\label{lem:gr-canint=0}
    For a homogeneous canonical form $\omega\in\cfs[d]$ of degree $d$, the integrals $I_G(\omega)=0$ vanish for all graphs with loop order $\ell(G)<2+2\lceil d/4 \rceil$.
\end{lemma}
\begin{proof}
	Suppose $I_G(\omega)\neq 0$ and let $\ell=\ell(G)$. Then $G$ must have $m=1+\ell+d$ edges (\cref{cor:deg-match-canint}). By \cref{lem:pfbeta-vanish}, $G$ must be connected, and therefore $G$ has $n=m-\ell+1=2+d$ vertices, and furthermore each vertex must have degree at least 3, giving $m\geq 3n/2$. This amounts to $\ell\geq 2+d/2$. Since $\ell$ must be even (otherwise $\pff_G=0$), we conclude $\ell=2\lceil \ell/2 \rceil\geq 2\lceil 1+d/4\rceil$.
\end{proof}

We denote the restriction of a cochain $I(\omega)$ to graphs with loop order $\ell$ as
\begin{equation*}
	\gr_\ell I(\omega)\in\cGC_3\otimes\R,\qquad
	G\mapsto\begin{cases}
		I_G(\omega) & \text{if $\ell(G)=\ell$,} \\
		0 & \text{otherwise.}
	\end{cases}
\end{equation*}
Since $\gr_\ell I_G(\omega)$ is supported in a fixed loop order and fixed degree $k=1+|\omega|-2\ell$ (for homogeneous $\omega$), there are only finitely many isomorphism classes of such graphs. Thus $\gr_\ell I(\omega)\in\cGC_3\otimes\R$ can be written with an ordinary (not completed) tensor product.

\begin{corollary}\label{cor:gr-cocycle}
	For any $\omega\in\cfs[d]$,
	the restriction of $I(\omega)$ to graphs with $2+2\lceil d/4\rceil$ loops defines a cocycle in $\cGC_3\otimes\R$ in degree $k\in\{-6,\ldots,-3\}$ with $k\equiv d+1\mod 4$.
\end{corollary}
\begin{proof}
Define $c= 2 \lceil|\omega|/4\rceil $ and likewise $c',c''$ with $\omega$  replaced by $\omega'$, $\omega''$.
	The Lie bracket $\gr_\ell [I(\omega''),I(\omega')]=\sum_{\ell=\ell'+\ell''} \left[\gr_{\ell''} I(\omega''),\gr_{\ell'} I(\omega')\right]$ preserves loop order. Non-vanishing of a summand requires, by \cref{lem:gr-canint=0}, that
    $\ell'\geq 2+c'$ and $\ell''\geq 2+c''$
	which imply $\ell=\ell'+\ell''\geq 4+c$. Hence, applying $\gr_{2+c}$ to the Stokes' relation \eqref{eq:Stokes} kills all Lie brackets, leaving only $0=\delta \gr_{2+c} I(\omega)$. Since $\ell=2+c=2+2\lceil d/4\rceil$, the degree $k=1+d-2\ell$ is equal to $-6+((d+3)\mod 4)$, where $d+3\mod 4\in\{0,1,2,3\}$.
\end{proof}

\begin{definition}
	For every integer $i\geq 1$, let $\tau_i\in\cGC_3\otimes\R$ denote the cocycle in degree $-6$ given by
	$\tau_i\defas \gr_{4+2i} I(\beta^{4i+1})$.
\end{definition}
In \cref{sec:gr6H-6} we demonstrate that $\tau_1$ is not exact, so its class spans $\gr_6 H^{-6}(\cGC_3)\otimes\R$.
\begin{question}
	Are the classes $[\tau_i]\in \gr_{4+2i} H^{-6}(\cGC_3)\otimes\R$ non-zero for all $i>0$?
\end{question}
The first such class is a bracket $[\tau_1]=[[Y_3],[D_3]]$ of the classes $[D_3],[Y_3]\in H^{-3}(\cGC_3)$ of the dipole and prism graph (see \cref{sec:gr6H-6}).
In fact, it follows from \cite[Theorem~2]{diff} that $\tau_k=\delta c_k+[c_k,\MCdip]$ for some (not necessarily closed) cochains $c_k\in\cGC_3$.\footnote{Explicit $c_k$ should be constructible from integrals of primitives \cite{BismutCheeger:TransgressedEuler} of the Pfaffian forms $\pff_G$.} This suggests that perhaps \emph{all} classes $[\tau_k]$ are brackets with $[D_3]$.
If true, this and $[ [D_3],[D_3] ]=0$ would give another explanation of the following observation.
\begin{lemma}
    For all integers $i>0$, we have $\Big[ [\tau_i],[D_3] \Big] = 0 \in H^{-9}(\cGC_3)$.
\end{lemma}
\begin{proof}
	Restrict Stokes' relation from \cref{cor:closed-twisted} to loop order $\ell = 6+2i$. This yields
	\begin{equation*}
		\delta \Big(\gr_{6+2i} I(\beta^{4i+1})\Big)
		= -\left[\gr_{4+2i} I(\beta^{4i+1}), \tfrac{1}{12} D_3\right]
		=-\tfrac{1}{12} \left[\tau_i,D_3\right]
	\end{equation*}
	where $D_3/12=\gr_2 \MCdip$. So, indeed, $[[\tau_i], [D_3]]$ is zero in cohomology.
\end{proof}

To illustrate an example where $\omega$ is not primitive, consider $\omega=\beta^5\wedge\beta^9$. The corresponding cocycle $\gr_{10}I(\omega)$ from \cref{cor:gr-cocycle} lies in degree $-5$ and is trivial, since $\gr_{10} H^{-5}(\cGC_3)=0$ (\cref{table:GC3}). So there is a cochain $c\in\cGC_3\otimes\R$ such that $\gr_{10} I(\omega)=\delta c$. Restricting Stokes' relation to $\ell=12$ loops (instead of $\ell=10$), we find
\begin{equation*}
	0=\delta \big( \gr_{12} I(\omega)\big) + \big[\gr_{10} I(\omega), D_3/12\big]
	= \delta \eta
	\quad\text{where}\quad
	\eta\defas \gr_{12} I(\omega) + [c, D_3/12].
\end{equation*}
Therefore, we have constructed a cocycle in degree $-9$ with $\ell=12$ loops (this is in the unknown region of \cref{table:GC3}). Its cohomology class does not depend on the choice of $c$.
\begin{question}
	Is the class $[\eta]\in H^{-9}(\cGC_3)\otimes\R$ non-zero?
\end{question}
Examining Stokes' relation for $\omega=\beta^5\wedge\beta^9$ at $\ell=14$ loops, we learn that
\begin{equation*}
	0=\delta\big(\gr_{14} I(\omega)\big) + [\gr_{12} I(\omega), D_3^\vee] + [\gr_{10} I(\omega), D_5^\vee] + [\tau_2, \tau_1]
\end{equation*}
where $D_3^\vee\defas D_3/12=\gr_2\MCdip$ and $D_5^\vee \defas D_5/240=\gr_4\MCdip$.
By the (graded) Jacobi identity and $\delta D_5^\vee = -[D_3^\vee, D_3^\vee]/2$ from the Maurer-Cartan equation for $\MCdip$, we can replace
\begin{equation*}
	[\gr_{10} I(\omega), D_5^\vee]
  = \delta [c, D_5^\vee] - [c, \delta D_5^\vee]
  = \delta [c, D_5^\vee] + [[c, D_3^\vee], D_3^\vee].
\end{equation*}
Therefore, the class $[\eta]$ constructed from $\beta^5\wedge\beta^9$ is related to the classes $[\tau_1]$ and $[\tau_2]$ constructed from $\beta^5$ and $\beta^9$ through the Lie algebra structure on graph cohomology as
\begin{equation}
	\Big[ [\tau_1], [\tau_2]\Big] = \tfrac{1}{12} \Big[ [\eta], [D_3] \Big]
	\quad\in H^{-12}(\cGC_3)
	.
\end{equation}

\subsection{A class in degree \texorpdfstring{$-6$}{-6}}\label{sec:gr6H-6}
Let $X\in\hGC_3$ denote the chain shown in \cref{fig:cycle6}.\footnote{In terms of the graphs defined in \cref{tab:integrals} and the \href{https://dx.doi.org/10.5287/ora-ngnborbr4}{ancillary files that accompany this paper}, this element is (row by row, and within each row from left to right)
$
X=G_{266}+G_{103}+G_{106}-2G_{107}-G_{109}
+G_{99}+G_{101}+G_{112}+G_{236}-G_{195}
-\tfrac{1}{3}G_{198}-G_{199}-G_{234}+G_{87}-\tfrac{1}{3}G_{288}
+G_{35}+G_{175}+G_{247}-2 G_7-G_{22}
$.
}
\begin{figure}
\begin{gather*}
    +\;\Graph[0.7]{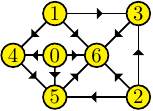}
    +\Graph[0.7]{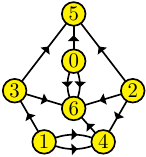}
    +\Graph[0.7]{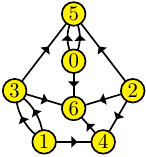}
    -2\times\Graph[0.7]{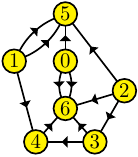}
    -\Graph[0.7]{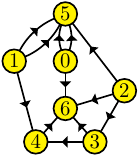}
    \\
    +\;\Graph[0.7]{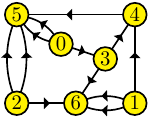}
    +\Graph[0.7]{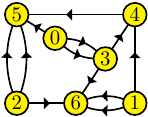}
    +\Graph[0.7]{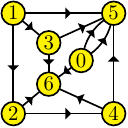}
    +\Graph[0.7]{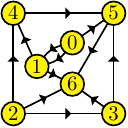}
    -\Graph[0.7]{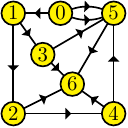}
    \\
    -\frac{1}{3}\times\Graph[0.7]{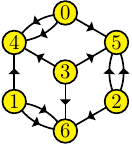}
    -\Graph[0.7]{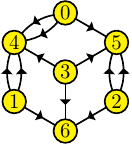}
    -\Graph[0.7]{gc234}
    +\Graph[0.7]{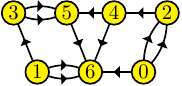} 
    -\frac{1}{3}\times\!\Graph[0.7]{gc288} 
    \\
    +\;\Graph[0.7]{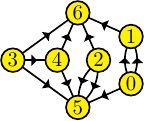} 
    +\Graph[0.7]{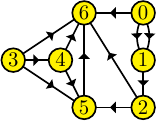} 
    +\Graph[0.7]{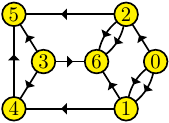} 
    -2\times\Graph[0.7]{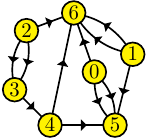} 
    -\Graph[0.7]{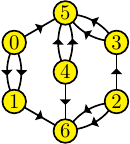} 
\end{gather*}%
    \caption{A graph cycle representing the homology class $\gr_6 H_{-6}(\hGC_3)\cong\Q$.}%
    \label{fig:cycle6}%
\end{figure}
With a computer program we verified that $X$ is a cycle ($\partial X=0$) and that it spans the one-dimensional 6-loop part $\gr_6 H_{-6}(\hGC_3)=\Q\cdot[X]$ of the homology group in degree $-6$.

This cycle pairs non-trivially with the cocycle $[Y_3,D_3]$ obtained in \cref{ex:[D3prism]} from the prism and theta graphs. Only two graphs are shared between the cocycle and $X$, so\footnote{The factors $2$ and $6$ are sizes of automorphism groups ($2$ from the double-edge) in \eqref{eq:GC-duality}.}
\begin{align}
  \inner{[Y_3,D_3],X}
  &= \inner{[Y_3,D_3],\Graph[0.6]{gc112}}
    - \frac{1}{3} \inner{[Y_3,D_3],\Graph[0.6]{gc288}} \nonumber\\
  &= -1\cdot 2 \cdot 72 - \frac{1}{3} \cdot 6 \cdot 24
  = - 192. \label{eq:[Y3,D3](X)}
\end{align}
Therefore, the homology class $[X]\in H_{-6}(\hGC_3)$ and the cohomology class $[ [Y_3],[D_3]]\in H^{-6}(\cGC_3)$ are clearly non-zero. With the knowledge of the ranks from \cref{table:GC3}, we conclude that the bracket with $D_3$ restricts to an isomorphism $\gr_4 H^{-3}(\cGC_3)\cong \gr_6 H^{-6}(\cGC_3)$, as predicted by the spectral sequence from \cite[Table~2, second arrow from the left]{diff}.
\begin{remark}\label{rem:[K4.K4]}
This cohomology class is also representable by $[K_4,K_4]=2K_4\circ K_4$, the bracket of the complete graph $K_4$ with itself. Explicitly, these are
\begin{equation*}
    K_4 = \Graph[0.8]{gcK4}
    \quad\text{and}\quad
  [K_4, K_4] = -192\times \Graph[0.8]{gc288} - 288\times \Graph[0.8]{gc287} 
  \end{equation*}
  and one reads off that
  $\inner{[K_4,K_4],X} =(1/3)\cdot 6\cdot 192 = 384 $. Thus, in the cohomology group $H^{-6}(\cGC_3)$, there is a relation
$ [[K_4], [K_4]] = -2 [[Y_3],[D_3]] $.
\end{remark}

We computed the canonical integrals $\tau_1(G)$ for all 288 graphs that define non-zero elements of $\hGC_3$ at $6$ loops and degree $-6$ (that is, $12$ edges). The calculations are detailed in \cref{sec:integration6}; for each graph, $\tau_1(G)$ is a $\Q$-linear combination of the numbers
\begin{equation}\label{eq:6loop-periods}
    1,\
    \pi^2,\
    \zeta(3),\
    \pi^2 \ln 2,\
    \pi^4,\
    (\ln 2)^4 + 24\Li{4}(1/2),\ \text{and}\
    \pi^2\Im\Li{2}(\iu)+24\Im\Li{4}(\iu).
\end{equation}
Our results show in particular that the cocycle $\tau_1$ does detect the homology class $[X]$:
\begin{thm}\label{thm:tau1(X)}
    For $X\in \hGC_3$ shown in \cref{fig:cycle6}, the canonical integral is
    \begin{equation}\label{eq:tau1(X)}
        \tau_1(X) 
        = I_X(\beta^5)
        = 40 \cdot\left( 13\zeta(3) - 2\pi^2\ln 2\right).
    \end{equation}
\end{thm}
Comparing with \eqref{eq:[Y3,D3](X)}, we can formulate this result also as follows: The cohomology class defined by the cocycle $\tau_1=\gr_6 I(\beta^5)\colon \hGC_3\rightarrow \R$ is equal to
\begin{equation*}
    [\tau_1] = \frac{5}{24}\left( 2\pi^2\ln 2-13\zeta(3)\right) \cdot \Big[ [Y_3], [D_3] \Big]
    \in H^{-6}(\cGC_3)\otimes \R.
\end{equation*}
Written in terms of $[K_4,K_4]$ as in \cref{rem:[K4.K4]}, we get the form stated in the introduction.

It is striking to observe the simplicity of $\tau_1(X)$ as compared to the individual integrals $\tau_1(G)$. Most polylogarithms cancel, and only an alternating sum of weight 3 remains:
\begin{equation*}
    2\pi^2\ln 2-13\zeta(3)
    =8\Li{1,2}(-1,-1)
    =8\sum_{0<m<n} \frac{(-1)^{n+m}}{mn^2}.
\end{equation*}
Assuming that the 7 numbers in \eqref{eq:6loop-periods} are linearly independent over $\Q$, the fact that $\tau_1$ is a cocycle (with values in $\R$) implies that each of the 7 coefficient functions $\lambda_i\colon \hGC_3\rightarrow \Z$ in \eqref{eq:period-basis} is a cocycle as well. Indeed, we verified that each $\lambda_i$ as given in \cref{tab:integrals} is a cocycle. The simplification in \eqref{eq:tau1(X)} thus amounts to the exactness of each $\lambda_i$ with $i\in\{1,2,5,6,7\}$, as well as the exactness of the combination $\lambda_3-(4/3)\lambda_4$ from rewriting
\begin{equation*}
\lambda_3 \zeta(3) + \frac{\lambda_4}{6} \Big(2\pi^2\ln 2-21 \zeta(3)\Big)
=\Big(\lambda_3-\frac{4}{3}\lambda_4\Big)\zeta(3) + \frac{\lambda_4}{6}\Big(2\pi^2\ln 2-13\zeta(3)\Big)
\end{equation*}
in order to adapt our basis from \eqref{eq:period-basis} to include the particular combination of $\pi^2\ln 2$ and $\zeta(3)$ that survives in the integral of the cycle $X$.
\begin{remark}
    Similarly, each cocycle $\tau_k(G)=\sum_i \lambda_{k,i}(G) p_{k,i}$ takes values in some finite-dimensional $\Q$-vector space of real numbers that are periods in the sense of \cite{KZ:Periods}. The coefficient functions in any basis $\{p_{k,i}\}$ provide rational cocycles $\lambda_{k,i}$. A \emph{single} cocycle $\tau_k$ might thereby supply \emph{several} rational cohomology classes $[\lambda_{k,i}]$. For example, these coefficients of $[\tau_2]$ might span a subspace of dimension $0$, $1$, or $2$ in $\gr_8 H^{-6}(\cGC_3)\cong\Q^2$.

    Technically, this can be implemented (avoiding transcendence conjectures) by promoting the cocycles to taking values in the ring of motivic periods, similar to \cite[\S9]{invariant}.
\end{remark}

\begin{remark}
  Any element $h\in\hGC_3$ has a unique decomposition $h=h_s+h_m$ into a simple part $h_s$ (no multiedges) and the part $h_m$ from graphs with one or more multiedges. The latter span a subcomplex that is almost acyclic \cite[Theorem~2]{multi}; its homology is spanned by the theta graph $D_3$. For $\ell>2$ loops, the class $[h]\in H_{\bullet}(\hGC_3)$ is therefore determined by $h_s$. In the case of $h=X$ the cycle from \cref{fig:cycle6}, this simple part is
    \begin{equation*}
        X_s = \Graph[0.7]{gc266}-\frac{1}{3}\times\!\!\!\Graph[0.7]{gc288}.
    \end{equation*}
    However, canonical integrals of graphs with multiedges can be non-zero (see \cref{sec:integration6}). Hence to obtain $[\tau_1]([X])$, it was not enough to only compute the integrals of the two graphs in $X_s$; instead, we also needed the multiedge part $X_m=X-X_s$. In contrast, for the acyclic subcomplex of graphs with cut vertices the situation is better, see \cref{rem:bicon-qiso}.
\end{remark}

\appendix

\section{Integrals at 6 loops and 12 edges}
\label{sec:integration6}

This appendix lists the integrals $\tau_1(G)=I_G(\beta^5)$ for all connected graphs $G$ with $6$ loops, $12$ edges, and at least 3 edges incident at any vertex. The number of isomorphism classes of such graphs which do not have any odd automorphisms, i.e.\ the dimension of $\hGC_3$ in this bidegree $(\ell,k)=(6,-6)$, is $288$ by \cite[Table~3]{BNM}. We confirmed this number using Sage \cite{sagemath} and {\nauty} \cite{McKayPiperno:II}, and we computed the corresponding Pfaffians and traces with {\Maple}. For all 288 graphs, we find that
\begin{equation}\label{eq:b5pf-psi5/2}
    \pff_G\wedge \omega^5_G = \frac{Q_G\Omega_{12}}{\Symanzik_G^{5/2}}
\end{equation}
for a homogeneous polynomial $Q_G\in\Z[x_1,\ldots,x_{12}]$ of degree $3$. Thus for these graphs, the bound from \cref{thm:pfbeta-order} is not optimal: we observe the cancellation of an additional power of $\Symanzik_G$.

For most graphs, the product $\pff_G\wedge\omega^5_G$ is in fact zero, and therefore $\tau_1(G)=0$. Most of these zeroes are explained by the presence of a 2-edge cut, according to \cref{lem:pfbeta-vanish}.

Only 45 graphs give a non-zero integrand $\pff_G\wedge\omega^5_G$. These graphs are listed in \cref{tab:integrals}. We write their adjacency lists as $B_0|\cdots|B_6$, with each block $B_i$ encoding the edges pointing away from vertex $i$, represented by the vertex they are pointing towards. For example, the graph $G_{199}$ with adjacency list
$445|446|556|456|||$ has the edges
\begin{equation*}
    0\rightarrow 4,0\rightarrow 4,0\rightarrow 5,
    1\rightarrow 4,1\rightarrow 4,1\rightarrow 6,
    2\rightarrow 5,2\rightarrow 5,2\rightarrow 6,
    3\rightarrow 4,3\rightarrow 5,3\rightarrow 6.
\end{equation*}
Whitney flips relate several of these graphs and such pairs have identical $\tau_1(G)$---up to sign---by \eqref{eq:pfbeta-flip}. We furthermore exploit Stokes' relations $\tau_1(\partial H)=0$. For example, we have $\tau_1(G_{267})=-\tau_1(G_{112}+G_{185}+G_{259}+G_{264}+G_{266})$ from the boundary
\begin{equation}\label{eq:Stokes-267}\begin{aligned}
    \partial\Graph[0.7]{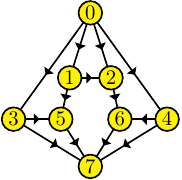}
    &=-2\times\Graph[0.7]{gc112}-2\times\Graph[0.7]{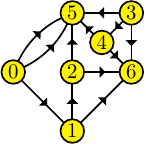}-2\times\Graph[0.7]{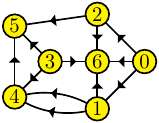}
    \\
    &\quad -2\times\Graph[0.7]{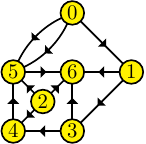}-2\times\Graph[0.7]{gc266}-2\times\Graph[0.7]{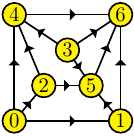}
    .
\end{aligned}\end{equation}
Combining all these relations, the 45 values $\tau_1(G)$ can be expressed in terms of only 14 remaining integrals.
To compute these integrals, we used a range of methods for Feynman integrals, which we illustrate in this appendix. The results are summarized in \cref{tab:integrals}.

For all 45 graphs, there are integer vectors $\lambda\in\Z^7$ such that
\begin{equation}\label{eq:period-basis}\begin{aligned}
    \tau_1(G) &=
    10\cdot\bigg(
        \frac{\lambda_1}{3}
        +\frac{\lambda_2}{9} \pi^2
        +\lambda_3\zeta(3)
        +\frac{\lambda_4}{6} \Big(2\pi^2 \ln 2-21\zeta(3)\Big)
        + \frac{\lambda_5}{180}\pi^4
        \\ &\quad\qquad
        + \frac{\lambda_6}{3} \Big((\ln 2)^4 + 24\Li{4}(1/2)\Big)
        + \frac{\lambda_7}{9}\Big(\pi^2\Catalan+24\Im\Li{4}(\iu)\Big)
    \bigg)
\end{aligned}\end{equation}
is a linear combination of special values of polylogarithms.\footnote{We combine $\pi^2\ln 2$ with $-21\zeta(3)/2$ for two reasons: 1) this sets $\lambda_3=0$ for many graphs; 2) this particular linear combination is a motivic Galois conjugate of $(\ln 2)^4+24\Li{4}(1/2)$ and thus makes it more easily apparent that \cref{tab:integrals} supports a generalisation of the coaction principle \cite{PanzerSchnetz:Phi4Coaction}.} Here $\zeta(3)=\sum_{n=1}^{\infty} 1/n^3\approx \numprint{1.202057}$ is Ap\'{e}ry's constant and $\Catalan=\Im\Li{2}(\iu)$ denotes Catalan's constant, thus
\begin{equation*}
    \Catalan = \sum_{n=0}^{\infty} \frac{(-1)^n}{(2n+1)^2}
    \approx \numprint{0.915966}
    \qquad\text{and}\qquad
    \Im\Li{4}(\iu) = \sum_{n=0}^{\infty} \frac{(-1)^n}{(2n+1)^4}
    \approx \numprint{0.988945}
    .
\end{equation*}
In addition to the derivations explained below, all results in \cref{tab:integrals}, with the exception of the graphs $G_{266}$ and $G_{267}$, have furthermore been confirmed independently within a precision of roughly 5 digits by numeric integration of \eqref{eq:b5pf-psi5/2} using {\pySecDec} \cite{BHJJKSZ:pySecDec}.

All results reported in this appendix, e.g.\ definitions of graphs and values of their canonical integrals (as summarized in \cref{tab:integrals}), the canonical forms (the numerator polynomials $Q_G$ of \eqref{eq:b5pf-psi5/2}) and the program used to compute them, are available and documented in plain text files (in a format compatible with Maple) deposited at
\begin{center}
    \href{https://dx.doi.org/10.5287/ora-ngnborbr4}{https://dx.doi.org/10.5287/ora-ngnborbr4}
\end{center}

\begin{table}
    \centering
    {\small\begin{tabular}{rrrrrrrrrr}
\toprule
$G$ & edges & $\lambda_1$ & $\lambda_2$ & $\lambda_3$ & $\lambda_4$ & $\lambda_5$ & $\lambda_6$ & $\lambda_7$ & $\tau_1(G)\approx$ \\
\midrule
97 & 446|566|346|56|5|| &  &  &  &  &  &  &  & 0.00000 \\
99 & 355|466|556|46|5|| & $-8$ & $-3$ &  &  &  &  & $2$ & 13.26773 \\
100 & 335|466|566|46|5|| & $-8$ & $-3$ &  &  &  &  & $2$ & 13.26773 \\
101 & 335|466|556|46|5|| & $-8$ & $-3$ &  &  &  &  & $2$ & 13.26773 \\
102 & 335|446|556|46|5|| & $-8$ & $-3$ &  &  &  &  & $2$ & 13.26773 \\
103 & 566|344|456|56|6|| & $-8$ & $6$ &  & $-2$ &  &  & $-2$ & 4.83434 \\
104 & 566|334|456|56|6|| & $-8$ & $6$ &  & $-2$ &  &  & $-2$ & 4.83434 \\
105 & 556|344|456|56|6|| & $-8$ & $6$ &  & $-2$ &  &  & $-2$ & 4.83434 \\
106 & 556|334|456|56|6|| & $-8$ & $6$ &  & $-2$ &  &  & $-2$ & 4.83434 \\
107 & 566|455|356|46|6|| & $8$ & $-6$ &  & $2$ &  &  & $2$ & -4.83434 \\
108 & 566|445|356|46|6|| & $8$ & $-6$ &  & $2$ &  &  & $2$ & -4.83434 \\
109 & 556|455|356|46|6|| & $8$ & $-6$ &  & $2$ &  &  & $2$ & -4.83434 \\
110 & 556|445|356|46|6|| & $8$ & $-6$ &  & $2$ &  &  & $2$ & -4.83434 \\
112 & 556|235|46|56|56|| &  & $-3$ & $-16$ & $-8$ &  &  & $2$ & 1.75220 \\
193 & 112|34|56|556|56|| &  & $-4$ &  & $8$ & $-70$ &  & $16$ & 5.83972 \\
194 & 146|56|4456|456||| &  & $-2$ &  & $5$ & $86$ & $-3$ & $-6$ & 2.12214 \\
195 & 155|23|46|56|56|6| &  & $3$ & $16$ & $8$ &  &  & $-2$ & -1.75220 \\
196 & 115|23|46|56|56|6| &  & $3$ & $16$ & $8$ &  &  & $-2$ & -1.75220 \\
198 & 445|466|556|456||| & $-24$ & $6$ &  &  &  &  &  & -14.20264 \\
199 & 445|446|556|456||| & $-24$ & $6$ &  &  &  &  &  & -14.20264 \\
227 & 355|455|346|6|6|6| & $-24$ & $-8$ &  &  & $30$ &  &  & -5.38133 \\
228 & 335|455|346|6|6|6| & $-24$ & $-8$ &  &  & $30$ &  &  & -5.38133 \\
229 & 335|445|346|6|6|6| & $-24$ & $-8$ &  &  & $30$ &  &  & -5.38133 \\
234 & 112|46|56|456|6|6| & $-24$ & $14$ &  & $-6$ & $-24$ & $2$ & $-4$ & -2.07204 \\
235 & 155|46|346|56|5|6| & $8$ & $6$ & $16$ &  & $15$ &  & $-10$ & 1.80196 \\
236 & 115|46|346|56|5|6| & $8$ & $6$ & $16$ &  & $15$ &  & $-10$ & 1.80196 \\
237 & 122|56|35|46|56|6| &  & $1$ &  & $-3$ & $-7$ & $1$ & $-2$ & 0.22470 \\
238 & 112|56|35|46|56|6| &  & $1$ &  & $-3$ & $-7$ & $1$ & $-2$ & 0.22470 \\
239 & 1466|56|356|46|5|| &  & $2$ & $48$ & $19$ & $32$ & $-1$ & $-10$ & -0.34061 \\
240 & 1446|56|356|46|5|| &  & $6$ &  & $-1$ & $-37$ & $1$ & $2$ & -0.16321 \\
241 & 246|346|56|56|5|6| & $-8$ & $-18$ &  & $2$ & $64$ & $-2$ &  & -0.58760 \\
243 & 112|46|56|456|55|| &  & $34$ &  &  & $-70$ &  &  & -5.96141 \\
244 & 112|46|56|4566|5|| &  &  &  &  &  &  &  & 0.00000 \\
245 & 112|46|56|4556|5|| &  & $10$ &  & $-6$ &  &  & $-6$ & 6.77320 \\
246 & 112|46|56|4456|5|| &  & $-10$ &  & $6$ &  &  & $6$ & -6.77320 \\
248 & 112|46|556|456|5|| &  & $-5$ &  & $-2$ & $30$ &  & $-4$ & 0.38791 \\
250 & 112|446|56|456|5|| &  & $5$ &  & $2$ & $-30$ &  & $4$ & -0.38791 \\
256 & 126|46|56|456|55|| &  & $18$ & $32$ & $10$ & $24$ & $-2$ & $-12$ & -2.08860 \\
257 & 126|46|56|4456|5|| &  & $-28$ & $-16$ & $-3$ & $62$ & $-1$ & $4$ & -2.55978 \\
259 & 126|446|56|456|5|| &  & $-1$ &  & $3$ & $7$ & $-1$ & $2$ & -0.22470 \\
261 & 112|35|46|56|56|6| & $-8$ & $-15$ &  & $1$ & $47$ & $-1$ &  & 1.74975 \\
264 & 155|36|456|46|5|6| & $8$ & $15$ &  & $-1$ & $-47$ & $1$ &  & -1.74975 \\
265 & 115|36|456|46|5|6| & $8$ & $15$ &  & $-1$ & $-47$ & $1$ &  & -1.74975 \\
266 & 456|346|356|6|5|6| & $-8$ & $-2$ & $-16$ & $-4$ & $-39$ & $2$ & $8$ & 0.76008 \\
267 & 124|56|45|456|6|6| &  & $-9$ & $32$ & $10$ & $79$ & $-2$ & $-12$ & -0.53784 \\
\bottomrule
\end{tabular}
}
    \caption{The canonical Pfaffian integrals $\tau_1(G)$ of graphs $G$ with $\ell=6$ loops and $12$ edges, given numerically (last column) and exactly in the basis \eqref{eq:period-basis}. The vanishing for $G_{97}$ and $G_{244}$ can be explained by Whitney flips, see \cref{rem:flip-vanishing}.}%
    \label{tab:integrals}%
\end{table}

\subsection{Momentum space propagators}\label{sec:pint}
Canonical integrals can be interpreted as linear combinations of propagator Feynman integrals:
since the Symanzik polynomial $\Symanzik_G=x_e \Symanzik_{G\setminus e} + \Symanzik_{G/e}$ is linear with respect to any variable, the integral over a single variable in a graph with $N$ edges gives
\begin{equation*}
    \int_{\FS_N}\frac{\Omega_{N}}{\Symanzik_G^{\FIdim/2}} \prod_{i=1}^{N} x_i^{n_i-1}
    = \frac{\Gamma(n_e)\Gamma(\FIdim/2-n_e)}{\Gamma(\FIdim/2)} \int_{\FS_{N-1}}\frac{\Omega_{N-1}}{\Symanzik_{G\setminus e}^{n_e\phantom{/}}\Symanzik_{G/e}^{\FIdim/2-n_e}}\prod_{i\neq e} x_i^{n_i-1}
\end{equation*}
for any monomial $\prod_i x_i^{n_i-1}$ in the numerator. The right-hand side can be identified with a massless Feynman integral of the graph $G\setminus e$, considered with two external legs attached at the endpoints of the deleted edge $e$, carrying a momentum $p$ with norm $\norm{p}=1$ \cite{Smirnov:AnalyticToolsForFeynmanIntegrals,Nakanishi:GraphTheoryFeynmanIntegrals}. Such integrals are heavily used in field theory calculations, for an overview of techniques see \cite{KotikovTeber:MultiTechniques}. They admit a momentum space representation
\begin{equation}\label{eq:pint}
    \pInt(G,n,\FIdim) =
    \left(\prod_{i=1}^{\ell(G)} \int_{\R^{\FIdim}} \frac{\td[\FIdim] q_i}{\pi^{\FIdim/2}}\right)
    \prod_{j=1}^{E(G)} \frac{1}{\norm{k_j}^{2n_j}}
\end{equation}
where the integral is over vectors $q_i\in\R^{\FIdim}$ that parametrize the loop space $H_1(G)\otimes\R^{\FIdim}$. The vectors $k_j\in\R^{\FIdim}$ are assigned to edges and denote a linear combination of loop momenta $q_i$ and the external momentum $p$. For example, the ``bubble'' integral is
\begin{align}
    \pInt\left(\Graph[0.5]{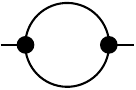},n_1,n_2,\FIdim\right)
    &=\left.\int_{\R^{\FIdim}}\frac{\td[\FIdim] q_1}{\pi^{\FIdim/2}} \frac{1}{\norm{q_1}^{2n_1}\norm{q_1-p}^{2n_2}}\right|_{\norm{p}=1}
    \nonumber\\
    &= \frac{\Gamma(\FIdim/2-n_1)\Gamma(\FIdim/2-n_2)\Gamma(n_1+n_2-\FIdim/2)}{\Gamma(n_1)\Gamma(n_2)\Gamma(\FIdim-n_1-n_2)} %
    \label{eq:bubble} %
\end{align}
Accounting for all prefactors, the precise relation between a canonical integral and a propagator Feynman integral of the form \eqref{eq:pint} is
\begin{equation}\label{eq:cut-pint}
    \int_{\FS_{N}}\frac{\Omega_{N}}{\Symanzik_G^{\FIdim/2}} \prod_{i=1}^{N} x_i^{n_i-1}
    =
    \pInt(G\setminus e,n,\FIdim)
    \cdot\frac{1}{\Gamma(\FIdim/2)}\cdot
    \prod_{i=1}^{N} \Gamma(n_i)
    .
\end{equation}
For example, the integral of the invariant Pfaffian form for the 3-edge dipole graph is
\begin{equation*}
    \int_{\FS_3} \pff_{D_3}
    = -\int_{\FS_3} \frac{\Omega_3}{\Psi_{D_3}^{3/2}}
    = -\frac{\Gamma(1)^3}{\Gamma(3/2)} \pInt\left(\Graph[0.5]{bubble}, 1, 1, 3\right)
    = -\frac{\Gamma(1/2)^3}{\Gamma(3/2)}
    = -2\pi
\end{equation*}
in agreement with \cref{ex:dipole-integrals}.

Consider now the integral \eqref{eq:b5pf-psi5/2} for the graph $G_{199}$ with adjacency list $445|446|556|456|||$.
The numerator of $\pff_G\wedge\omega^5_G$ is in this case $Q_G=10(x_1+x_2)(x_4+x_5)(x_7+x_8)$. Since $x_1$ and $x_2$ correspond to parallel edges $\{0,4\}$, swapping $x_1\leftrightarrow x_2$ leaves $\Symanzik_G$ invariant, and similarly for $x_4\leftrightarrow x_5$ and $x_7\leftrightarrow x_8$. All 8 monomials obtained by expanding $Q_G$ therefore yield the same integral and can be combined into
\begin{equation*}
    \tau_1\left(\Graph[0.8]{gc199}\right)
    =\frac{10}{(-2\pi)^3}\int_{\FS_{12}}\frac{8x_1x_4x_7\Omega_{12}}{\Symanzik_G^{5/2}}
    = \frac{80}{(-2\pi)^3} \frac{1}{\Gamma(5/2)}\pInt\left( \Graph[0.6]{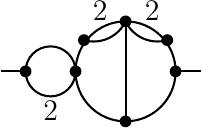}, 5 \right).
\end{equation*}
Here we chose to delete the edge $\{2,6\}$, and we adopt the convention that $n_i=1$ for all edges, unless a different value is indicated explicitly next to an edge.
We may replace
\begin{itemize}
    \item a sequential pair $\Graph[0.4]{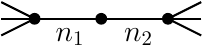}$ of edges with a single edge of index $n_1+n_2$, and
    \item a parallel pair $\Graph[0.4]{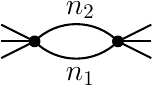}$ with a single edge of index $n_1+n_2-\FIdim/2$ and a factor \eqref{eq:bubble}.
\end{itemize}
Applying these series-parallel reductions, the example simplifies to
\begin{equation*}
    \pInt\left( \Graph[0.6]{pf199p}, 5 \right)
    = \left[\pInt\left(\Graph[0.5]{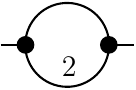},5\right)\right]^3 \pInt\left(\Graph[0.6]{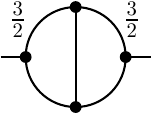},5\right).
\end{equation*}
The bubble integrals on the right-hand side each evaluate to $\pi^{3/2}/2$ via \eqref{eq:bubble}. We have therefore expressed the canonical integral of the graph $G_{199}$ as a 2-loop propagator,
\begin{equation*}
    \tau_1(G_{199}) = -10\cdot\frac{\pi}{6}\cdot \pInt\left(\Graph[0.6]{pf199pred},5\right).
\end{equation*}

\subsection{Hypergeometric functions}\label{sec:3F2}
The 2-loop propagator integral above can be expressed in terms of hypergeometric functions. One such formula was derived in \cite{BroadhurstGraceyKreimer:PositiveKnots}, which in this case shows that
\begin{equation}\label{eq:p199}
    \pInt\left(\Graph[0.6]{pf199pred},5-2\dimeps\right)
    =-8\pi+\frac{16}{\pi\dimeps}\left(
        \dimeps-\frac{1}{2}
        +\pFq{3}{2}{\tfrac{1}{2},1,-\tfrac{1}{2}+3\dimeps}{\tfrac{3}{2}+\dimeps,\tfrac{1}{2}+2\dimeps}{1}
    \right)
    +\asyO(\dimeps).
\end{equation}
In order to take the limit $\dimeps\rightarrow 0$, we expand the $_3 F_2$ function in $\dimeps$ using the algorithms from \cite{HuberMaitre:HypExp2}. Their implementation {\HypExpTwo} readily calculates that
\begin{equation*}
    \pFq{3}{2}{\tfrac{1}{2},1,-\tfrac{1}{2}+3\dimeps}{\tfrac{3}{2}+\dimeps,\tfrac{1}{2}+2\dimeps}{1}
    = \frac{1}{2} + \dimeps\left(2+\frac{\pi^2}{4}\right) + \dimeps^2\left(4-\frac{\pi^2}{2}-\frac{21\zeta(3)}{2}\right) + \asyO(\dimeps^3)
\end{equation*}
and therefore the propagator integral \eqref{eq:p199} evaluates at $\dimeps=0$ to $48/\pi-4\pi$. In conclusion, the canonical integral of $\pff_G\wedge\omega^5_G/(-2\pi)^3$ for the graph $G_{199}$ takes the value
\begin{equation*}
    \tau_1(G_{199}) = 10\cdot\left(\frac{2}{3}\pi^2-8\right)
    \approx -\numprint{14.202637}
    .
\end{equation*}

In the same way we can reduce another canonical integral to a 2-loop propagator. The graph $G_{100}$ with adjacency list $335|466|566|46|5||$ has numerator $Q_G=-10(x_1+x_2)(x_5+x_6)(x_8+x_9)$ which again raises the index to $n_i=2$ for one edge in each of the 3 pairs of double-edges. Therefore,
\begin{equation*}
    \tau_1(G_{100})
    =\tau_1\left(\Graph[0.8]{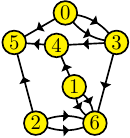}\right)
    =10\cdot\frac{\pi}{6}\cdot\pInt\left(\Graph[0.6]{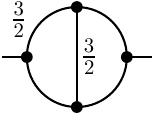},5\right).
\end{equation*}
Using again the formula from \cite{BroadhurstGraceyKreimer:PositiveKnots}, we can express the propagator integral in terms of hypergeometric functions, after introducing a regulator $\dimeps$. This gives a linear combination
\begin{align*}
    \pInt\left(\Graph[0.6]{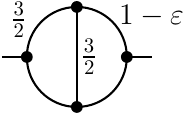},5-2\dimeps\right)
    &= A_1(\dimeps)
    +A_2(\dimeps)\cdot
    \pFq{3}{2}{1,\tfrac{1}{2},-\tfrac{1}{2}+2\dimeps}{\tfrac{3}{2},1+2\dimeps}{1}
    \\ & \quad
    +A_3(\dimeps)\cdot
    \pFq{3}{2}{1,-\dimeps,-\tfrac{1}{2}+2\dimeps}{\tfrac{3}{2},\tfrac{1}{2}+\dimeps}{1}
\end{align*}
where $A_i(\dimeps)$ are expressions in terms of $\Gamma$-functions that are easily expanded in $\dimeps$ and which have triple poles at $\dimeps=0$. We thus need to expand these $_3 F_2$ functions to order $\dimeps^3$. Unfortunately the implementation {\HypExpTwo} does not cover these functions, which belong to the case referred to as $3^1_2$ in \cite[\S B.5]{HuberMaitre:HypExp2}.
We therefore apply a different strategy: use a standard integral representation of hypergeometric functions, concretely
\begin{equation*}
    \pFq{3}{2}{1,\tfrac{1}{2},-\tfrac{1}{2}+2\dimeps}{\tfrac{3}{2},1+2\dimeps}{1}
    =
    \frac{\Gamma(1+2\dimeps)H(\dimeps)}{2\Gamma(\tfrac{1}{2}+2\dimeps)\sqrt{\pi}},
    \quad
    H(\dimeps)
    =
    \int_0^1 \!\frac{\td t_1}{\sqrt{1-t_1}} \int_0^1 \!\frac{\td t_2}{\sqrt{t_2}} \left(\frac{1-t_1t_2}{1-t_2}\right)^{\tfrac{1}{2}-2\dimeps},
\end{equation*}
and expand the integrand in $\dimeps$ to integrate term-by-term. These integrands are polynomials in $\log(1-t_2)$ and $\log(1-t_1t_2)$, with coefficients that are rational functions of $t_1,t_2$ and the square roots $\sqrt{t_2}$, $\sqrt{1-t_1}$, $\sqrt{1-t_2}$, and $\sqrt{1-t_1t_2}$. We parametrize
\begin{equation}\label{eq:3f2-linearization}
    t_1 = 1-\left(\frac{(1-u^2)v}{(1-v^2)u}\right)^2
    \quad\text{and}\quad
    t_2 = \left(\frac{2u}{1+u^2}\right)^2
\end{equation}
so that all four square roots become rational functions of the coordinates $u$ and $v$. After this change of variables, we obtain
\begin{equation*}
    H(\dimeps)
    =8\int_0^1 \frac{\td u}{u} \int_0^u \!\td v \frac{(1+v^2)^{2-4\dimeps}(1-u^2)^2}{(1-v^2)^{3-4\dimeps}(1+u^2)^2}
    =8\int_0^1 \!\td v \frac{(1+v^2)^{2-4\dimeps}}{(1-v^2)^{3-4\dimeps}} \left(
        \frac{v^2-1}{1+v^2}
        -\ln v
    \right).
\end{equation*}
We used {\HyperInt} \cite{Panzer:HyperInt} to express the $\dimeps$-expansion of this integral in terms of hyperlogarithms (iterated integrals) with letters in $\{0,\pm1,\pm\iu\}\subset\C$. This produces a large expression, but such iterated integrals with 4th roots of unity are well-understood \cite{Deligne:GroupeFondamentalMotiviqueN} and fulfil many linear relations. We used the implementation in {\HyperlogProcedures} \cite{Schnetz:HyperlogProcedures} of the decomposition algorithm \cite{Brown:DecompositionMotivicMZV} into the \emph{parity $f$-alphabet} \cite{PanzerSchnetz:Phi4Coaction} to simplify the expressions.

In a similar way, we can compute the $\dimeps$-expansion of the other $_3 F_2$ function
\begin{equation*}
    \pFq{3}{2}{1,-\dimeps,-\tfrac{1}{2}+2\dimeps}{\tfrac{3}{2},\tfrac{1}{2}+\dimeps}{1}
    =
    \frac{\Gamma(\tfrac{1}{2}+\dimeps)}{2\Gamma(-\dimeps)\Gamma(\tfrac{1}{2}+2\dimeps)}
    \int_0^1 \!\frac{\td t_1}{\sqrt{1-t_1}} \int_0^1 \!\frac{\td t_2}{t_2^{1+\dimeps}} \left(\frac{1-t_1t_2}{1-t_2}\right)^{\tfrac{1}{2}-2\dimeps}.
\end{equation*}
An additional complication here is caused by the divergence of the integral over $t_2$ when $\dimeps=0$. This translates into a singularity near $u=0$ in the parametrization \eqref{eq:3f2-linearization}. To expand in $\dimeps$, we add $0=u-\int_0^u \td v$ to split the integral as
\begin{multline*}
\int_0^1 \frac{2^{-\dimeps}\td u}{u^{2+2\dimeps}} \int_0^u \!\td v \frac{(1+v^2)^{2-4\dimeps}(1-u^2)^2}{(1-v^2)^{3-4\dimeps}(1+u^2)^{1-2\dimeps}}
 \\
 =
 \int_0^1 \frac{2^{-\dimeps}\td u}{u^{2+2\dimeps}} \int_0^u \!\td v \left(\frac{(1+v^2)^{2-4\dimeps}(1-u^2)^2}{(1-v^2)^{3-4\dimeps}(1+u^2)^{1-2\dimeps}} - 1 \right)
 -\frac{1}{\dimeps2^{1+\dimeps}}.
\end{multline*}
The integral in the second row can be expanded in $\dimeps$ on the integrand level, since the inner $\td v$-integral vanishes with order $u^2$ for small $u$, cancelling the singularity of the outer $\td u$-integration. Applying the same steps as before, we can thus calculate the $\dimeps$-expansion in terms of hyperlogarithms at 4th roots of unity.

Putting everything together, we obtain a surprisingly simple result, namely
\begin{equation}\label{eq:int_100(b5pf)}
    \tau_1(G_{100}) = 10\cdot\left(
        -\frac{8}{3}
        -\frac{\pi^2}{3}
        + \frac{2}{9}\left[
            \pi^2\Catalan + 24 \Im \Li{4}(\iu)
        \right]
    \right)
    \approx \numprint{13.267735}
    .
\end{equation}
The appearance of the specific combination $\pi^2\Catalan+24\Im\Li{4}(\iu)$ in Feynman integrals in odd dimensions was observed already in \cite{Hager:phi6eps3}, using very high precision numerics.

\subsection{Integration by parts}
Consider the graph $G_{110}$ with edge list $556|445|356|46|6||$. It has two double-edges, $\{0,5\}$ (variables $x_1,x_2$) and $\{1,4\}$ (variables $x_4,x_5$). The numerator is $Q=10x_{11}(x_1+x_2)(x_4+x_5)$, where the Schwinger parameter $x_{11}$ corresponds to the edge $\{3,6\}$.
Proceeding as in \cref{sec:pint}, we find
\begin{equation*}
    \tau_1(G_{110})
    =\tau_1\left(\Graph[0.8]{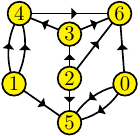}\right)
    =\frac{10\cdot 4}{(-2\pi)^3\Gamma(\tfrac{5}{2})}\pInt\left(\Graph[0.5]{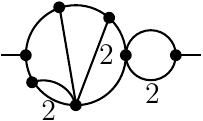}, 5\right)
    =\frac{-5}{3\sqrt{\pi}}\pInt\left(\Graph[0.5]{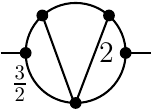}, 5\right)
\end{equation*}
where we chose to delete edge $\{1,5\}$ and integrated out the bubbles using \eqref{eq:bubble}. To compute this 3-loop propagator integral, we apply the triangle rule \cite[\S4.3]{ChetyrkinTkachov:IBP}. It gives
\begin{equation*}
\pInt\left(\Graph[0.5]{pf110pred}, D\right)
=\frac{
    2\pInt\left(\Graph[0.4]{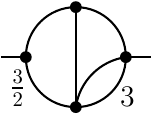}, D\right)
    -2\pInt\left(\Graph[0.4]{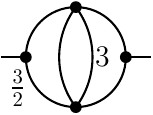}, D\right)
    +\pInt\left(\Graph[0.4]{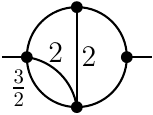}, D\right)
    -\pInt\left(\Graph[0.4]{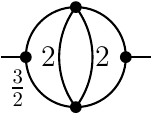}, D\right)
}{5-D}
\end{equation*}
which we expand in $D=5-2\dimeps$. Integrating out the parallel edges (bubbles) with \eqref{eq:bubble}, this expression simplifies to
\begin{align*}
    \pInt\left(\Graph[0.5]{pf110pred}, 5\right)
    &=
    \lim_{\dimeps=0} \bigg[
        \frac{\pi^{3/2}}{4}\Big(-\frac{1}{\dimeps} + \EM-\ln 4+2\Big)
        \pInt\left(\Graph[0.45]{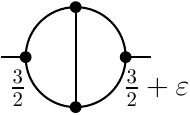}, 5-2\dimeps\right)
    \\ & \quad\quad
        +\frac{2}{\sqrt{\pi}}\Big(\frac{1}{\dimeps} - \EM-\ln 4+4\Big)
        \pInt\left(\Graph[0.45]{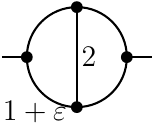}, 5-2\dimeps\right)
    \bigg]
    -\pi^{3/2}  \pInt\left(\Graph[0.45]{pf100p}, 5\right)
\end{align*}
where $\EM\approx\numprint{0.577216}$ denotes the Euler-Mascheroni constant. The last integral is \eqref{eq:int_100(b5pf)} divided by $5\pi/3$, and the $\dimeps$-expansions of the two integrals inside the limit can be computed with {\HypExpTwo} as explained in \cref{sec:3F2} for \eqref{eq:p199}. The final result is
\begin{equation*}
    \tau_1(G_{110})
    =10\cdot\left(
        \frac{8}{3}
        -\frac{2\pi^2}{3}
        -7\zeta(3)+\frac{2}{3}\pi^2\ln 2
        +\frac{2}{9}\left[
            \pi^2\Catalan + 24 \Im \Li{4}(\iu)
        \right]
    \right)
    \approx -\numprint{4.834340}
    .
\end{equation*}
The same strategy applies to the graph $G_{103}$ with edge list $566|344|456|56|6||$. The result turns out to be exactly the same, up to the sign: $\tau_1(G_{103})=-\tau_1(G_{110})$.

The graph $G_{234}$ with edge list $112|46|56|456|6|6|$ has numerator $Q=10x_{11}x_{12}(x_1+x_2)$. This amounts to doubled propagators $n_{11}=n_{12}=2$ on the edges $\{4,6\}$ and $\{5,6\}$. After cutting edge $\{0,2\}$ and integrating out the bubble of parallel edges $\{0,1\}$, the corresponding Feynman integral is
\begin{equation*}
    \tau_1(G_{234})
    =\tau_1\left(\Graph[0.8]{gc234}\right)
    =-\frac{5}{3\pi^2} \pInt\left(\Graph[0.5]{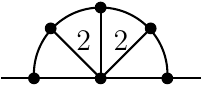}, 5\right)
    .
\end{equation*}
This can be reduced to 2-loop propagators by applying the triangle rule twice, and those remaining propagators can again be expanded in $\dimeps$ as in \cref{sec:3F2}. The result reads
\begin{align*}
    \tau_1(G_{234})
    &=10\cdot\bigg(
    -8
    +\frac{14\pi^2}{9}
    +21\zeta(3)-2\pi^2\ln 2
    -\frac{2\pi^4}{15}
    \\ & \quad\quad\quad\quad
    -\frac{4}{9}\left[
            \pi^2\Catalan + 24 \Im \Li{4}(\iu)
    \right]
    +\frac{2}{3}\left[
            (\ln 2)^4+24\Li{4}\left(\tfrac{1}{2}\right)
    \right]
    \bigg)
    \approx -\numprint{2.072043}
    .
\end{align*}

Consider now the graph $G_{227}$ with edge list $355|455|346|6|6|6|$. Its numerator polynomial is $Q=10x_9(x_2+x_3)(x_5+x_6)$. Cut edge $\{1,4\}$ and integrate out both bubbles, to get
\begin{equation*}
    \tau_1(G_{227})
    =\tau_1\left(\Graph[0.8]{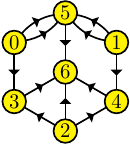}\right)
    =-\frac{5}{3\sqrt{\pi}} \pInt\left(\Graph[0.5]{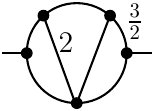}, 5\right)
    .
\end{equation*}
Applying the triangle rule to this graph directly does not help, since one of the 4 terms it produces still has the full set of edges.\footnote{The indices change only by integers, hence the edge with a half-integer index can never be contracted.}
Instead, we can apply the method of uniqueness \cite{Kazakov:MethodOfUniqueness,Gracey:MasslessByUniqueness} to replace trivalent vertices whose indices sum up to $n_1+n_2+n_3=5/2$ by triangles (denoted $Y\mapsto \Delta$) with indices $5/2-n_i$ (summing up to $5$). Therefore,
\begin{align*}
    \pInt\left(\Graph[0.5]{pf227p}, 5\right)
    &=\pInt\left(\Graph[0.5]{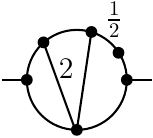}, 5\right)
    \underset{Y\mapsto\Delta}{=}\frac{\sqrt{\pi}}{4} \pInt\left(\Graph[0.5]{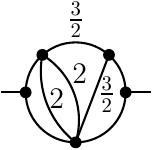}, 5\right)
    =\frac{\pi^2}{8} \pInt\left(\Graph[0.5]{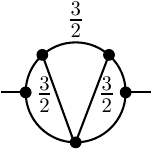}, 5\right)
    \\
    &=\frac{\sqrt{\pi}}{4} \pInt\left(\Graph[0.5]{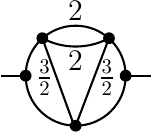}, 5\right)
    \underset{\Delta\mapsto Y}{=}\pInt\left(\Graph[0.5]{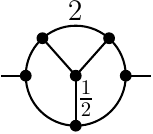}, 5\right)
    \underset{Y\mapsto\Delta}{=}\frac{\sqrt{\pi}}{4} \pInt\left(\Graph[0.5]{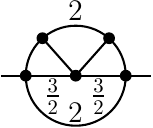}, 5\right)
    \\
    &=\pInt\left(\Graph[0.5]{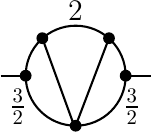}, 5\right)
    .
\end{align*}
In the last graph, all 3 edges whose indices get reduced by the triangle rule have integer indices, hence by applying the triangle rule twice, this integral can again be reduced to 2-loop propagators and thus computed as before. The final result is
\begin{equation*}
    \tau_1(G_{227})
    =10\cdot\left(
        -8-\frac{8}{9}\pi^2+\frac{1}{6}\pi^4
    \right)
    \approx -\numprint{5.381332}
    .
\end{equation*}

With the techniques illustrated above we also computed the canonical integrals of the graphs $G_{243}$ and $G_{246}$ (see \cref{tab:integrals} for their definition and results).

For the remaining integrals, we use more involved integration-by-parts relations than just the triangle rule. Assume that $G$ has 6 loops and a pair of parallel edges, then:
\begin{enumerate}
    \item Following \cref{sec:pint}, cut an edge and integrate out the parallel pair to write $\tau_1(G)$ as a linear combination of 4-loop propagators $\pInt(G',n,5)$.
    \item Introduce dimensional regularization: $\pInt(G',n,5)=\lim_{\dimeps=0} \pInt(G',n,5-2\dimeps)$.
    \item Use {\LiteRed} \cite{Lee:LiteRed1.4} to reduce the propagators $\pInt(G',n,5-2\dimeps)$ to the basis of 4-loop master integrals $M_i(5-2\dimeps)$ defined in \cite[Fig.~2]{BaikovChetyrkin:FourLoopPropagatorsAlgebraic}.
    \item Use {\LiteRed}'s dimension shift to replace the master integrals $M_i(5-2\dimeps)$ near 5 dimensions by the same integrals in $3-2\dimeps$ dimensions.
    \item Substitute the $\dimeps$-expansions of the master integrals $M_i(3-2\dimeps)$ near 3 dimensions computed in \cite{LeeMingulov:SummerTime}.
\end{enumerate}
To illustrate this process, consider again the graph $G_{100}$ from \cref{sec:pint}: Integrating out only one of the 3 bubbles, we have
\begin{equation*}
    \tau_1(G_{100})
    =10\cdot\frac{\pi}{6}\cdot\pInt\left(\Graph[0.6]{pf100p},5\right)
    =10\cdot\frac{2}{3\pi^2}\cdot\pInt\left(\Graph[0.6]{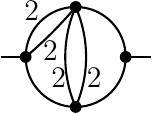},5\right).
\end{equation*}
Applying {\LiteRed} to this 4-loop propagator\footnote{The concrete function call is \texttt{IBPReduce[LoweringDRR[p4,0,0,0,0,0,2,2,2,1,0,1,0,2,1]]}.} produces the reduction
\begin{align*}
\pInt\left(\Graph[0.5]{M21b},D+2\right)
&=\frac{(D-4)(7D-18)}{4(D-1)(3D-10)} M_{13}(D)
-\frac{D-4}{4D-4} M_{21}(D)
\\ & \quad
+\frac{(2D-5)(5D-12)(9D^2-50D+72)}{4(D-3)(D-1)(3D-10)(3D-8)} M_{01}(D)
\end{align*}
in terms of 3 master integrals defined by
\begin{align*}
    M_{01}(D) &
    = \pInt\left(\Graph[0.5]{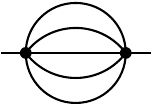}, D\right)
    = \frac{\Gamma(D/2-1)^5\Gamma(5-2D)}{\Gamma(5D/2-5)}
    \\
    M_{13}(D) &
    = \pInt\left(\Graph[0.5]{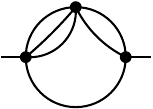}, D\right)
    = \frac{\Gamma(D/2-1)^6\Gamma(2-D/2)\Gamma(3-D)\Gamma(2D-5)\Gamma(6-2D)}{\Gamma(D-2)\Gamma(3D/2-3)\Gamma(5-3D/2)\Gamma(5D/2-6)},
    \\
    M_{21}(D) &
    = \pInt\left(\Graph[0.5]{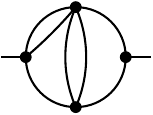}, D\right).
\end{align*}
The above expressions for $M_{01}$ and $M_{13}$ in terms of $\Gamma$-functions follow from series-parallel reductions. Expanding in $D=3-2\dimeps$, these calculations result in
\begin{equation*}
    \tau_1(G_{100})=10\cdot \left( -\frac{8}{3}-\frac{\pi^2}{3} + \frac{M_{21}(3)}{12\pi^2} \right).
\end{equation*}
Similarly, we can express all other canonical integrals (of graphs that have at least one pair of parallel edges) in terms of 4-loop master integrals in $3-2\dimeps$ dimensions. The latter $\dimeps$-expansions were computed in \cite{LeeMingulov:SummerTime}.\footnote{In that paper the $\dimeps$-expansions are computed numerically to 55 significant digits and then fitted to a basis of the expected transcendental numbers. Our calculation only uses very few of those coefficients, for example the expansion of $M_{36}$ up to the finite part, which was computed exactly in \cite{DamourJaranowski:4Lstat2}.}

\begin{remark}
In the above example, we can reverse the logic, and infer the value of $M_{21}(3)$ from our calculation of $\tau_1(G_{100})$. Our result \eqref{eq:int_100(b5pf)} amounts to
\begin{equation*}\begin{split}
    M_{21}(3-2\dimeps)
    &=
    \frac{8\pi^2}{3}\big[
        \pi^2\Catalan + 24 \Im \Li{4}(\iu)
    \big]
    +\asyO(\dimeps) \\
    &= M_{01}(3-2\dimeps)\cdot\left(
    -8\dimeps\big[
        \pi^2\Catalan + 24 \Im \Li{4}(\iu)
    \big]
    +\asyO(\dimeps^2)
    \right)
\end{split}\end{equation*}
and therefore identifies the numeric constant $M_{21}^{(1)}\approx \numprint{-262.199098}$ from \cite[eq.~(A.3)]{LeeMingulov:SummerTime} with $-8[\pi^2\Catalan + 24 \Im \Li{4}(\iu)]$. This identity was pointed out explicitly in \cite[eq.~(8)]{PikelnerGusyninKotikovTeber:4qQED}, and in fact observed already implicitly in \cite{Hager:phi6eps3}.
\end{remark}

\subsection{Gegenbauer polynomials}
The only remaining graphs not amenable to the above techniques are the simple graphs. The corresponding 5-loop propagators and their master integrals are not known in 5 dimensions, and without any parallel edges, we cannot simply reduce them to 4-loop propagators as we did for all other graphs.

Instead of momentum space, we now use the position space representation
\begin{equation}\label{eq:xint}
    \int_{\FS_{N}}\frac{\Omega_{N}}{\Symanzik_G^{\FIdim/2}} \prod_{i=1}^{N} x_i^{n_i-1}
    =
    \frac{1}{\Gamma(\FIdim/2)}
    \left(\prod_{v\neq v_0,v_1} \int_{\R^{\FIdim}} \frac{\td[\FIdim] z_v}{\pi^{\FIdim/2}}\right)
    \prod_{i=1}^{N} \frac{\Gamma(n_i^\star)}{\norm{z_{a(i)}-z_{b(i)}}^{2n_i^\star}}
\end{equation}
where we set $n_i^\star\defas \FIdim/2-n_i$. The integral on the right-hand side is over variables $z_v\in\R^\FIdim$ associated to each vertex, except for two arbitrary vertices $v_0$ and $v_1$ which are fixed at the origin $z_{v_0}=0\in\R^\FIdim$ and a unit vector $z_{v_1}=(1,0,\ldots,0)$. With $a(i)$ and $b(i)$ we denote the two endpoints (vertices) of edge $i$. We write $z_i=r_i\hat{z}_i$ in polar coordinates
\begin{equation*}
    r_i \defas \norm{z_i}
    \quad\text{and}\quad
    \hat{z}_i\defas z_i/r_i.
\end{equation*}

The method \cite{ChetyrkinKataevTkachov:Gegenbauer} is based on expanding all propagators into the series
\begin{equation}\label{eq:Gegen-expansion}\begin{aligned}
    \frac{1}{\norm{z_i-z_j}^{2\alpha}}
    &= \sum_{n=0}^{\infty}
    C^{\alpha}_n(\hat{z}_i\cdot\hat{z}_j)
    R^{\alpha}_n(r_i,r_j)
    \quad\text{with}\quad
    R^{\alpha}_n(s,t)=\begin{cases} \frac{s^n}{t^{n+2\alpha}} & \text{if $s<t$,} \\ \frac{t^n}{s^{n+2\alpha}} & \text{if $t<s$.} \\ \end{cases}
\end{aligned}\end{equation}
This form separates the dependence on the lengths of the vectors $z_i$ and $z_j$ from the dependence on the enclosed angle $\phi$ (the scalar product $\hat{z}_i\cdot\hat{z}_j=\cos\phi$). The Gegenbauer polynomials $C^{\alpha}_n$ are defined by the generating series
\begin{equation*}
    \sum_{n=0}^{\infty} C_n^{\alpha}(u) t^n = (1-2ut+t^2)^{-\alpha}.
\end{equation*}
The spherical angles in the measure $\td[\FIdim] z_i=r_i^{\FIdim-1}\td r_i\, \td[\FIdim-1]\hat{z}_i$ can be integrated out, thanks to the orthogonality relation
\begin{equation}\label{eq:gegen-ortho}
    \int_{\norm{z_i}=1} C_n^{\alpha}(\hat{z}_i\cdot\hat{z}_j) C_m^{\alpha}(\hat{z}_i \cdot \hat{z}_k)  \,\frac{\td[\FIdim-1]\hat{z}_i}{\pi^{\FIdim/2}}
    = \frac{2\alpha}{\Gamma(\FIdim/2)}
    \frac{C^{\alpha}_n(\hat{z}_j\cdot\hat{z}_k)}{\alpha+n}
    \,\delta_{n,m}
\end{equation}
and further identities, whereas the integrals over the lengths $r_i$ produce rational functions of the summation indices $n$ in \eqref{eq:Gegen-expansion}. The end result is a formula for the integral \eqref{eq:xint} as a multiple sum, which can be evaluated numerically to high precision.

\begin{figure}
    \centering
    $G_{241}=\Graph[0.8]{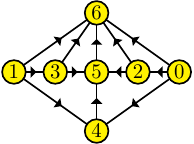}\qquad 
    G_{266}=\Graph[0.8]{gc266} \qquad
    G_{267}=\Graph[0.8]{gc267}$
    \caption{The simple graphs with non-zero integrand $\pff_G\wedge\omega_G^5$.}%
    \label{fig:simple-gegen}%
\end{figure}
Consider for example the graph $G_{266}$ from \cref{fig:simple-gegen} with edge list $456|346|356|6|5|6|$; the variables $x_2,x_{10},x_{12}$ thus correspond to the edges $\{0,5\},\{3,6\},\{5,6\}$, respectively. The numerator is $Q=10\big[x_{10}x_{12}(x_2-x_{11})-x_2x_6x_9\big]$ and we choose $v_0=6$ and $v_1=5$. For the first of the 3 numerator monomials, $x_2 x_{10} x_{12}$, the integral \eqref{eq:xint} then reads
\begin{multline*}
    \frac{\Gamma(3/2)^9\Gamma(1/2)^3}{\Gamma(5/2)}
    \int_{\R^5} \frac{\td[5]z_4}{\pi^{5/2}}
    \frac{1}{\norm{z_4-z_5}^3}
    \int_{\R^5} \frac{\td[5]z_0}{\pi^{5/2}}
        \frac{1}{r_0^3\norm{z_0-z_4}^3\norm{z_0-z_5}}
    \\
    \times
    \int_{\R^5} \frac{\td[5]z_1}{\pi^{5/2}} \int_{\R^5} \frac{\td[5]z_2}{\pi^{5/2}} \int_{\R^5} \frac{\td[5]z_3}{\pi^{5/2}}
    \frac{1}{r_1^3r_2^3r_3\norm{z_5-z_2}^3\norm{z_2-z_3}^3\norm{z_3-z_1}^3\norm{z_1-z_4}^3}
    .
\end{multline*}
The angular integrals over vertices $1,2,3$ are straightforward with \eqref{eq:gegen-ortho}:
\begin{multline}\label{eq:gegen-266-a}
    \int\frac{\td[4]\hat{z}_1}{\pi^{5/2}}\int\frac{\td[4]\hat{z}_2}{\pi^{5/2}}\int\frac{\td[4]\hat{z}_3}{\pi^{5/2}}
    \frac{1}{\norm{z_5-z_2}^3\norm{z_2-z_3}^3\norm{z_3-z_1}^3\norm{z_1-z_4}^3}
    \\
    =
    \frac{64}{\pi^{3/2}}
    \sum_{n=0}^{\infty}
    \frac{C^{3/2}_n(\hat{z}_4\cdot\hat{z}_5)}{(n+3/2)^3}
       R_n^{3/2}(1,r_2) R_n^{3/2}(r_2,r_3) R_n^{3/2}(r_3,r_1) R_n^{3/2}(r_1,r_4)
\end{multline}
where the prefactor is $(4/\sqrt{\pi})^3=(3/\Gamma(5/2))^3$.
To compute the angular integral over $\hat{z}_0$, we replace $C^{1/2}_n(\hat{z}_0\cdot\hat{z}_5)$ in the expansion \eqref{eq:Gegen-expansion} of $1/\norm{z_0-z_5}$ by the identity
\begin{equation*}
    C^{1/2}_n = \frac{C^{3/2}_n-C^{3/2}_{n-2}}{2n+1}
\end{equation*}
so that we can apply the orthogonality relation \eqref{eq:gegen-ortho} with the polynomials $C^{3/2}_n(\hat{z}_0\cdot\hat{z}_4)$ from the expansion of $1/\norm{z_0-z_4}^3$. We therefore obtain that
\begin{multline}\label{eq:gegen-266-b}
    \int\frac{\td[4]\hat{z}_0}{\pi^{5/2}}
    \frac{1}{\norm{z_0-z_4}^3\norm{z_0-z_5}}
    \\
    =\frac{4}{\sqrt{\pi}} \sum_{m=0}^{\infty} \frac{C^{3/2}_{m}(\hat{z}_4\cdot\hat{z}_5)}{m+3/2}
    R_m^{3/2}(r_0,r_4)
    \left[
        \frac{R_m^{1/2}(1,r_0)}{2m+1}-\frac{R_{m+2}^{1/2}(1,r_0)}{2m+5}
    \right]
    .
\end{multline}
The final angular integral over $\td[4]\hat{z}_4$ involves the product of three Gegenbauer polynomials: one from the expansion of $1/\norm{z_4-z_5}^3$, one from \eqref{eq:gegen-266-a}, and one from \eqref{eq:gegen-266-b}. This
\begin{equation*}
    \int C^{3/2}_n(\hat{z}_i\cdot \hat{z}_j)C^{3/2}_m(\hat{z}_i\cdot \hat{z}_j)C^{3/2}_l(\hat{z}_i\cdot \hat{z}_j) \frac{\td[4]\hat{z}_i}{\pi^{5/2}}
    = \frac{4}{\sqrt{\pi}} \Delta(n,m,l)
\end{equation*}
can be computed as follows \cite[Appendix~A]{ChetyrkinKataevTkachov:Gegenbauer}: If $k=(n+m+l)/2$ is not an integer, or if any of the triangle inequalities $n\leq m+l$, $m\leq n+l$, $l\leq n+m$ is violated, then we have $\Delta(n,m,l)=0$. Otherwise,
\begin{equation*}
    \Delta(n,m,l)
    =\frac{2\Gamma(k+3)}{\pi\Gamma(k+5/2)} \frac{\Gamma(k-n+3/2)}{\Gamma(k-n+1)}\frac{\Gamma(k-m+3/2)}{\Gamma(k-m+1)}\frac{\Gamma(k-l+3/2)}{\Gamma(k-l+1)}
    .
\end{equation*}
In conclusion, we obtain a triple sum
\begin{equation*}
    \int_{\FS_{12}}\frac{x_2 x_{10} x_{12}\Omega_{12}}{\Symanzik_{G_{266}}^{5/2}}
    =\frac{(2\pi)^3}{3} \sum_{n,m,l=0}^{\infty} \frac{\Delta(n,m,l)F(n,m,l)}{(n+3/2)^3(m+3/2)}
\end{equation*}
where $F(n,m,l)$ is a rational function defined by the radial integrals:
\begin{multline*}
    F(n,m,l) = \int_0^{\infty} \!\!r_4^4 \td r_4\; R_l^{3/2}(r_4,1)
        \int_0^{\infty}\!\! r_0 \td r_0 \;
        R_m^{3/2}(r_0,r_4)
        \bigg[
            \frac{R_m^{1/2}(1,r_0)}{2m+1}-\frac{R_{m+2}^{1/2}(1,r_0)}{2m+5}
        \bigg]
    \\
    \times
        \int_0^{\infty} \!\! r_1 \td r_1\; \int_0^{\infty} \!\!r_2\td r_2\;\int_0^{\infty} \!\!r_3^3\td r_3\;
         R_n^{3/2}(1,r_2) R_n^{3/2}(r_2,r_3) R_n^{3/2}(r_3,r_1) R_n^{3/2}(r_1,r_4)
    .
\end{multline*}
These integrals are easily computed in closed form by subdividing the integration domain into $6!$ regions according to the total order of $\{r_0,r_1,r_2,r_3,r_4,1\}$, since in each such region, all functions $R_k^{\alpha}(r_i,r_j)$ are given by a fixed monomial.

Treating the remaining terms analogously, we obtain a triple sum for $\tau_1(G_{266})$. With a {\Maple} program we compute the truncations $S(M)=\sum_{n,m\leq M} \sum_{l=\abs{n-m}}^{n+m} (\text{summand})$ of these sums up to $M=700$. These truncations exhibit a power law remainder term,
\begin{equation*}
    S(M) = \tau_1(G_{266}) + \frac{1}{M^3}\left( c + \asyO\left(\frac{1}{M}\right) \right)
\end{equation*}
and thus the precision improves only very slowly with growing $M$. With methods to accelerate convergence\footnote{We used \cite[eq.~(21)]{Broadhurst:IrreducibleEulerSums} with $C=3$ for $G_{266}$, and with $C=5$ for $G_{241}$.} we can nevertheless reliably identify 41 significant figures,
\begin{equation*}
    \tau_1(G_{266}) \approx \numprint{0.76007971277105870470418001694287511301749}.
\end{equation*}
This is ample precision to confirm the relation (found with PSLQ) that
\begin{multline}\label{eq:pfb5-266}
    \tau_1(G_{266}) = 10 \cdot  \bigg(
        -\frac{8}{3}-\frac{2}{9}\pi^2-\frac{4}{3}\pi^2\ln(2)-2\zeta(3)-\frac{13}{60}\pi^4
    \\
        +\frac{2}{3}\Big(\ln^4(2)+24\Li{4}(\tfrac{1}{2})\Big)
        +\frac{8}{9}\Big(\pi^2\Catalan+24\Im\Li{4}(\iu)\Big)
    \bigg).
\end{multline}
With the same method we obtain 41 digits for $\tau_1(G_{241})$ from another triple sum (choosing again to fix the vertices $v_0=6$ and $v_1=5$). This approximation fits the expression
\begin{equation}\label{eq:pfb5-241}
    \tau_1(G_{241}) = 10 \cdot\left(
    -\frac{8}{3}-2\pi^2+\frac{2}{3}\pi^2\ln(2)-7\zeta(3)
        +\frac{16}{45}\pi^4
        -\frac{2}{3}(\ln^4(2)+24\Li{4}(\tfrac{1}{2}))
    \right).
\end{equation}
These identifications of $\tau_1(G_{241})$ and $\tau_1(G_{266})$ are further confirmed by a Stokes relation, which relates their sum to previously computed integrals:
\begin{equation*}
    \tau_1(G_{241}+G_{266})
    =\tau_1(G_{238}-G_{236}-G_{265})
\end{equation*}
We checked that the values from \eqref{eq:pfb5-266} and \eqref{eq:pfb5-241} fulfil this identity. 

The integral of the 3rd simple graph with non-vanishing integrand, $G_{267}$, is determined by another Stokes relation, namely the boundary given in \eqref{eq:Stokes-267}.

\printbibliography

\end{document}